\documentclass{article}
\usepackage{graphicx} 
\usepackage{fullpage}

\usepackage{graphicx}
\usepackage{multirow}%
\usepackage{amsmath,amssymb,amsfonts,amsthm}
\usepackage{hyperref,cleveref}
\usepackage{stackengine}

\usepackage{xcolor}%
\hypersetup{
	colorlinks,
	linkcolor={red!50!black},
	citecolor={blue!50!black},
	urlcolor={blue!80!black}
}
\theoremstyle{plain}
\newtheorem{theorem}{Theorem}[section]
\newtheorem{lemma}[theorem]{Lemma}
\newtheorem{corollary}[theorem]{Corollary}
\newtheorem{proposition}[theorem]{Proposition}
\newtheorem{definition}[theorem]{Definition}
\newtheorem{assumption}[theorem]{Assumption}
\newtheorem{remark}[theorem]{Remark}

\DeclareMathOperator*{\du}{d\!}

\def\bu{\mathbf{u}}
\def\bv{\mathbf{v}}
\def\bw{\mathbf{w}}
\def\be{\mathbf{e}}
\def\bphi{\boldsymbol{\varphi}}
\def\bpsi{\boldsymbol{\psi}}
\def\bro{\boldsymbol{\varrho}}
\def\bV{\boldsymbol{V}\!}
\def\bH{\boldsymbol{H}}
\def\bW{\boldsymbol{W}}
\def\bB{\boldsymbol{B}}
\def\bL{\boldsymbol{L}}
\def\bA{\boldsymbol{A}_\sigma}
\def\blf{\mathbf{f}}
\def\bq{\mathbf{q}}
\def\bF{\mathbf{F}}

\DeclareMathOperator{\bn}{\boldsymbol{n}}

\DeclareMathOperator{\bta}{\boldsymbol{\eta}}
\DeclareMathOperator{\bzeta}{\boldsymbol{\zeta}}

\DeclareMathOperator{\bvsigma}{\boldsymbol{\varsigma}}

\DeclareMathOperator{\ws}{\mathrel{\ensurestackMath{\stackon[1pt]{\rightharpoonup}{\scriptstyle\ast}}}}

\newcommand{\dive}{\operatorname{div}}

\stackMath
\newcommand\sbullet[1][.5]{\mathbin{\vcenter{\hbox{\scalebox{#1}{$\bullet$}}}}}
\newcommand\frightarrow{\scalebox{.4}[.4]{$\sbullet[1.5]$}}
\newcommand\comeq{\mathrel{%
		\stackengine{-1.25pt}{\leq}{\frightarrow}{U}{r}{F}{T}{S}}}

\title{Analysis of Unregularized Optimal Control Problems Constrained by the Boussinesq
	System\thanks{N.J. was supported by the FWF grants P-31400-N32 and I4571-N.}}
\author{Nicolai Jork\thanks{Institute of Statistics and Mathematical Methods in Economics, Vienna University of Technology, Vienna, Austria.}  \and John Sebastian Simon\thanks{Radon Institute for Computational and Applied Mathematics, Austrian Academy of Sciences, Altenbergerstraße
		69, Linz, 4040, Upper Austria, Austria. email: \texttt{john.simon@ricam.oeaw.ac.at, jhsimon1729@gmail.com}}}
\date{September 2023}

\begin{document}

	\maketitle
	\begin{abstract}
		This paper investigates solution stability properties of unregularized tracking-type optimal control problems constrained by the Boussinesq system. In our model, the controls may appear linearly and distributed in both of the equations that constitute the Boussiniesq system and in the objective functional. We establish, not only the existence of weak solutions, but also unique existence of strong solutions in $L^p$ sense for the Boussinesq system as well as its corresponding linearized and adjoint systems. The optimal control problem is then analyzed by providing the existence of an optimal control, and by establishing first-order necessary and second order sufficient conditions. Then, using assumptions on the joint growth of the first and second variations of the objective functional, we prove the strong metric Hölder subregularity of the optimality mapping, which in turn allows the study of solution stability of the optimal control and states under various linear and nonlinear perturbations. Such perturbations may appear in the Boussinesq system and the objective functional. As an application, we provide a convergence rate for the optimal solutions of the Tikhonov regularized problem as the Tikhonov parameter tends to zero. Furthermore, the obtained stability of the optimal states provides, to the best knowledge of the authors, the first result on the stability of the second-order sufficient condition in affine PDE-constrained optimization under an assumption on the desired profile which is natural for tracking-type objective functionals.
	\end{abstract}

	\section{Introduction}
	
	Optimal control is concerned with finding a function that steers a given dynamical system according to some goal. 
	For instance, driving a fluid toward a given velocity field.
	In the past few decades, optimal control problems with dynamics given by solutions of partial differential equations have become of great interest among mathematicians and engineers alike. 
	This is of course because a lot of phenomena related to physics or economics are modeled with such kinds of equations.
	Notably challenging are optimal control problems with equations describing nonlinear dynamics, of which one is considered in this paper. 
	In particular, we shall consider the so-called Boussinesq system that describes a heat-conducting fluid and is governed by the viscous incompressible Navier--Stokes equations and the heat equation.
	The fluid is affected by the local variations of the temperature along the direction of the gravitational force, while the temperature is affected by the fluid by virtue of the drift of the fluid velocity.
	Such systems, and their variations, have been well-studied due to their ability to model fluid flow phenomena wherein thermal effects cannot be ignored, eg. on atmospheric models \cite{hsia2015} and ocean circulation \cite{greatbatch2001}. 
	Mathematically speaking, the Boussinesq system has also gained attention due to its nonlinear nature. 
	We refer to the recent works \cite{amorim2024,gong2023,lin2024,sun2024} dealing with specific variations of the system we intend to study.
	
	Optimal control problems subject to the Boussinesq system have also been the focus of several papers. For example, in \cite{BFA} the authors established the existence of an optimal solution for a tracking-type objective functional and a corresponding Pontryagin minimum principle for the controls appearing in the heat equation. Li \cite{li2005}, on another hand, studied a problem where constraints are imposed on the fluid velocity and temperature variable. The author further considered a generalized objective function that consists of a Lipschitz continuous functional with respect to the state variables and a functional that is at least quadratic. We also mention the following works that dealt with optimal control subject to the Boussinesq system \cite{baranovskii2021,chierici2022,lee2002,li2003}. 
	In the present paper, the controls may appear both on the momentum and the heat equation and are supported on small regions of the domain. This allows the model to apply to various situations, for instance, a control problem subject to a two-dimensional Navier--Stokes equation, which is reciprocally influenced by a separately controlled heat equation. We study distributed optimal control problems where the control appears at most affinely in the governing state equations and the objective functional. 
	Affinely controlled optimal control problems have attracted much attention recently as they appear naturally when considering tracking type objective functionals where the Tikhonov regularization term is absent. 
	This setting promotes a bang-bang structure of the optimal controls. 
	The study of optimal control problems subject to partial differential equations involving bang-bang solutions can only be traced to a handful of works. We mention \cite{DH2012} where the authors provided methods for approximating the solutions for a bang-bang optimal control problem for a system governed by the Poisson equation. In the aforementioned paper, the authors also laid out special cases wherein the so-called structural assumption for the adjoint variable will imply that the controls are indeed bang-bang. In \cite{casas2012}, E. Casas provided a second-order sufficient condition for an optimization problem constrained with an elliptic semilinear equation. For the analysis of optimal control problems with bang-bang solutions concerning more involved nonlinear systems, we mention \cite{casas2017}, wherein a data-tracking problem is analyzed for the Navier--Stokes equations. The authors established necessary and sufficient conditions for the mentioned optimization and provided error estimates for the finite element approximation of the solutions. 
	
	Aside from the necessary and sufficient conditions for the optimization problem, this paper delves into the stability of the optimal controls as well as the optimal states. The analysis in this direction will be aided by the so-called strong metric H{\"o}lder subregularity which has been studied in the case of elliptic \cite{casas2022,dominguez2022}, parabolic \cite{jork2023}, and the Navier--Stokes equations \cite{DJNS2023}. The underlying concept in such regularity is the stability, under well-chosen perturbations, of the optimality mapping derived from the first-order optimality conditions.
	
	\bigskip
	
	In what follows we present an overview of the core results of the present paper. For that let us first fix some terminology. 
	Let $\Omega\subset\mathbb{R}^2$ be a domain with $C^3$ boundary $\partial\Omega$, and let $0<T\in \mathbb R$ be a fixed final time. We denote $I:= (0,T)$, $Q:= I\times \Omega$ and $\Sigma:= I\times \partial\Omega$. And fix target functions $\bu_{d}:(0,T)\times\Omega\to \mathbb{R}^2$, $\bu_{T}:\Omega\to \mathbb{R}^2$, $\theta_d: (0,T)\times\Omega\to \mathbb{R}$ and $\theta_{T}:\Omega\to \mathbb{R}$. For us to be able to consider controls of bang-bang type we consider the functions $\overline{\bq}=(\overline{q}_{1},\overline{q}_{2}),\underline{\bq}=(\underline{q}_{1},\underline{q}_{2})\in L^\infty((0,T)\times\omega_q)^2$, and $\overline{\Theta},\underline{\Theta}\in L^\infty((0,T)\times\omega_h)$, which gives us the set of admissible controls that satisfy box constraints, i.e.
	\begin{equation}\label{box}
		\mathcal{U}:=\Big\{(\bq,\Theta)\in L^\infty((0,T)\times\omega_q)^2\times L^\infty((0,T)\times\omega_h) \Big \vert \ \underline{\bq} \comeq \bq \comeq \overline{\bq},\, \underline{\Theta}\leq \Theta \leq \overline{\Theta} \Big\},
	\end{equation}
	where the notation $\comeq$ stands for component-wise inequality, i.e. $(a_1,a_2)\comeq (b_1,b_2)$ if and only if $a_1 \le b_1$ and $a_2\le b_2$.
	The Boussinesq system consists of the following equations
	\begin{align}
		\partial_t \bu - \nu\Delta\bu + (\bu\cdot\nabla)\bu + \nabla p & = \be_2\theta + \blf + \bq \chi_{\omega_q} &&\text{in }Q,\label{B1}\\
		\dive\bu & = 0 &&\text{in }Q,\label{B2}\\
		\partial_t\theta -\kappa\Delta \theta + \bu\cdot\nabla\theta & = h + \Theta\chi_{\omega_h} &&\text{in }Q\label{B3},\\
		\bu = 0, \quad \theta & = 0 && \text{on }\Sigma,\label{B4}\\
		\bu(0,\cdot) = \bu_0, \quad \theta(0,\cdot) & = \theta_0 && \text{in }\Omega,\label{B5}
	\end{align}
	where $\omega_q,\omega_h\subset\Omega$ are two subdomains with sufficient regularity, $(\bu,p):(0,T)\times\Omega\to \mathbb{R}^2\times\mathbb{R}$ denotes the fluid velocity and pressure, $\theta:(0,T)\times\Omega\to \mathbb{R}$ accounts for the temperature, the term $\be_2\theta$, with $\be_2:= (0,1)$, takes into account the buoyancy force, $\bq$ and $\Theta$ are external force and heat source for the fluid and the temperature acting as controls, $\blf$ and $h$ respectively are given distributed external force and heat source, and the constant parameters $\nu>0$ and $\kappa>0$ are the fluid viscosity and heat conductivity, respectively. The tracking type optimal control problem, which is the main object of interest in our study, is then defined by
	\begin{align}\label{objectivfunk}
		\begin{aligned}
			\min_{(\bq,\Theta)\in \mathcal{U}} 
			J(\bq,\Theta) :=&\, \frac{\alpha_1}{2}\int_Q \vert \bu-\bu_{d}\vert^2 \du x\du t + \frac{\alpha_2}{2}\int_Q \vert \theta-\theta_{d}\vert^2 \du x\du t\\
			& + \frac{\beta_1}{2}\int_\Omega \vert \bu(T,\cdot)-\bu_{T}\vert^2\du x + \frac{\beta_2}{2}\int_\Omega \vert \theta(T,\cdot)-\theta_{T}\vert^2\du x ,
		\end{aligned}
	\end{align}
	subject to \eqref{B1}--\eqref{B5}, where $\alpha_1,\alpha_2,\beta_1,\beta_2\ge 0$ are weight parameters that satisfy $\alpha_1 + \alpha_2 + \beta_1 + \beta_2 > 0$.
	\bigskip
	
	The paper is concerned with the solution stability of the optimal controls and states of \eqref{objectivfunk}. To concertize this and to state the  main result in this paper, let us consider the Tikhonov regularized problem with additional perturbations in the desired data:
	\begin{align}\label{genpurt}
		\min_{(\bq,\Theta)\in \mathcal{U}}&\ \  \frac{\alpha_1}{2}\int_Q \vert \bu-(\bu_{d}+\widehat{\bu}_{d})\vert^2\ \du x\du t + \frac{\alpha_2}{2}\int_Q \vert \theta-(\theta_{d}+\widehat{\theta}_d)\vert^2\ \du x\du t + \frac{\beta_1}{2}\int_\Omega \vert \bu(T,\cdot)-\bu_{T}\vert^2  \du x \\
		&+ \frac{\beta_2}{2}\int_\Omega \vert \theta(T,\cdot)-\theta_{T}\vert^2\du x +\frac{\varepsilon_1}{2}\int_0^T\int_{\omega_q}\vert\bq \vert^2 \du x\du t+ \frac{\varepsilon_2}{2}\int_0^T\int_{\omega_h} \vert\Theta \vert^2 \du x\du t\nonumber
	\end{align}
	subject to \eqref{B1}-\eqref{B5}, but also with perturbations on the initial data, i.e., $\bu(0,\cdot) = \bu_0 + \widehat{\bu}_0$ and $\theta(0,\cdot) = \theta_0 + \widehat{\theta}_0$, and on the distributed external force and heat source denoted as $\widehat{\blf}$ and $\widehat{h}$. 
	Suppose that $(\bq^\star,\Theta^\star)\in \mathcal{U}$ is a local solution to the unperturbed optimal control problem. Then, under assumptions on the joint growth of the first and second variation of the objective functional to be specified later on, we obtain the following stability estimate for the regularized and perturbed problem in terms of the perturbations: there exist positive constants $c,\alpha $ and $\mu\in [1,2)$ such that
	\begin{align*}
		&\| \bq^\star - \bq\|_{\bL^1((0,T)\times\omega_q)} + \| \Theta^\star - \Theta\|_{\bL^1((0,T)\times\omega_h)}\\
		& \leq c\Big( \| \widehat{\bu}_{d}\| + \|\widehat{\theta}_{d}\| + \| \widehat{\blf}\| + \|\widehat{h}\| +  \| \widehat{\bu}_{0}\| + \|\widehat{\theta}_{0}\| +\varepsilon_1\| \bq^\star\|_{\bL^\infty(Q)}+\varepsilon_2\|\Theta^\star\|_{L^\infty}\Big)^{1/\mu}
	\end{align*}
	for all local minimizers $(\bq^\star,\Theta^\star)\in \mathcal{U}$ of the perturbed optimal control problem satisfying $$\| \bq^\star - \bq\|_{\bL^1((0,T)\times\omega_q)} + \| \Theta^\star - \Theta\|_{\bL^1((0,T)\times\omega_h)}<\alpha.$$  
	The stability under perturbations of the control problem is practical also in estimating the error of the numerical approximation obtained by FEM schemes see \cite{J2023} or the SQP-method which is generally known. But solution stability also has implications of analytical type. For instance, in a forthcoming paper, the authors show the applicability of solution stability for the investigation of the regularity of the local value function. 
	Furthermore, as we shall see in Section \ref{sec7}, the stability of the optimal states implies a result regarding the stability of the second-order optimality condition. That is, assume that the optimal solution $(\bq^\star,\Theta^\star)$ of \eqref{objectivfunk} satisfies
	\begin{align}\label{naturaltracl}
		\|\bu^\star - \bu_d\|_{\bL^s(Q)} + \|\theta^\star -\theta_d \|_{L^s(Q)} < \frac{\min\{\alpha_1,\alpha_2 \}}{\delta} \quad\text{for }s\ge 4,
	\end{align}
	for a positive constant $\delta$. 
	Then we prove that for sufficiently small perturbations the second-order derivative of the perturbed functional satisfies a growth with respect to the difference of the optimal states for the perturbed problem. One can easily see that condition \eqref{naturaltracl} is natural for a tracking type problem, in the sense that it is satisfied as soon as the set of admissible controls allows for a good approximation of the desired data.
	
	\bigskip
	
	The remainder of this article is as follows: Section \ref{sec2} is dedicated to introducing the functional analysis tools. The analysis for the governing equations will be presented in Section \ref{sec3}, additionally, we establish the existence of weak and strong solutions for the corresponding linearization of the nonlinear system \eqref{B1}--\eqref{B5} as well as its adjoint system. The control problem will be analyzed in Section \ref{sec4}. In Section \ref{sec5} we discuss the stability of solutions in detail, evoking notions and approaches from Variational Analysis. In particular, we discuss in subsection \ref{subsection4.17} the stability of the optimal control and states emanating from assumption on the growth of the derivatives of the objective functional with respect to the controls (see Assumption \ref{growth1}). While subsection \ref{sec6} deals with the state stability under the assumption on the growth of the derivatives of the objective functional with respect to the states (see Assumption \ref{growth2}). We conclude and give our final remarks in Section \ref{conclusion}.

	\section{Preliminaries}\label{sec2}
	Given a Banach space $X$ with the norm $\| \cdot\|_X$, we denote its dual as $X^*$ and we write $\langle x^*,x\rangle_X $ for the dual pairing of $x\in X$ and $x^*\in X^*$. 
	The space of Lebesque measurable and $s$-integrable, with $s\geq 1$, on a measurable set $D$ is denoted by $L^s(D)$. Following the usual notation, we write $\|\cdot\|_{L^s}$ for the norm in $L^s(D)$, whenever the domain $D$ is known, and we write $(\cdot,\cdot)_D$ for the $L^2$-inner product in the domain $D$. The Sobolev-Slobodeckij spaces, given $s\ge 1$ and $m\ge 0$, are denoted as $W^{m,s}(\Omega)$ and for the case $s=2$ we use the notation $H^m(\Omega):=W^{m,2}(\Omega)$. The norm in $W^{m,s}(\Omega)$ is denoted by $\|\cdot\|_{W^{m,s}}$. Due to the Dirichlet conditions imposed on the governing equations, we are also inclined to consider zero trace Sobolev spaces which we denote as $W^{m,s}_0(\Omega)$. 
	To emphasize spaces consisting of vector-valued functions, we write in bold letter, for instance, $\bL^s(\Omega)$, $\bW^{m,s}(\Omega)$.
	For simplicity, we use the notations $H:= L^2(\Omega)$ and $V := H_0^1(\Omega)$ which will be primarily used for the analysis of the temperature.
	
	To facilitate the analyses, we review some useful estimates.
	\begin{lemma}[Rellich-Konrachov]
		Let $\Omega$ be $C^1$, $0<m$, $0\le \bar{m}$ and $1\le s,s_1,s_2$. Suppose that either of the two cases holds: i. $m>0$ and $1\le s_2 \le 2s_1/(2-ms_2) $; or ii. $2 = ms_1$ and $1\le s_2 < +\infty$. Then the embedding $W^{m,s_1}(\Omega)\hookrightarrow L^{s_2}(\Omega)$ is compact. Furthermore, if $ms>2$ then we get the compact embedding $W^{m+\bar{m},s}(\Omega)\hookrightarrow C^{\bar{m}}(\overline{\Omega})$.
	\end{lemma}
	
	\begin{lemma}[Gagliardo-Nirenberg]
		Let $\Omega$ be a Lipschitz domain, and suppose that $m_1,m_2\in\mathbb{N}$ with $m_1<m_2$, $1\le s_1,s_2,s_3 $ and $\theta\in[0,1]$ satisfy
		\begin{align*}
			\frac{1}{s_1} = \frac{m_1}{2} + \theta\left( \frac{1}{s_2} - \frac{m_2}{2} \right) + \frac{1-\theta}{s_3},\quad \frac{m_1}{m_2}\le \theta\le 1.
		\end{align*}
		Then there exists $c>0$ such that
		\begin{align*}
			\|D^{m_1}u\|_{L^{s_1}} \le c \|u\|_{W^{m_2,s_2}}^\theta \|u\|_{L^{s_3}}^{1-\theta}
		\end{align*}
		for any $u\in L^{s_3}(\Omega)\cap W^{m_2,s_2}(\Omega)$.
	\end{lemma}
	
	To take into account the incompressibility of the fluid we consider the following solenoidal spaces
	\begin{align*}
		\bW^{m,s}_{0,\sigma} := \{ \bphi\in \bW_0^{m,s}(\Omega): \dive \bphi = 0 \text{ in }\Omega \}\text{ and }\bH_\sigma := \{ \bphi\in\bL^2(\Omega) : \dive\bphi = 0 \text{ in }L^2(\Omega), \bu\cdot\bn = 0 \text{ on }\partial\Omega \}.
	\end{align*}
	We also use the notation $\bV_\sigma := \bW^{1,2}_{0,\sigma}$. The pairs $(V,H)$ and $(\bV_\sigma,\bH_\sigma)$ are known to satisfy the Gelfand triple, i.e., the embeddings $\bV_\sigma\hookrightarrow \bH_\sigma \hookrightarrow \bV_\sigma^*$ and $V\hookrightarrow H\hookrightarrow V^* $ are dense, continuous and compact. In fact, the compactness follows from Rellich-Kondrachov embedding and Schauder theorems. 
	The orthogonal complement of $\bH_\sigma$ can be characterized as $
	\bH_\sigma^\bot=\{\bpsi\in   \bL^2(\Omega):\,\psi = \nabla  q \text{ for some } q\in H^1(\Omega) \}$.
	The representation  $ \bL^2(\Omega)=  \bH_\sigma\oplus  \bH_\sigma^\bot$ is called \textit{the Helmholtz--Leray decomposition}, from which we define the \textit{Leray projection operator} $P:  \bL^2(\Omega) \to  \bH_\sigma$ by $P\bpsi = \bpsi_1$, where $\bpsi_1$ is the unique element of $\bH_\sigma$ such that $\bpsi-\bpsi_1$ belongs to $ \bH_\sigma^\bot$. 
	{The Stokes operator $\bA:\bV_\sigma\cap \bH^2(\Omega)\subset \bH_\sigma\to \bH_\sigma$ is then defined by $\bA \bpsi = -P\Delta \bpsi$ and can be identified as a linear operator from $\bV_\sigma$ to $\bV_\sigma^*$ by $\langle \bA\bpsi_1,\bpsi_2 \rangle_V = (\nabla\bpsi_1,\nabla\bpsi_2)_\Omega$ for any $\bpsi_1,\bpsi_2\in \bV_\sigma$.}
	\bigskip
		
		Due to the time-dependence of our state equations, we consider also spaces of functions defined on the interval $I$ mapped to Banach spaces $X$. We begin with the space of functions from $I$ to $X$ that can be continuously extended on the closed interval $\overline{I}$, which we denote as $C(\overline{I};X)$. 
		The space $L^s(I; X)$ consists of functions $\psi:I\to X$ such that $t\to \|\psi(t)\|_{X}$ belongs to $L^s(0,T)$. 
		We see that the elements of $L^s(I;L^s(\Omega))$ can be identified as elements in $L^s(Q)$.
		The spaces that we are particularly interested in are the following 
		\begin{align*}
			& W^s(X):= \{v\in L^2(I;X): \partial_t v\in L^s(I;X^*) \},
			\quad W^{2,1}_{s} := \{\psi\in L^s(I;W^{2,s}(\Omega)\cap V) : \partial_t\psi \in L^s(Q) \}\\
			&\qquad\qquad\qquad \bW^{2,1}_{s,\sigma} := \{\bphi\in L^s(I;\bW^{2,s}(\Omega)\cap \bV_{\sigma}): \partial_t\bphi \in \bL^s(Q) \}.
		\end{align*}
		We see by Aubin-Lions-Simon embedding theorem that $W^2(V)\hookrightarrow C(\overline{I};H)$ and $W^2(\bV_\sigma)\hookrightarrow C(\overline{I}; \bH_\sigma)$. 
		From Amman embedding theorem \cite[Theorem 3]{amann2001}, we get the compact embeddings $W^{2,1}_s \hookrightarrow C(\overline{I};W^{2-2/s,s}_0(\Omega))$ and $\bW^{2,1}_{s,\sigma} \hookrightarrow C(\overline{I};\bW^{2-2/s,s}_{0,\sigma}(\Omega))$. 
		These embeddings compel us to assume that the initial data satisfy $\bu_0\in \bH_\sigma$ and $\theta_0\in H$ for the weak solution, while we assume $\bu_0\in \bW^{2-2/s,s}_{0,\sigma}(\Omega)$ and $\theta_0\in W^{2-2/s,s}_{0}(\Omega)$ for the strong solutions.
		Furthermore, due Rellich-Kondrachov embedding theorem we have the compact embeddings $W^{2,1}_s \hookrightarrow C(\overline{Q})$ and $\bW^{2,1}_{s,\sigma} \hookrightarrow C(\overline{Q})^2$ if $s>2$, and $W^{2,1}_s \hookrightarrow C(\overline{I};C^1(\overline{\Omega}))$ and $\bW^{2,1}_{s,\sigma} \hookrightarrow C(\overline{I};C^1(\overline{\Omega})^2)$ if $s>4$.

		\section{Analysis of the governing systems}\label{sec3}
		In this section, we provide existence results for the governing equations, including the linearization of \eqref{B1}--\eqref{B3} and a corresponding adjoint system.
		\subsection*{Solutions to the nonlinear system}
		We begin this section by providing the existence of solutions for the system
		\begin{align}
			\partial_t \bu - \nu\Delta\bu + (\bu\cdot\nabla)\bu + \nabla p & = \be_2\theta + \bF &&\text{in }Q,\label{Bg1}\\
			\dive\bu & = 0 &&\text{in }Q,\label{Bg2}\\
			\partial_t\theta -\kappa\Delta \theta + \bu\cdot\nabla\theta & = G &&\text{in }Q\label{Bg3},\\
			\bu = 0, \quad \theta & = 0 && \text{on }\Sigma,\label{Bg4}\\
			\bu(0,\cdot) = \bu_0, \quad \theta(0,\cdot) & = \theta_0 && \text{in }\Omega.\label{Bg5}
		\end{align}
		
		Let us define the notion of a weak solution. We say that $(\bu,\theta):Q \to \mathbb{R}^2\times\mathbb{R}$ is a weak solution of the system \eqref{Bg1}--\eqref{Bg5} if the following variational system is satisfied for a.e. $t\in (0,T)$
		\begin{align}
			\langle \partial_t\bu(t),\bphi\rangle_{\bV_\sigma^*} + \nu(\nabla\bu(t),\nabla\bphi) + ( (\bu(t)\cdot\nabla)\bu(t),\bphi ) & = (\be_2\theta(t),\bphi) + \langle \bF(t),\bphi \rangle_{\bV_\sigma^*} &&\forall \bphi\in\bV_\sigma,\label{var:Bg1}\\
			\langle \partial_t\theta(t),\psi\rangle_{V^*} + \kappa(\nabla\theta(t),\nabla\psi) + (\bu(t)\cdot\nabla\theta(t),\psi) & = \langle G(t),\psi\rangle_{V^*} &&\forall\psi\in V,\label{var:Bg2}
		\end{align}
		and that the initial conditions $\bu(0)=\bu_0$ and $\theta(0)=\theta_0$ are satisfied in $\bH_\sigma$ and $H$, respectively.
		
		The following theorem gives us the existence of weak solution in the sense defined above.
		\begin{theorem}\label{theorem:weakBouss}
			Let $\bF\in L^2(I;\bV_\sigma^*)$, $G\in L^2(I;V^*)$, $\bu_0\in \bH_\sigma$ and $\theta_0\in H$. Then there exists a unique element $(\bu,\theta)\in  W^2(\bV_\sigma)\times W^2(V)$ that solves \eqref{var:Bg1}--\eqref{var:Bg2} and satisfies  the estimates
			\begin{align}
				& \|\bu\|_{L^\infty(\bH_\sigma)} + \|\bu\|_{L^2(\bV_\sigma)} + \|\theta\|_{L^\infty(H)} + \|\theta\|_{L^2(V)} \le c(\|\bF\|_{L^2(\bV_\sigma^*)} + \|G\|_{L^2(V)} + \|\bu_0\|_{\bH_\sigma} + \|\theta_0\|_{H})\label{estimate:weak}\\
				& \begin{aligned}
					\|\partial_t\bu\|_{L^2(\bV_\sigma^*)} + \|\partial_t\theta\|_{L^2(V^*)} \le &\, c\left(1 + \|\bF\|_{L^2(\bV_\sigma^*)} + \|G\|_{L^2(V^*)} +\|\bu_0\|_{\bH_\sigma} + \|\theta_0\|_{H} \right)\\
					&\times\left(  \|\bF\|_{L^2(\bV_\sigma^*)} + \|G\|_{L^2(V^*)} +\|\bu_0\|_{\bH_\sigma} + \|\theta_0\|_{H} \right).
				\end{aligned}
			\end{align}
			for some constant $c:=c(\Omega,T)>0$.
		\end{theorem}
		\begin{proof}
			The proof of the theorem above shall utilize a Galerkin method wherein we discretize the system and show that the solutions of the discretized system converge to a solution to the original system. To do this one relies on the \textit{a priori} estimates to get precisely the uniform bounds on the spaces stated in the theorem. From there, compactness and density arguments will establish the wanted convergence. Since this follows a series of routine steps, we only show the very crucial part, which is the establishment of \textit{a priori} estimates.
			
			Indeed, by taking $\bphi=\bu(t)$ and $\psi = \theta(t)$, and applying H{\"o}lder and Young inequalities to \eqref{Bg1} and \eqref{Bg2} yields
			\begin{align}
				\begin{aligned}
					&\frac{1}{2}\frac{d}{dt}\Big( \|\bu(t)\|_{\bH_\sigma}^2 + \|\theta(t)\|_{H}^2  \Big) + \frac{\nu}{2}\|\bu(t)\|_{\bV_\sigma}^2 + \frac{\kappa}{2}\|\theta(t)\|_{V}^2 \\
					& \le \frac{c_{\be_2}}{2} \Big( \|\bu(t)\|_{\bH_\sigma}^2+\|\theta(t)\|_H^2  \Big)+ \frac{1}{\nu}\|\bF(t)\|_{\bV_\sigma^*}^2 + \frac{1}{\kappa}\|G(t)\|_{V^*},\label{tested}
				\end{aligned}
			\end{align}
			where $c_{\be_2}:= \max\left\{1,\|\be_2\|_{L^\infty}^2 \right\}>0$. Ignoring for the meantime the parts with the norms in $\bV_\sigma$ and $V$, Gronwall lemma implies that
			\begin{align}
				\|\bu\|_{L^\infty(\bH_\sigma)} + \|\theta\|_{L^\infty(H)} \le c \left(  \frac{1}{\sqrt{\nu}}\|\bF\|_{L^2(\bV_\sigma^*)} + \frac{1}{\sqrt{\kappa}}\|G\|_{L^2(V^*)} +\|\bu_0\|_{\bH_\sigma} + \|\theta_0\|_{H} \right), \label{Linftyestimate}
			\end{align}
			where $c\approx \exp{\left(c_{\be_2}T/2\right)}>$. Integrating \eqref{tested} over $(0,T)$ and using \eqref{Linftyestimate} thus give us 
			\begin{align}
				\nu\|\bu\|_{L^2(\bV_\sigma)} + \kappa\|\theta\|_{L^2(V)} \le c \left(  \frac{1}{\sqrt{\nu}}\|\bF\|_{L^2(\bV_\sigma^*)} + \frac{1}{\sqrt{\kappa}}\|G\|_{L^2(V^*)} +\|\bu_0\|_{\bH_\sigma} + \|\theta_0\|_{H} \right), \label{L2estimate}
			\end{align}
			where $c \approx \max\{1,\sqrt{c_{\be_2}T}\exp{(c_{\be_2}T/2)}\}$. We see that on the discretized level, \eqref{Linftyestimate} and \eqref{L2estimate}  provide the uniform boundedness from which we extract the solution.
			
			For the time derivatives, we see that by definition
			\begin{align*}
				& \|\partial_t\bu\|_{L^2(\bV_\sigma^*)} \le 
				\left(\nu\|\bu\|_{L^2(\bV_\sigma)} + \|\bu\|_{L^4}^2  + c_{P,\be_2}\|\theta\|_{L^2(H)} + \|\bF\|_{L^2(\bV_\sigma^*)}\right),\\
				& \|\partial_t\theta\|_{L^2(V^*)} \le
				\left( \kappa\|\theta\|_{L^2(V)} + \|\bu\|_{L^4}\|\theta\|_{L^4} + \|G\|_{L^2(V^*)} \right),
			\end{align*}
			where $c_{P,\be_2}:= c_P\|\be_2\|_{L^\infty}>0$ and $c_P>0$ is the Poincar{\'e} constant.
			
			For any variable $\varphi\in L^2(I;V)\cap L^\infty(I;L^2(\Omega))$ we see by virtue of the Ladyzhenskaya inequality that
			\begin{align*} 
				\|\varphi\|_{L^4} \le c_L\left( \int_0^T \|\varphi(t)\|_{L^2}^2\|\varphi(t)\|_{V}^2 \du t \right)^{1/4} \le c_L \|\varphi\|_{L^\infty(L^2)}^{1/2}\|\varphi\|_{L^2(V)}^{1/2},
			\end{align*}
			where $c_L>0$ is the Ladyzhenskaya constant. From this, we see that
			\begin{align*}
				\begin{aligned}
					\|\partial_t\bu\|_{L^2(\bV_\sigma^*)} + \|\partial_t\theta\|_{L^2(V^*)} \le &\, c\left(1 + \|\bF\|_{L^2(\bV_\sigma^*)} + \|G\|_{L^2(V^*)} +\|\bu_0\|_{\bH_\sigma} + \|\theta_0\|_{H} \right)\\
					&\times\left(  \|\bF\|_{L^2(\bV_\sigma^*)} + \|G\|_{L^2(V^*)} +\|\bu_0\|_{\bH_\sigma} + \|\theta_0\|_{H} \right),
				\end{aligned}
			\end{align*}
			where $c>0$ is some constant not dependent on $\bu$ and $\theta$. From the estimate above Aubin-Lions embedding gives us strong convergences on the discretized level. Furthermore, the estimate above implies that $(\bu,\theta)\in W^2(\bV_\sigma^*)\times W^2(V^*)$ through which the compact embedding $W^2(\bV_\sigma^*)\times W^2(V^*)\hookrightarrow C(\overline{I}; \bH_\sigma\times H)$ validates the spaces where we take the initial conditions.
			
		\end{proof}

		Let us now discuss the regularity of the solution, given that we assume better regularity for the data. In particular, we consider solutions whose external forces and heat source are in some $L^p$ space, for some $p\ge 2$. Although the analysis that follows is based on the work of C. Gerhardt \cite{gerhardt1978}, to the best of our knowledge, the following regularity result is a novelty.
		The existence of a unique strong solution to \eqref{B1}--\eqref{B5} is summarized in the theorem below.
		
		\begin{theorem}\label{theorem:strongLP}
			Let $s\ge2$, $\bF\in \bL^s(Q)$, $G\in L^s(Q)$, $\bu_0\in \bW^{2-1/s,s}_{0,\sigma}(\Omega)$ and $\theta_0\in W^{2-1/s,s}_{0}(\Omega)$. Then, the weak solution of \eqref{Bg1}--\eqref{Bg5} satisfies $(\bu,\theta)\in \bW^{2,1}_{s,\sigma}\times W^{2,1}_{s}$. Furthermore, there exists $p\in L^s(I;W^{1,s}(\Omega)/\mathbb{R})$ satisfying -- together with $(\bu,\theta)$ -- the energy estimate
			\begin{align}
				\begin{aligned}
					&\|\bu\|_{\bW^{2,1}_{s,\sigma}} + \|\theta\|_{W^{2,1}_{s}} + \|p\|_{L^s(I;W^{1,s}(\Omega)/\mathbb{R})}\\
					& \le c \left(\|\bF\|_{\bL^s} + \|G\|_{L^s} + \|\bu_0\|_{\bW^{2-1/s,s}_{0,\sigma}(\Omega)} +  \|\theta_0\|_{W^{2-1/s,s}_{0}(\Omega)}\right),
				\end{aligned}\label{estimate:strongLp}
			\end{align}
			for some constant $c>0$ independent of $(\bu,\theta)$ and $p$.
		\end{theorem}
		
		\begin{proof}
			Inspired by the work of C. Gerhardt \cite{gerhardt1978}, we prove the regularity in three steps: i. $s=2$; ii. $2<s< 4$; and iii. $4\le s <\infty$.
			
			\noindent\textit{Step I.} Let us take the $L^2$-inner product of \eqref{Bg1} and \eqref{Bg3} with $\varphi = \bA\bu$ and $\psi = -\Delta\theta$, respectively. This gives us
			\begin{align}
				\begin{aligned}
					& \frac{1}{2}\frac{d}{dt}\Big( \|\bu(t)\|_{\bV_\sigma}^2 + \|\theta(t) \|_{V}^2 \Big) + \frac{5\nu}{8}\|\bA\bu(t)\|_{\bH_\sigma}^2 + \frac{4\kappa}{6}\|\Delta\theta(t)\|_{H}^2\\ 
					& \le c\Big( \|\be_2\|_{\bL^\infty}\|\theta(t)\|_{H}^2 + 
					\|\bF(t)\|_{\bL^2}^2 + \|G(t)\|_{L^2}^2 + \|(\bu(t)\cdot\nabla)\bu(t)\|_{\bL^2}^2 + \|\bu(t)\cdot\nabla\theta(t) \|_{L^2}^2
					\Big).
				\end{aligned}\label{estimate:initweak}
			\end{align}
			From H{\"o}lder, Galiardo-Nirenberg and Young inequalities we thus get
			\begin{align*}
				&c\|(\bu(t)\cdot\nabla)\bu(t)\|_{\bL^2}^2 \le c\|\bu(t)\|_{\bH_\sigma}^2\|\bu(t)\|_{\bV_\sigma}^4 + \frac{\nu}{8}\|\bA\bu(t)\|_{\bH_\sigma}^2,\\
				&c\|\bu(t)\cdot\nabla\theta(t)\|_{\bL^2}^2 \le c\|\bu(t)\|_{\bH_\sigma}^2\|\bu(t)\|_{\bV_\sigma}^2\|\theta(t)\|_V^2 + \frac{\kappa}{6}\|A\theta(t)\|_{H}^2.
			\end{align*}
			These give us
			\begin{align*}
				\begin{aligned}
					& \Big( \|\bu(t)\|_{\bV_\sigma}^2 + \|\theta(t) \|_{V}^2 \Big) +{\nu}\|\bA\bu\|_{L^2(\bH_\sigma)}^2 + {\kappa}\|\Delta\theta\|_{L^2(H)}^2\\ 
					& \le c\Big(  
					\|\bF\|_{\bL^2}^2 + \|G\|_{L^2}^2 + \|\bu_0\|_{\bV_\sigma}^2 + \|\theta_0\|_V^2
					+ \int_0^T\|\bu(t)\|_{\bH_\sigma}^2 \|\bu(t)\|_{\bV_\sigma}^2 \left(\|\bu(t)\|_{\bV_\sigma}^2 + \|\theta(t)\|_V^2 \right)\du t
					\Big).
				\end{aligned}
			\end{align*}
			Since we have 
			\begin{align*}
				\int_0^T\|\bu(t)\|_{\bH_\sigma}^2 \|\bu(t)\|_{\bV_\sigma}^2\du t \le \|\bu\|_{L^\infty(\bH_\sigma)}^2 \|\bu\|_{L^2(\bV_\sigma)}^2 < \infty,
			\end{align*}
			we can employ Gr{\"o}nwall inequality which yields
			\begin{align*}
				\|\bu\|_{L^\infty(\bV_\sigma)}^2 + \|\theta \|_{L^\infty(V)}^2 \le c\Big(\|\bF\|_{\bL^2}^2 + \|G\|_{L^2}^2 + \|\bu_0\|_{\bV_\sigma}^2 + \|\theta_0\|_V^2 \Big).
			\end{align*}
			Using the estimate above, we further get 
			\begin{align*}
				{\nu}\|\bA\bu\|_{L^2(\bH_\sigma)}^2 + {\kappa}\|\Delta\theta\|_{L^2(H)}^2\le c\Big(\|\bF\|_{\bL^2}^2 + \|G\|_{L^2}^2 + \|\bu_0\|_{\bV_\sigma}^2 + \|\theta_0\|_V^2 \Big).
			\end{align*}
			One can get the estimate for the time derivative by taking $\varphi = \partial_t\bu$ and $\psi = \partial_t\theta$ and following the same steps. We get
			\begin{align}
				\begin{aligned}
					&\|\bu\|_{\bW^{2,1}_{2,\sigma}} + \|\theta\|_{W^{2,1}_{2}}  \le c \left(\|\bF\|_{\bL^2} + \|G\|_{L^2} + \|\bu_0\|_{\bW^{2,2-1/2}_{0,\sigma}(\Omega)} +  \|\theta_0\|_{W^{2,2-1/2}_{0}(\Omega)}\right).
				\end{aligned}\label{estimate:L2}
			\end{align}

			\noindent\textit{Step II.} In this case, it follows that $(\bu,\theta)\in \bW^{2,1}_{2,\sigma}\times W^{2,1}_2$. This implies, by Sobolev embedding, that $(\bu,\theta)\in \bW^{2,1}_{2,\sigma}\times W^{2,1}_2 \hookrightarrow \bL^{q_1}(Q)\times L^{q_2}(Q)$ for any $1\le q_1,q_2 <\infty$. With these observations, we establish our result by showing that $(\bu\cdot\nabla)\bu \in \bL^s(Q)$ and $\bu\cdot\nabla\theta \in L^s(Q)$.

			Let $\bar{s} \ge 1$ be such that $\frac{1}{4} + \frac{1}{\bar{s}} \le \frac{1}{s}$. By utilizing H{\"o}lder and Gagliardo-Nirenberg inequalities we get
			\begin{align*}
				\begin{aligned}
					\|(\bu\cdot\nabla)\bu\|_{\bL^s} & \le \|\bu\|_{\bL^{\bar{s}}}\|\nabla\bu\|_{L^4}  \le \|\bu\|_{L^2(\bV_\sigma)}^{1/2}\|\bu\|_{\bW^{2,1}_{2,\sigma}}^{3/2}.
				\end{aligned}
			\end{align*}
			Using similar arguments, we can estimate the transport term for the heat as follows:
			\begin{align*}
				\begin{aligned}
					\|\bu\cdot\nabla\theta\|_{\bL^s} & \le \|\bu\|_{\bL^{\bar{s}}}\|\nabla\theta\|_{L^4}  \le \|\bu\|_{\bW^{2,1}_{2,\sigma}}\|\theta\|_{L^2(V)}^{1/2}\|\theta\|_{W^{2,1}_{2}}^{1/2}.
				\end{aligned}
			\end{align*}
			As every term on the right-hand side of the two inequalities above is bounded, we thus find by \cite[Theorem 1.1]{solonnikov2001} and \cite[Theorem 4.9.1]{ladyzhenskaya1988} the existence of unique strong solutions. From \eqref{estimate:weak} and \eqref{estimate:L2}, one also get an estimate of the form \eqref{estimate:strongLp}.
			
			\noindent\textit{Step III.} 
			Due to the assumption on the powers of the data, we now have $(\bu,\theta)\in \bW^{2,1}_{s_1,\sigma}\times W^{2,1}_{r_1}$ for any $ s_1,r_1\in(2,4)$, and thus we have the compact embedding $\bW^{2,1}_{s_1,\sigma}\times W^{2,1}_{r_1}\hookrightarrow C(\overline{Q})^2\times C(\overline{Q})$. In this final step, we point out that the crucial part in the proof of the previous cases was estimating the nonlinear parts where we used a \textit{bootstrap} argument. For now, we only establish $(\bu\cdot\nabla)\bu \in \bL^s(Q)$ as the other nonlinearity can be handled similarly. From H{\"o}lder and Gagliardo-Nirenberg inequalities , and by choosing $s_1 = 3$ and letting $\theta = (5s - 6)/6s$ we have
			\begin{align*}
				\|(\bu\cdot\nabla)\bu\|_{\bL^s} \le \|\bu\|_{\bW^{2,1}_{3,\sigma}}\|\nabla\bu\|_{\bL^s} \le c \|\bu\|_{\bW^{2,1}_{3,\sigma}} \|\bu\|_{L^3(\bW^{2,3})}^\theta \|\bu\|_{\bL^3}^{1-\theta}.
			\end{align*}
			All terms on the right-hand side of the inequality above are bounded due to Step II.
			
			Again, \eqref{estimate:weak} and \eqref{estimate:L2} implies an estimate of the form \eqref{estimate:strongLp}.
		\end{proof}
		
		\begin{remark}
			For that case $s=r=2$, it is enough to assume that $\bu_0\in \bV_\sigma$ and $\theta_0\in V$.
		\end{remark}

		\subsection*{Solutions to the linearized system}
		Given $\bu_1,\bu_2\in  W^2(\bV_\sigma)$, $\theta\in W^2(V)$, $(\bF,G)\in \bV^*\times V$ we consider the system
		\begin{align}
			\partial_t \bv - \nu\Delta\bv + (\bu_1\cdot\nabla)\bv + (\bv\cdot\nabla)\bu_2 + \nabla q & = \be_2\vartheta + \bF &&\text{in }Q,\label{Blg1}\\
			\dive\bv & = 0 &&\text{in }Q,\label{Blg2}\\
			\partial_t\vartheta -\kappa\Delta \vartheta + \bu_1\cdot\nabla\vartheta + \bv\cdot\nabla\theta & = G &&\text{in }Q\label{Blg3},\\
			\bv = 0, \quad \vartheta & = 0 && \text{on }\Sigma,\label{Blg4}\\
			\bv(0,\cdot) = \bv_0, \quad \vartheta(0,\cdot) & = \vartheta_0 && \text{in }\Omega.\label{Blg5}
		\end{align}
		
		To establish the existence of weak solutions we point out that the only crucial part are the terms $((\bv\cdot\nabla)\bu_2,\bv)$ and $(\bv\cdot\nabla\theta,\vartheta)$ after multiplying $\bv$ and $\vartheta$ to \eqref{Blg1} and \eqref{Blg3}, respectively. To eliminate all the tedium, we now present how we treat such terms. Using H{\"o}lder, Gagliardo-Nirenberg and Young inequalities, we get for any $\varepsilon>0$
		\begin{align*}
			\begin{aligned}
				|((\bv\cdot\nabla) \bu_2,\bv) | & \le \|\bv\|_{\bL^4}^2\|\bu_2\|_{\bV_\sigma} \le c \|\bv\|_{\bH_\sigma}\|\bv\|_{\bV_\sigma}\|\bu_2\|_{\bV_\sigma}\\
				& \le c(\varepsilon)\|\bu_2\|_{\bV_\sigma}^2\|\bv\|_{\bH_\sigma}^2 + \varepsilon\|\bv\|_{\bV_\sigma}^2.
			\end{aligned}
		\end{align*}
		Here, $c(\varepsilon)>0$ results from the Young inequality. Similarly, for arbitrary $\varepsilon_1,\varepsilon_2>0$ we have
		\begin{align*}
			\begin{aligned}
				|(\bv\cdot\nabla\theta, \vartheta)| & \le \|\bv\|_{\bL^4}\|\theta\|_{V}\|\vartheta\|_{L^4} \le c \|\bv\|_{\bH_\sigma}^{1/2}\|\bv\|_{\bV_\sigma}^{1/2}\|\theta\|_{V}\|\vartheta\|_{H}^{1/2}\|\vartheta\|_{V}^{1/2}\\
				& \le c(\varepsilon_1,\varepsilon_2)\|\theta\|_{V}^2\left(\|\bv\|_{\bH_\sigma}^2 + \|\vartheta\|_{H}^2 \right) + \varepsilon_1\|\bv\|_{\bV_\sigma}^2 + \varepsilon_2\|\vartheta\|_{V}^2,
			\end{aligned}
		\end{align*}
		where again $c(\varepsilon)>0$ results from the Young inequality.
		
		Using the same arguments as in \Cref{theorem:weakBouss}, we have the following theorem.
		\begin{theorem}\label{theorem:weaklBouss}
			Let $\bF\in L^2(I;\bV_\sigma^*)$, $H\in L^2(I;V^*)$, $\bv_0\in \bH_\sigma$ and $\vartheta_0\in H$. Then the unique weak solution $(\bv,\vartheta)\in  W^2(\bV_\sigma))\times  W^2(V))$ of \eqref{Blg1}--\eqref{Blg5} exists and satisfies  the estimates
			\begin{align}
				& \|\bv\|_{L^\infty(\bH_\sigma)} + \|\bv\|_{L^2(\bV_\sigma)} + \|\vartheta\|_{L^\infty(H)} + \|\vartheta\|_{L^2(V)} \le c(\|\bF\|_{L^2(\bV_\sigma^*)} + \|G\|_{L^2(V)} + \|\bv_0\|_{\bH_\sigma} + \|\vartheta_0\|_{H})\label{estimate:lweak}\\
				& \begin{aligned}
					\|\partial_t\bv\|_{L^2(\bV_\sigma^*)} + \|\partial_t\vartheta\|_{L^2(V^*)} \le &\, c\left(1 + \|\bF\|_{L^2(\bV_\sigma^*)} + \|G\|_{L^2(V^*)} +\|\bv_0\|_{\bH_\sigma} + \|\vartheta_0\|_{H} \right)\\
					&\times\left(  \|\bF\|_{L^2(\bV_\sigma^*)} + \|G\|_{L^2(V^*)} +\|\bv_0\|_{\bH_\sigma} + \|\vartheta_0\|_{H} \right).
				\end{aligned}
			\end{align}
			for some constant $c:=c(\Omega,T,\|\bu_1\|^2_{\bV_\sigma},\|\bu_2\|^2_{\bV_\sigma},\|\theta\|_V^2)>0$.
		\end{theorem}
		
		For the analysis of strong solutions, we additionally assume $\bu_1,\bu_2\in \bW^{2,1}_{s,\sigma}$, $\theta\in W^{2,1}_s$, $(\bF,G)\in \bL^s(Q)\times L^s(Q)$ and $(\bv_0,\vartheta_0)\in \bW^{2-1/s,s}_{0,\sigma}(\Omega)\times W^{2-1/s,s}_{0}(\Omega)$. 
		Since the convection and transport terms are linear with respect to the unknown variable $(\bv,\vartheta)$, the analysis providing the strong solutions is more accessible. We skip the proof, which follows the same arguments we used for the nonlinear case. 
		
		\begin{theorem}\label{theorem:linearLP}
			Let $s\ge2$, $\bF\in \bL^s(Q)$, $G\in L^s(Q)$, $\bu_1,\bu_2\in \bW^{2,1}_{s,\sigma}$, $\theta\in W^{2,1}_s$, $\bv_0\in \bW^{2-1/s,s}_{0,\sigma}(\Omega)$ and $\vartheta_0\in W^{2-1/s,s}_{0}(\Omega)$. Then, the weak solution of \eqref{Blg1}--\eqref{Blg5} satisfies $(\bv,\vartheta)\in \bW^{2,1}_{s,\sigma}\times W^{2,1}_{s}$. Furthermore, there exists $q\in L^s(I;W^{1,s}(\Omega)/\mathbb{R})$ satisfying the energy estimate
			\begin{align}
				\begin{aligned}
					&\|\bv\|_{\bW^{2,1}_{s,\sigma}} + \|\vartheta\|_{W^{2,1}_{s}} + \|q\|_{L^s(I;W^{1,s}(\Omega)/\mathbb{R})}\\
					& \le c \left(\|\bF\|_{\bL^s} + \|G\|_{L^s} + \|\bv_0\|_{\bW^{2,2-1/p}_{0,\sigma}(\Omega)} +  \|\vartheta_0\|_{W^{2-1/s,s}_{0}(\Omega)}\right),
				\end{aligned}\label{estimate:stronglinLp}
			\end{align}
			for some constant $c:=c(\Omega,T,\|\bu\|_{\bW^{2,1}_{s,\sigma}},\|\theta\|_{W^{2,1}_s})>0$ independent of $(\bv,\vartheta)$ and $q$.
		\end{theorem}

		\subsection*{Solutions to the adjoint system}
		The analysis of the optimal control problem will be aided by the so-called adjoint method. For this reason we consider, given $\bu_1,\bu_2\in  W^2(\bV_\sigma)$,  $\theta\in W^2(V)$, $(\bF,G)\in \bV^*\times V$, the following adjoint system
		\begin{align}
			-\partial_t \bw - \nu\Delta\bw - (\bu_1\cdot\nabla)\bw + (\nabla\bu_2)^\top\bw + \Psi\nabla\theta + \nabla r & =  \bF &&\text{in }Q,\label{Badg1}\\
			\dive\bv & = 0 &&\text{in }Q,\label{Badg2}\\
			-\partial_t\Psi -\kappa\Delta \Psi - \bu_1\cdot\nabla \Psi & = \be_2\cdot\bw + G &&\text{in }Q\label{Badg3},\\
			\bw = 0, \quad \Psi & = 0 && \text{on }\Sigma,\label{Badg4}\\
			\bw(T,\cdot) = \bw_T, \quad \Psi(T,\cdot) & = \Psi_T && \text{in }\Omega.\label{Badg5}
		\end{align}
		
		The weak solution of the system above can be easily proven to exist using the same flow as we did for the nonlinear system. The only differences are: 
		
		\noindent$\bullet\ \ $ i. to get the $L^\infty(I;L^2(\Omega))$ estimates, we mention that we get an estimate of the form \eqref{estimate:initweak} but have a negative sign on the part with the time derivative. To handle this, we integrate over the interval $(t,T)$ for $t\in [0,T)$.
		
		\noindent$\bullet\ \ $ ii. the time derivates will satisfy $\partial_t\bw\in  L^{4/3}(I;\bV^*)$ and $\partial_t \Psi \in L^{4/3}(I;V^*)$. The reason is that we will encounter the terms $((\bphi\cdot\nabla)\bu_2,\bw)$ and $(\bphi\cdot\nabla\theta, \Psi)$ where $\bphi\in\bV_\sigma$ and $\psi\in V$ are such that $\max\{\|\bphi\|_{\bV_\sigma},\|\psi\|_V\} \le 1$. We note that in this part of the proof, we should have been able to prove that $\bw \in L^2(I;\bV_\sigma)\cap L^\infty(I;\bH_\sigma)$ and $\Psi\in L^2(I;V)\cap L^\infty(I;H)$. Thus, to handle such terms, we have the following computation
		\begin{align*}
			\begin{aligned}
				\int_0^T|((\bphi\cdot\nabla)\bu_2, \bw)|^{4/3}\du t & \le \int_0^T \|\bu_2\|_{V}^{4/3}\|\bw\|_{\bH_\sigma}^{2/3}\|\bw\|_{\bV_\sigma}^{2/3} \du t\\
				& \le c\int_0^T \Big( \|\bu_2\|_{\bV_\sigma}^2 + \|\bw\|_{\bH_\sigma}^{2}\|\bw\|_{\bV_\sigma}^{2} \Big) \du t\\
				& \le c\Big( \|\bu_2\|_{L^2(\bV_\sigma)} + \|\bw\|_{L^\infty(\bH_\sigma)}\|\bw\|_{L^2(\bV_\sigma)}\Big).
			\end{aligned}
		\end{align*}
		The same arguments can be applied to the other term. From these, we prove our claim.
		
		The summary of the existence of a weak solution to the adjoint system is written below.

		\begin{theorem}\label{theorem:weakadjBouss}
			Let $\bF\in L^2(I;\bV_\sigma^*)$, $H\in L^2(I;V^*)$, $\bw_T\in \bH_\sigma$ and $\vartheta_T\in H$. Then at least one weak solution $(\bw,\Psi)\in  W^{4/3}(\bV_\sigma)\times  W^{4/3}(V)$ of \eqref{Badg1}--\eqref{Badg5} exists and satisfies  the estimates
			\begin{align}
				& \|\bw\|_{L^\infty(\bH_\sigma)} + \|\bw\|_{L^2(\bV_\sigma)} + \|\Psi\|_{L^\infty(H)} + \|\Psi\|_{L^2(V)} \le c(\|\bF\|_{L^2(\bV_\sigma^*)} + \|G\|_{L^2(V)} + \|\bw_T\|_{\bH_\sigma} + \|\Psi_T\|_{H})\label{estimate:adjweak}\\
				& \begin{aligned}
					\|\partial_t\bw\|_{L^{4/3}(\bV_\sigma^*)} + \|\partial_t\Psi\|_{L^{4/3}(V^*)} \le &\, c\left(1 + \|\bF\|_{L^2(\bV_\sigma^*)} + \|G\|_{L^2(V^*)} +\|\bw_T\|_{\bH_\sigma} + \|\Psi_T\|_{H} \right)\\
					&\times\left(  \|\bF\|_{L^2(\bV_\sigma^*)} + \|G\|_{L^2(V^*)} +\|\bw_T\|_{\bH_\sigma} + \|\Psi_T\|_{H} \right).
				\end{aligned}
			\end{align}
			for some constant $c:=c(\Omega,T,\|\bu\|^2_{\bV_\sigma},\|\theta\|_V^2)>0$.
		\end{theorem}
		

		Even though we find some differences between the weak solutions of the adjoint system and the other two systems we considered --- such as uniqueness and the regularity of the time derivative --- additional regularity on the given variables for the adjoint system will lead to the same nice regularity we previously got.
		\begin{theorem}\label{theorem:adjointLP}
			Let $s\ge2$, $\bF\in \bL^s(Q)$, $G\in L^s(Q)$, $\bu_1,\bu_2\in \bW^{2,1}_{s,\sigma}$, $\theta\in W^{2,1}_s$, $\bw_T\in \bW^{2-1/s,s}_{0,\sigma}(\Omega)$ and $\Psi_T\in W^{2-1/s,s}_{0}(\Omega)$. Then, the weak solution of \eqref{Badg1}--\eqref{Badg5} is unique and satisfies $(\bw,\Psi)\in \bW^{2,1}_{s,\sigma}\times W^{2,1}_{s}$. Furthermore, there exists $r\in L^s(I;W^{1,s}(\Omega)/\mathbb{R})$ such that the solution $(\bw,\Psi)$ and $r$ satisfy the energy estimate
			\begin{align}
				\begin{aligned}
					&\|\bw\|_{\bW^{2,1}_{s,\sigma}} + \|\Psi\|_{W^{2,1}_{s}} + \|r\|_{L^s(I;W^{1,s}(\Omega)/\mathbb{R})}\\
					& \le c \left(\|\bF\|_{\bL^s} + \|G\|_{L^r} + \|\bw_T\|_{\bW^{2-1/s,s}_{0,\sigma}(\Omega)} +  \|\Psi_T\|_{W^{2-1/s,s}_{0}(\Omega)}\right),
				\end{aligned}\label{estimate:strongadjLp}
			\end{align}
			for some constant $c:=c(\Omega,T,\|\bu\|_{\bW^{2,1}_{s,\sigma}},\|\theta\|_{W^{2,1}_s})>0$ independent of $(\bv,\vartheta)$ and $q$.
		\end{theorem}

		To end the section we give a stability result for the linear and the adjoint systems to facilitate some of the analyses we shall delve into later.
		\begin{lemma}\label{lemma:LSL1}
			Let $s > 2$, $\bF\in \bL^s(Q)$, $G\in L^s(Q)$, $(\bu,\theta)\in \bW^{2,1}_{s,\sigma}\times W^{2,1}_s$, $\bv_0 = 0$ and $\vartheta_0 = 0$. Then for any $\bar{s} \in [1,2)$ there exists $c>0$ such that
			\begin{align}
				\|\bv\|_{\bL^{\bar{s}}} + \|\vartheta\|_{L^{\bar{s}}} \le c( \|\bF\|_{\bL^1} + \|G\|_{L^1} ).\label{estimate:estLsL1}
			\end{align}
		\end{lemma}
		
		\begin{proof}
			By virtue of \Cref{theorem:linearLP} and the embedding $\bW^{2,1}_{s,\sigma}\times W^{2,1}_{s}\hookrightarrow C(\overline{Q})^2\times C(\overline{Q})$, we get $(\bv,\vartheta)\in C(\overline{Q})^2\times C(\overline{Q})$. This implies that $(|\bv|^{\bar{s}-2}\bv, |\vartheta|^{\bar{r}-2}\vartheta ) \in \bL^{\bar{s}'}(Q)\times L^{\bar{r}'}(Q)$, where $\bar{s}' > 2$ is the H{\"o}lder conjugates of $\bar{s}$.
			
			By letting $(\bw,\Psi)\in \bW^{2,1}_{\bar{s}',\sigma}\times W^{2,1}_{\bar{s}'}$ to be the adjoint variable solving \eqref{Badg1}--\eqref{Badg5} with $\bF = |\bv|^{\bar{s}-2}\bv$, $G = |\vartheta|^{\bar{s}-2}\vartheta$, $\bw_T = 0$ and $\Psi_T = 0$ we get
			\begin{align}\label{e1}
				\|\bw\|_{\bW^{2,1}_{\bar{s}',\sigma}} + \|\Psi\|_{W^{2,1}_{\bar{s}'}} \le c ( \|\bv\|_{\bL^{\bar{s}}}^{\bar{s}-1} + \|\vartheta\|_{L^{\bar{s}}}^{\bar{s}-1} ).
			\end{align}
			
			From the adjoint and linear systems, by using H{\"o}lder inequality, and \eqref{e1} we get 
			\begin{align*}
				\begin{aligned}
					& \|\bv\|_{\bL^{\bar{s}}}^{\bar{s}} + \|\vartheta\|_{L^{\bar{s}}}^{\bar{s}}  = \int_Q |\bv|^{\bar{s}-2}\bv\cdot\bv + |\vartheta|^{\bar{s}-2}\vartheta\cdot \vartheta \du x\du t\\
					&  = \int_Q \left(-\partial_t \bw - \nu\Delta\bw - (\bu\cdot\nabla)\bw + (\nabla\bu)^\top\bw + \Psi\nabla\theta + \nabla r  \right)\cdot\bv \du x\du t\\
					& \ \ \ + \int_Q \left( -\partial_t\Psi -\kappa\Delta \Psi + \bu\cdot\nabla \Psi - \be_2\cdot\bw \right)\vartheta \du x\du t\\
					& = \int_Q \bw\cdot \left(\partial_t \bv - \nu\Delta\bv + (\bu\cdot\nabla)\bv + (\bv\cdot\nabla)\bu + \nabla q - \be_2\vartheta\right) \du x\du t\\
					& \ \ \ + \int_Q \Psi \left( \partial_t\vartheta -\kappa\Delta \vartheta + \bu\cdot\nabla\vartheta + \bv\cdot\nabla\theta \right) \du x\du t\\
					& = \int_Q \bw\cdot\bF + \Psi G \du x\du t \le c\max\left\{\|\bv\|_{\bL^{\bar{s}}}^{\bar{s}-1},\|\vartheta\|_{L^{\bar{s}}}^{\bar{s}-1}\right\}( \|\bF\|_{\bL^1} + \|G\|_{L^1} ).
				\end{aligned}
			\end{align*}
			
			Now, if $\|\bv\|_{\bL^{\bar{s}}} \le \|\vartheta\|_{L^{\bar{s}}}$ we have
			\begin{align*}
				\|\vartheta\|_{L^{\bar{s}}}^{\bar{s}} \le \|\vartheta\|_{L^{\bar{s}}}^{\bar{s}-1}c( \|\bF\|_{\bL^1} + \|G\|_{L^1} ).
			\end{align*}
			On the other hand if $\|\vartheta\|_{L^{\bar{s}}} \le \|\bv\|_{\bL^{\bar{s}}}$ we get
			\begin{align*}
				\|\bv\|_{\bL^{\bar{s}}}^{\bar{s}} \le \|\bv\|_{\bL^{\bar{s}}}^{\bar{s}-1}c( \|\bF\|_{\bL^1} + \|G\|_{L^1} ).
			\end{align*}
			In both cases, we get \eqref{estimate:estLsL1}.
			
		\end{proof}
		
		\section{Analysis of the optimal control problem}\label{sec4}
		Here, we prove the existence of at least one global optimal solution and introduce the optimality conditions.
		Before we begin, we recall
		\begin{align*}
			\mathcal{U}=\Big\{\bro:= (\bq,\Theta)\in \bL^\infty(I\times\omega_q)\times L^\infty(I\times\omega_h) \Big \vert \ \underline{\bq} \comeq \bq \comeq \overline{\bq},\, \underline{\Theta}\leq \Theta \leq \overline{\Theta} \Big\}.
		\end{align*}
		As a convention, whenever we use scripts on an element of $\mathcal{U}$, we also assume that such notation is carried over its arguments, e.g., $\bro^{\star} = (\bq^{\star},\Theta^{\star})$. For any $s,r\in [1,+\infty]$ we also use the notation
		\begin{align*}
			\|\bro\|_{\bL^s\times L^r} := \|\bq\|_{\bL^s((0,T)\times\omega_q)} + \|\Theta\|_{L^r((0,T)\times\omega_h)} \qquad\forall \bro\in \mathcal{U}.
		\end{align*}
		
		By the definition of $\mathcal{U}$, we see that there exists $M_{\mathcal{U}}>0$ such that
		\begin{align}\label{boundcon}
			\|\bro\|_{\bL^\infty\times L^\infty} \le M_{\mathcal{U}}\qquad \forall\bro\in\mathcal{U}.
		\end{align}
		For this reason, we can observe that even though the energy estimates for the solutions of the nonlinear, linear, and adjoint systems may be dependent on the norms of the controls with more than one power, we can instead assume linear dependence on the said norms.
		
		Let us reiterate our optimal control problem:
		\begin{align}\label{optcon}
			\min_{\bro\in\mathcal{U}} J(\bro) \text{ subject to }\eqref{B1}-\eqref{B2}.\tag{P}
		\end{align}

		For the analysis of the optimal control problem in this section, we give the following assumptions.
		\begin{assumption}\label{standingassumps}
			The following statements hold:
			\begin{itemize}
				\item[i.] The set $\Omega$ is an open connected  bounded subset of $\mathbb R^2$ with  boundary $\partial\Omega$ of class $C^3$;
				\item[ii.] the initial data $\bu_0\in \bW^{2-2/s,s}_{0,\sigma}(\Omega)^2$ and $\theta_0\in W^{2-2/s,s}_{0}(\Omega)^2$ for arbitrary $s>2$ are fixed;
				\item[iii.] the external force $\blf\in \bL^\infty(Q)$ and heat source $h\in L^\infty(Q)$ are given.
			\end{itemize}
		\end{assumption}
		
		With the assumptions laid out, we present a \textit{continuity} of the control-to-state operator, which we shall introduce later.
		\begin{lemma}
			Suppose that $s>2$ and that \Cref{standingassumps} hold. Let $\{\bro_n\}_n \subset \mathcal{U}$ be a sequence converging weakly to $\bro\in\mathcal{U}$ in $\bL^s(I\times\omega_q)\times L^s(I\times\omega_h)$. Then $(\bu_n,\theta_n)\to (\bu_{\bro},\theta_{\bro})$ in $C(\overline{Q})^2\times C(\overline{Q})$, where $(\bu_n,\theta_n)$ and $(\bu_{\bro},\theta_{\bro})$ are the strong solutions of \eqref{B1}--\eqref{B5} with $\bro_n$ and $\bro$ as controls, respectively.
		\end{lemma}
		\begin{proof}
			This is a direct consequence of \Cref{theorem:strongLP} and the compact embedding $\bW^{2,1}_{s,\sigma}\times W^{2,1}_s \hookrightarrow C(\overline{Q})^2\times C(\overline{Q})$.
		\end{proof}
		
		The lemma above can be used to establish the existence of an optimal control $\bro\in\mathcal{U}$. Since this follows a routine procedure we skip the proof, but one can refer to the arguments used in \cite[Theorem 2.1]{dwachsmuth2006}.
		\begin{theorem}
			The optimal control problem \eqref{optcon} has at least one solution.
		\end{theorem}
		
		The optimal control problem is nonlinear and nonconvex; we may have local minima besides at least one global minimum. We recall the two notions of local minimizers from \cite[Definition 1.6]{BBS} which we adapted to the problem in this paper.
		
		\begin{definition}
			Let $\bro^\star\in \mathcal{U}$ and $(\bu^\star,\theta^\star)\in \bW^{2,1}_{s,\sigma}\times W^{2,1}_s$ be the strong solution of \eqref{B1}--\eqref{B5} with $\bro^\star$ as the control.
			\begin{itemize}
				\item 
				We say that $\bro^{\star} \in \mathcal{U}$ is a weak local minimizer of \eqref{optcon}, if there exists a positive constant $\alpha$ such that
				\begin{equation}\label{mini}
					J(\bro )-J(\bro^{\star} )\geq 0,
				\end{equation}
				for all $\bro \in \mathcal{U}$ with $\|\bq- \bq^\star \|_{L^1(Q)}+\|\Theta- \Theta^\star \|_{L^1(Q)}<\alpha$.
				\item
				We say that $\bro^\star\in \mathcal{U}$ is a strong local minimizer of \eqref{optcon}, if there exists a positive  constant $\alpha$ such that \eqref{mini} holds for all
				$\bro\in \mathcal{U}$ with $\|\bu -  \bu^\star \|_{L^\infty(Q)}+\|\theta - \theta^\star \|_{L^\infty(Q)}<\alpha$, where $(\bu,\theta)$ is the strong solution of \eqref{B1}--\eqref{B5} with $\bro$ as control.
			\end{itemize}
		\end{definition}

		To facilitate the analysis of the optimality conditions we introduce the map 
		\begin{align*}
			\mathcal{F}:&\ \bW^{2,1}_{s,\sigma}\times W^{2,1}_s \times \bW^{2-1/s,s}_{0,\sigma}(\Omega)\times W^{2-1/s,s}_0(\Omega) \times L^s(I;\bL^s(\omega_q)\times L^s(\omega_h))\\
			& \to \bL^s_\sigma(Q)\times L^s(Q) \times \bW^{2-1/s,s}_{0,\sigma}(\Omega)\times W^{2-1/s,s}_0(\Omega)
		\end{align*} 
		defined as 
		\begin{align}\label{themap}
			&\mathcal{F}(\bu,\theta,\bu_0,\theta_0,\bq,\Theta) :=  \left[ 
			\begin{aligned}
				\partial_t\bu + \nu \bA\bu + (\bu\cdot\nabla)\bu & - \be_2\theta - \blf  - \bq\chi_{\omega_q}\\
				\partial_t\theta + \kappa A\theta + \bu\cdot\nabla\theta & - h - \Theta\chi_{\omega_h}\\
				\bu(0) & - \bu_0\\
				\theta(0) & - \theta_0
			\end{aligned}
			\right].
		\end{align}
		
		We define the data-to-state operator 
		\begin{align*}
			\widetilde{\mathcal{S}}: &\ \bW^{2-1/s,s}_{0,\sigma}(\Omega)\times W^{2-1/s,s}_0(\Omega) \times \bW^{2-1/s,s}_{0,\sigma}(\Omega)\times W^{2-1/s,s}_0(\Omega) \times L^s(I;\bL^s(\omega_q)\times L^s(\omega_h))\\
			&\to \bW^{2,1}_{s,\sigma}\times W^{2,1}_s 
		\end{align*}
		as 
		\begin{align}\label{defn:con-to-}
			\widetilde{\mathcal{S}}(\blf,h,\bu_0,\theta_0,\bq,\Theta) = (S_1(\blf,h,\bu_0,\theta_0,\bq,\Theta), S_2(\blf,h,\bu_0,\theta_0,\bq,\Theta)) = (\bu,\theta)
		\end{align}
		if and only if $\mathcal{F}(\bu,\theta,\bu_0,\theta_0,\bq,\Theta) = 0$. The well-definedness of the maps introduced is backed up by \Cref{theorem:strongLP}. As we assumed the initial data and the distributed external force and heat source to be fixed, we also introduce the control-to-state operator ${\mathcal{S}}: L^s(I;\bL^s(\omega_q)\times L^s(\omega_h)) \to \bW^{2,1}_{s,\sigma}\times W^{2,1}_s$ defined as ${\mathcal{S}}(\bq,\Theta) = \widetilde{\mathcal{S}}(\blf,h,\bu_0,\theta_0,\bq,\Theta)$.

		\begin{proposition}
			The control-to-state operator is of class $C^\infty$. The first-order Fr\'{e}chet derivative at $\bro\in L^s(I;\bL^s(\omega_q)\times L^s(\omega_h))$ in direction $\delta\bro\in L^s(I;\bL^s(\omega_q)\times L^s(\omega_h))$, denoted by 
			\[
			\mathcal{S}'(\bro)(\delta\bro)=:\left(\bv, \vartheta\right),
			\]
			is given as the solution to the linearized Boussinesq system \eqref{Blg1}--\eqref{Blg5} with $\bu_1 = \bu_2 = S_1(\bro)$ and $\theta = S_2(\bro)$, $\bF=\delta\bq\chi_{\omega_q}$, $G=\delta\Theta\chi_{\omega_h}$, $\bv_0=0$ and $\vartheta_0=0$.
			The second-order Fr\'{e}chet derivative at $\bro\in L^s(I;\bL^s(\omega_q)\times L^s(\omega_h))$ in direction $(\delta\bro_1,\delta \bro_2)\in L^s(I;\bL^s(\omega_q)\times L^s(\omega_h))^2$, denoted by 
			\begin{equation}
				\mathcal{S}''(\bro)[\delta\bro_1,\delta\bro_2]:=\big(\widetilde{\bv},\widetilde{\vartheta}\big),
			\end{equation}
			is the solution to the linearized Boussinesq system \eqref{Blg1}--\eqref{Blg5} with $\bu_1 = \bu_2 = S_1(\bro)$ and $\theta = S_2(\bro)$, $\bF=- [(\bv_1\cdot \nabla) \bv_2+(\bv_2\cdot \nabla) \bv_1]$, $G=-\left[\bv_1\cdot \nabla \vartheta_2 + \bv_2\cdot \nabla \vartheta_1\right]$, $\bv(0,\cdot)=0$ and $\vartheta(0,\cdot)=0$, where $(\bv_1,\vartheta_1) = \mathcal{S}'(\bro)(\delta\bro_1)$ and $(\bv_2,\vartheta_2) = \mathcal{S}'(\bro)(\delta\bro_2)$.
		\end{proposition}
		\begin{proof}
			The differentiability follows from the fact the $\mathcal{F}$ is at most quadratic and from the Implicit Function Theorem, see, e.g., \cite{pata2019}. We move on to deriving the system that the derivatives solve.

			We introduce the operator $L:\bW^{2,1}_{s,\sigma}\times W^{2,1}_s \to \mathcal{L}(\bW^{2,1}_{s,\sigma}\times W^{2,1}_s, \bL^s(Q)\times L^s(Q)\times \bW^{2-1/s,s}_{0,\sigma}(\Omega)\times W^{2-1/s,s}_0(\Omega))$ as 
			\begin{align}
				L(\bu,\theta)(\bv,\vartheta) =  \left[ 
				\begin{aligned}
					\partial_t\bv + \nu \bA\bv + (\bu\cdot\nabla)\bv &+ (\bv\cdot\nabla)\bu  - \be_2\vartheta\\
					\partial_t\vartheta + \kappa A\vartheta + \bu\cdot\nabla\vartheta &+ \bv\cdot\nabla\theta \\
					\bv(0) & \\
					\vartheta(0) &
				\end{aligned}
				\right].
			\end{align}
			We note that by \eqref{theorem:linearLP} $L(\bu,\theta)\in \mathcal{L}(\bW^{2,1}_{s,\sigma}\times W^{2,1}_s, \bL^s(Q)\times L^s(Q)\times \bW^{2-1/s,s}_{0,\sigma}(\Omega)\times W^{2-1/s,s}_0(\Omega))$ is an isomorphism. We use the chain rule over $\mathcal{F}(\mathcal{S}(\bro),\bro) = 0$, where we omitted the initial data for the meantime. Hence, we have $L(\mathcal{S}(\bro))\mathcal{S}'(\bro)\delta\bro + \frac{\partial\mathcal{F}}{\partial \bro} \delta\bro = 0$. Since $\frac{\partial\mathcal{F}}{\partial \bro}\delta\bro = -(\delta\bq\chi_{\omega_q},\delta\Theta\chi_{\omega_h},0,0)^\top$, we get that $\mathcal{S}'(\bro)\delta\bro = L(\mathcal{S}(\bro))^{-1}\left(\delta\bq\chi_{\omega_q},\delta\Theta\chi_{\omega_h},0,0 \right)^\top$ is the unique solution of \eqref{Blg1}--\eqref{Blg5} with $\bF = \delta\bq\chi_{\omega_q}$, $G = \delta\Theta\chi_{\omega_h}$, $\bv_0 = 0$ and $\vartheta(0) = 0$.
			
			Using the chain rule again, we get $L'(\mathcal{S}(\bro))[\mathcal{S}'(\bro)\delta\bro_1,\mathcal{S}'(\bro)\delta\bro_2] + L(\mathcal{S}(\bro))(\mathcal{S}''(\bro)[\delta\bro_1,\delta\bro_2]) = 0$. Note that for any $(\bv_1,\theta_1),(\bv_2,\theta_2)\in \bW^{2,1}_{s,\sigma}\times W^{2,1}_r$, we have $$L'(\bu,\theta)[(\bv_1,\theta_1),(\bv_2,\theta_2)] = ( (\bv_2\cdot\nabla)\bv_2 + (\bv_2\cdot\nabla)\bv_1, \bv_2\cdot\nabla\vartheta_1 + \bv_1\cdot\nabla\vartheta_2, 0, 0)^\top.$$
			This implies that $\mathcal{S}''(\bro)[\delta\bro_1,\delta\bro_2] = -L'(\mathcal{S}(\bro))^{-1}\left( L'(\mathcal{S}(\bro))[\mathcal{S}'(\bro)\delta\bro_1,\mathcal{S}'(\bro)\delta\bro_2]\right)$ solves \eqref{Blg1}--\eqref{Blg5} with $\bF$, $G$, $\bv_0$ and $\vartheta(0)$ as mentioned in the proposition.
		\end{proof}
		

		Let us define the operator $D:\bL^s(Q)\times L^s(Q)\to  \bL^s(Q)\times L^s(Q)\times \bW^{2-1/s,s}_{0,\sigma}(\Omega)\times W^{2-1/s,s}_0(\Omega)$ by $D(\bro) = (\bq\chi_{\omega_q},\Theta\chi_{\omega_h},0,0)$, its dual can then be defined as $D^*(\bq,\Theta,\bphi,\psi) = (\bq\chi_{\omega_q},\Theta\chi_{\omega_h}) $ for any $(\bphi,\psi)\in \bW^{2-1/s,s}_{0,\sigma}(\Omega)\times W^{2-1/s,s}_0(\Omega)$. From these operators, we see that $\mathcal{S}'(\bro) = L(\mathcal{S}(\bro))^{-1}D$. We thus define the operator $\mathcal{S}'(\bro)^* = D^*\!\left(L(\mathcal{S}(\bro))^{-1}\right)^*$ as the adjoint of  $\mathcal{S}'(\bro)$. Now, define $g_T : \bW^{2,1}_{s,\sigma}\times W^{2,1}_s \to \bW^{2-1/s,s}_{0,\sigma}(\Omega)^*\times W^{2-1/s,s}_0(\Omega)^*$ by 
		$$\langle g_T(\bphi,\psi), (\bv,\vartheta)\rangle = (\bphi,\bv(T)) + (\psi,\vartheta(T)).$$ 
		From this we see that the adjoint variable $(\bw,\Psi) = \mathcal{S}'(\bro)^*( \delta\bro + g_T(\bw_T,\Psi_T) )$ solves \eqref{Badg1}--\eqref{Badg5} with $\bu_1 = \bu_2 = {S}_1(\bro)$ and $\theta = {S}_2(\bro)$, $\bF = \delta{\bq}$, $G = \delta \Theta$, $\bw(T) = \bw_T$ and $\Psi(T) = \Psi_T$. Indeed, for any $\bro^{\star}\in \mathcal{U}$ we have
		\begin{align}
			\langle \bro^{\star}, (\bw,\Psi) \rangle = \langle \mathcal{S}'(\bro)\bro^{\star}, \delta\bro + g_T(\bw_T,\Psi_T) \rangle = \langle (\bv,\vartheta), \delta\bro\rangle + (\bv(T),\bw_T) + (\vartheta(T),\Psi_T),
		\end{align}
		where $(\bv,\vartheta) = \mathcal{S}'(\bro)\bro^{\star}$ and the left-most term should be understood as $\langle \bro^{\star}, (\bw,\Psi) \rangle = (\bq^{\star}\chi_{\omega_q}, \bw ) + (\Theta^{\star}\chi_{\omega_h},\Psi)$.
		On the other hand, since $(\bv, \vartheta)$ solves \eqref{Blg1}--\eqref{Blg5} with $\bu_1 = \bu_2 = S_1(\bro)$ and $\theta = S_2(\bro)$, $\bF = \bq^{\star}\chi_q$, $G = \Theta^{\star}\chi_h$, $\bv_0 = 0$ and $\vartheta_0 = 0$, we get
		\begin{align*}
			\begin{aligned}
				\langle \bro^{\star}, (\bw,\Psi) \rangle 
				& = \int_Q \left(\partial_t \bv - \nu\Delta\bv + (\bu\cdot\nabla)\bv + (\bv\cdot\nabla)\bu + \nabla q - \be_2\vartheta \right)\cdot \bw \du x\du t\\
				&\quad + \int_Q \left(\partial_t\vartheta -\kappa\Delta \vartheta + \bu\cdot\nabla\vartheta + \bv\cdot\nabla\theta \right)\Psi \du x\du t\\
				& = \int_Q \bv\cdot \left(  -\partial_t \bw - \nu\Delta\bw - (\bu\cdot\nabla)\bw + (\nabla\bu)^\top\bw + \Psi\nabla\theta + \nabla r \right)\du x\du t\\
				& \quad \int_Q \vartheta\left( -\partial_t\Psi -\kappa\Delta \Psi + \bu\cdot\nabla \Psi - \be_2\cdot\bw \right) \du x\du t + (\bv(T),\bw(T)) + (\vartheta(T),\Psi(T)).
			\end{aligned}
		\end{align*}
		The arbitrary nature of $\bro^{\star}$ supports our claim. From this, we finally define the control-to-adjoint operator as $$\mathcal{D}(\bro) = \mathcal{S}'(\bro)^*( ( \alpha_1(S_1(\bro) - \bu_d) , \alpha_2(S_2(\bro) - \theta_d))   + g_T( \beta_1(S_1(\bro)(T) - \bu_T), \beta_2(S_2(\bro) - \theta_T) ) ),$$ where $S_1$ and $S_2$ are the components of the control-to-state operator as defined in \eqref{defn:con-to-}.
		For convenience, we define the map
		\begin{align}\label{themapadj}
			&\mathcal{G}(\bu,\theta,\bw, \Psi ) :=  \left[ 
			\begin{aligned}
				-\partial_t \bw - \nu\Delta\bw - (\bu\cdot\nabla)\bw + (\nabla\bu)^\top\bw & + \Psi\nabla\theta + \nabla r - \alpha_1(\bu - \bu_d) \\
				-\partial_t\Psi -\kappa\Delta \Psi - \bu \cdot\nabla \Psi & - \be_2\cdot\bw - \alpha_2(\theta - \theta_d)\\
				\bw(T) &- \beta_1(\bu(T) - \bu_T)\\
				\Psi(T)& - \beta_2(\theta(T) - \theta_T)
			\end{aligned}
			\right].
		\end{align}

		\begin{theorem}
			The objective functional is of class $C^\infty$. Furthermore, the first and second variations can be calculated as stated below:
			\begin{align}\label{jderivative}
				\begin{aligned}
					J'(\bro)(\delta\bro)&=\alpha_1\int_Q(\bu -\bu_d)\cdot\bv \du x\du t+\alpha_2\int_Q ( \theta -\theta_d)\vartheta \du x\du t\\
					&+\beta_1\int_\Omega (\bu(T)-\bu_T)\cdot\bv(T) \du x + \beta_2\int_\Omega (\theta(T)-\theta_T) \vartheta(T)\du x\\
					&= \int_0^T\int_{\omega_q} \bw\cdot\delta\bq \du x\du t + \int_0^T\int_{\omega_h} \Psi\cdot\delta\Theta \du x\du t
				\end{aligned}
			\end{align}
			\begin{align}
				\begin{aligned}
					J''(\bro)(\delta\bro, \delta\bro)&= \alpha_1 \|\bv\|_{\bL^2}^2 + \alpha_2\|\vartheta \|_{L^2}^2 + \beta_1 \|\bv(T) \|_{\bL^2}^2 + \beta_2\|\vartheta(T) \|_{L^2}^2\\
					&\quad +\alpha_1\int_Q(\bu -\bu_d)\cdot\widetilde{\bv}\du x\du t + \alpha_2\int_Q ( \theta -\theta_d)\widetilde{\vartheta}\du x\du t\\
					&\quad +\beta_1\int_\Omega (\bu(T)-\bu_T)\cdot\widetilde{\bv}(T) \du x + \beta_2\int_\Omega ( \theta(T)-\theta_T)\widetilde{\vartheta}(T) \du x\\
					&= \alpha_1 \|\bv\|_{\bL^2}^2 + \alpha_2\|\vartheta\|_{L^2}^2 + \beta_1 \|\bv(T) \|_{\bL^2}^2 + \beta_2\|\vartheta(T) \|_{L^2}^2\\
					&\quad - 2( (\bv \cdot\nabla)\bv, \bw )_Q - 2( \bv\cdot\nabla \vartheta, \Psi)_Q.
				\end{aligned}
			\end{align}
			where $(\bu,\theta) := \mathcal{S}(\bro)$, $(\bv,\vartheta):=\mathcal{S}'(\bro)(\delta\bro)$, $(\bw,\Psi):=  \mathcal{D}(\bro)$ and $(\widetilde{\bv},\widetilde{\vartheta}) = \mathcal{S}''(\bro)[\delta\bro,\delta\bro]$.
		\end{theorem}
		\begin{proof}
			We use the chain rule to prove the theorem. For the first-order derivative, we have
			\begin{align*}
				\begin{aligned}
					J'(\bro)\delta\bro & = \alpha_1\left( S_1(\bro) - \bu_d, S_1'(\bro)\delta\bro \right)_Q + \alpha_2 \left( S_2(\bro) - \theta_d, S_2'(\bro)\delta\bro  \right)_Q\\
					&\quad + \beta_1(  S_1(\bro)(T) - \bu_T, S_1'(\bro)\delta\bro(T) )_\Omega + \beta_2\left( S_2(\bro)(T) - \theta_T, S_2'(\bro)\delta\bro  (T)\right)_\Omega\\ 
					& = \langle ( \alpha_1(S_1(\bro) - \bu_d),\alpha_2(S_2(\bro) - \theta_d) ), \mathcal{S}'(\bro)\delta\bro \rangle\\
					&\quad + \langle g_T( \beta_1(  S_1(\bro)(T) - \bu_T),\beta_2( S_2(\bro)(T) - \theta_T)),\mathcal{S}'(\bro)\delta\bro \rangle\\
					& = \langle (\bw,\Psi),\delta\bro  \rangle = (\bw,\delta\bq  )_{\omega_q} + (\Psi,\delta\Theta)_{\omega_f}.
				\end{aligned}
			\end{align*}

			Using the chain rule once again, we get
			\begin{align*}
				\begin{aligned}
					J''(\bro)[\delta\bro,\delta\bro] & = \alpha_1\|S_1'(\bro)\delta\bro\|_{\bL^2}^2 + \alpha_2\|S_2'(\bro)\delta\bro\|_{L^2}^2 + \beta_1\|S_1'(\bro)\delta\bro(T)\|_{\bL^2}^2 +
					\beta_2\|S_2'(\bro)\delta\bro(T)\|_{L^2}^2\\
					&\quad + \alpha_1\left( S_1(\bro) - \bu_d, S_1''(\bro)[\delta\bro,\delta\bro] \right)_Q + \alpha_2 \left( S_2(\bro) - \theta_d, S_2''(\bro)[\delta\bro,\delta\bro]   \right)_Q\\
					&\quad + \beta_1(  S_1(\bro)(T) - \bu_T, S_1''(\bro)[\delta\bro,\delta\bro] (T) )_\Omega + \beta_2\left( S_2(\bro)(T) - \theta_T, S_2''(\bro)[\delta\bro,\delta\bro]   (T)\right)_\Omega
				\end{aligned}
			\end{align*}
			Denoting for simplicity for now $(\bw,\Psi) = \mathcal{D}(\bro)$, $(\bu,\theta) = \mathcal{S}(\bro)$ and $(\bv,\vartheta) = \mathcal{S}''(\bro)[\delta\bro,\delta\bro]$ we further get
			\begin{align*}
				\begin{aligned}
					J''(\bro)[\delta\bro,\delta\bro] & = \alpha_1\|S_1'(\bro)\delta\bro\|_{\bL^2}^2 + \alpha_2\|S_2'(\bro)\delta\bro\|_{L^2}^2 + \beta_1\|S_1'(\bro)\delta\bro(T)\|_{\bL^2}^2 +
					\beta_2\|S_2'(\bro)\delta\bro(T)\|_{L^2}^2\\
					&\quad + \left( -\partial_t \bw - \nu\Delta\bw - (\bu\cdot\nabla)\bw + (\nabla\bu)^\top\bw + \Psi\nabla\theta + \nabla r, \bv \right)_Q\\
					&\quad + \alpha_2 \left(-\partial_t\Psi -\kappa\Delta \Psi + \bu\cdot\nabla \Psi - \be_2\cdot\bw, \vartheta  \right)_Q + (  \bw(T), \bv(T) )_\Omega + \left( \Psi(T), \vartheta(T)\right)_\Omega\\
					& = \alpha_1\|S_1'(\bro)\delta\bro\|_{\bL^2}^2 + \alpha_2\|S_2'(\bro)\delta\bro\|_{L^2}^2 + \beta_1\|S_1'(\bro)\delta\bro(T)\|_{\bL^2}^2 +
					\beta_2\|S_2'(\bro)\delta\bro(T)\|_{L^2}^2\\
					&\quad + \left( \bw, \partial_t \bv - \nu\Delta\bv + (\bu\cdot\nabla)\bv + (\bv\cdot\nabla)\bu + \nabla q - \be_2\vartheta \right)_Q\\
					&\quad + \alpha_2 \left(\Psi, \partial_t\vartheta -\kappa\Delta \vartheta + \bu\cdot\nabla\vartheta + \bv\cdot\nabla\theta \right)_Q\\
					& = \alpha_1\|S_1'(\bro)\delta\bro\|_{\bL^2}^2 + \alpha_2\|S_2'(\bro)\delta\bro\|_{L^2}^2 + \beta_1\|S_1'(\bro)\delta\bro(T)\|_{\bL^2}^2 +
					\beta_2\|S_2'(\bro)\delta\bro(T)\|_{L^2}^2\\
					&\quad - 2( (\bv \cdot\nabla)\bv, \bw)_Q - 2( \bv\cdot\nabla \vartheta, \Psi)_Q.
				\end{aligned}
			\end{align*}
		\end{proof}

		We can now give the first-order necessary condition whose proof is straightforward.
		\begin{theorem}\label{firstoderoptcon}
			Let $\bro^\star\in \mathcal U$ be a local minimizer of problem \eqref{box}-\eqref{B5} and $(\bw^\star,\Psi^\star) = \mathcal{D}(\bro^{\star})$. Then for all $\bro\in \mathcal U$ it holds
			\begin{equation}
				\int_0^T\left[ \int_{\omega_q} \bw^\star\cdot (\bq- \bq^\star) \du x + \int_{\omega_h} \Psi^\star (\Theta- \Theta^\star) \text{ d}x \right]\text{d}t\geq 0.
			\end{equation}
			Further, testing with controls of the form $\bro=(\bq,\Theta^\star)$ and $\bro=(\bq^\star,\Theta)$ we find that it holds for all $\bro\in \mathcal U$
			\begin{equation}
				\int_0^T \int_{\omega_q} \bw^\star\cdot (\bq- \bq^\star) \du x \text{d}t\geq 0, \text{ and } \int_0^T \int_{\omega_h} \Psi^\star (\Theta- \Theta^\star) \text{ d}x \text{d}t\geq 0.
			\end{equation}
			Even stronger, we have for $i\in \{1,2\}$
			\begin{equation}
				\bw^\star\cdot (\bq- \bq^\star)\geq 0 \text{ a.e. on } [0,T]\times \omega_q, \  \text{ and } \ \  \Psi^\star (\Theta- \Theta^\star) \geq 0 \text{ a.e. on } [0,T]\times \omega_h.
			\end{equation}
		\end{theorem}

		To help us analyze the second-order sufficient conditions, we present the following lemmata.
		\begin{lemma}\label{lemma:diffstab}
			Let $\bro,\bro^\star \in L^s(I;\bL^s(\omega_q)\times L^s(\omega_h))$, $(\bu,\theta) = \mathcal{S}(\bro),(\bu^\star,\theta^\star) = \mathcal{S}(\bro^*) \in \bW^{2,1}_{s,\sigma}\times W^{2,1}_s$. Then there exists a constant $c>0$ such that 
			\begin{align}\label{stcongap}
				\|\bu - \bu^\star \|_{\bW^{2,1}_{s,\sigma}} + \|\theta-\theta^\star\|_{W^{2,1}_s} \le c( \|\bq-\bq^{\star}\|_{\bL^s(\omega_q)} + \|\Theta - \Theta^{\star}\|_{L^s(\omega_h)} )
			\end{align}
		\end{lemma}
		\begin{proof}
			Knowing that $\bv = \bu - \bu^\star \in \bW^{2,1}_{s,\sigma}$ and $\vartheta = \theta-\theta^\star\in W^{2,1}_s$  solves a linear system of the form \eqref{Blg1}--\eqref{Blg5} with $\bu_1 = \bu^\star$, $\bu_2 = \bu$, $\bF = (\bro - \bro^\star)\chi_{\omega_q}$, $G = (\Theta - \Theta^{\star})\chi_{\omega_h}$, $\bv_0 = 0$, $\vartheta_0 = 0$. Utilizing \Cref{theorem:linearLP} gives us the desired estimate.
		\end{proof}

		\begin{lemma}\label{lemma:supp1}
			Let $\bro,\bro^\star \in L^s(I;\bL^s(\omega_q)\times L^s(\omega_h))$, $(\bu,\theta) = \mathcal{S}(\bro),(\bu^\star,\theta^\star) = \mathcal{S}(\bro^*) \in \bW^{2,1}_{s,\sigma}\times W^{2,1}_s$ and $(\bv^\star,\vartheta^\star) = \mathcal{S}'(\bro^{\star})(\bro-\bro^\star)\in \bW^{2,1}_{s,\sigma}\times W^{2,1}_s$. Then there exists $c>0$ such that 
			\begin{align}
				&\begin{aligned}
					\|\bu - &\bu^\star - \bv^{\star}\|_{L^\infty(\bL^2)} + \|\theta-\theta^\star - \vartheta^\star\|_{L^\infty(L^2)}
					\le  c \|\bu - \bu^\star\|_{\bL^\infty} 
					\Big( \|\bu - \bu^\star\|_{\bL^2} + \| \theta - \theta^\star \|_{L^2}\Big)
				\end{aligned} \label{diffvslin}\\
				&\begin{aligned}
					\|\nabla(\bu - &\bu^\star - \bv^{\star} )\|_{\bL^2} + \|\nabla(\theta-\theta^\star - \vartheta^\star) \|_{L^2}
					\le  c \|\bu - \bu^\star\|_{\bL^\infty} 
					\Big( \|\bu - \bu^\star\|_{\bL^2} + \| \theta - \theta^\star\|_{L^2}\Big).
				\end{aligned}
			\end{align}
		\end{lemma}
		
		\begin{proof}
			Note that the variables $\bv = \bu - \bu^\star - \bv^\star\in \bW^{2,1}_{s,\sigma}$ and $\vartheta = \theta - \theta^\star - \vartheta^\star\in W^{2,1}_{s}$ solve a system of the form \eqref{Blg1}--\eqref{Blg5} with $\bu_1 = \bu_2 = \bu^*$, $\theta = \theta^\star$, $\bv_0 = 0$, $\vartheta_0 = 0$, $\bF = -((\bu -\bu^{\star})\cdot\nabla)(\bu - \bu^{\star})$ and $G =-(\bu -\bu^{\star})\cdot\nabla(\theta - \theta^\star)$. Estimate \eqref{estimate:lweak} thus gives us
			\begin{align*}
				\begin{aligned}
					&\|\bv\|_{L^\infty(\bL^2)} + \|\nabla\bv\|_{L^2(\bL^2)} + \|\vartheta\|_{L^\infty(L^2)} + \|\nabla\vartheta\|_{L^2(L^2)} \\
					& \le c(\|((\bu - \bu^{\star})\cdot\nabla)(\bu -\bu^{\star})\|_{L^2(\bV_\sigma^*)} + \|(\bu -\bu^{\star})\cdot\nabla(\theta - \theta^\star)\|_{L^2(V^*)})\\
					& \le \sup_{(\bphi,\psi)\in L^2(\bV)\times L^2(V)\atop \| \bphi\|_{L^2(\bV)} = \|\psi\|_{L^2(V)} = 1} c\Big( \Big|\big(((\bu -\bu^{\star})\cdot\nabla)\bphi,\bu -\bu^{\star}\big)\Big| + \Big|\big(  (\bu -\bu^{\star})\cdot\nabla\psi,\theta - \theta^\star \big) \Big|\Big)\\
					& \le c\|\bu - \bu^\star\|_{\bL^\infty} \Big(  \|\bu - \bu^\star\|_{\bL^2} + \| \theta - \theta^\star\|_{L^2}\Big).
				\end{aligned}
			\end{align*}
		\end{proof}

		\begin{lemma}\label{lemma:supp1b}
			Let $\bro,\bro^\star \in L^s(I;\bL^s(\omega_q)\times L^s(\omega_h))$, $(\bu,\theta) = \mathcal{S}(\bro),(\bu^\star,\theta^\star) = \mathcal{S}(\bro^*) \in \bW^{2,1}_{s,\sigma}\times W^{2,1}_s$ and $(\bv^{\star},\vartheta^{\star}) = \mathcal{S}'(\bro^{\star})(\bro-\bro^\star)\in \bW^{2,1}_{s,\sigma}\times W^{2,1}_s$. Then there exists $c>0$ such that 
			\begin{align}\label{nablavest}
				\|\nabla\bv^{\star}\|_{\bL^2} + \|\nabla\vartheta^{\star}\|_{L^2} \le  c \|\bro - \bro^{\star}\|_{\bL^2\times L^2}^{1/2}\Big(\|\bu       - \bu^\star\|_{\bL^2}^{1/2} + \|\theta - \theta^\star\|_{L^2}^{1/2}  \Big)
			\end{align}
		\end{lemma}
		\begin{proof}
			Recall that $(\bv^{\star},\vartheta^{\star}) = \mathcal{S}'(\bro^{\star})(\bro-\bro^\star)\in \bW^{2,1}_{s,\sigma}\times W^{2,1}_s$ solves the linearized Boussinesq system \eqref{Blg1}--\eqref{Blg5} with $\bu_1 = \bu_2 = \bu$, $\bF=(\bq - \bq^\star)\chi_{\omega_q}$, $G=(\Theta - \Theta^\star)\chi_{\omega_h}$, $\bv_0=0$ and $\vartheta_0=0$.
			Taking the $L^2$ inner product of such system with $(\bv^\star,\vartheta^\star)$ results to 
			\begin{align*}
				\begin{aligned}
					& \frac{1}{2}\frac{d}{dt}\Big( \|\bv^\star\|_{\bL^2}^2 + \|\vartheta^\star\|_{L^2}^2  \Big) + \nu \|\nabla\bv^\star\|_{\bL^2}^2 + \kappa \|\nabla\vartheta^\star\|_{L^2}^2
					\\
					&  \le c\Big( \|\be_2\|_{\bL^\infty}\|\vartheta^\star\|_{\bL^2}\|\bv^\star\|_{\bL^2} + \|\bu\|_{\bL^\infty}\|\bv^\star\|_{\bL^2}\|\nabla\bv^\star\|_{\bL^2}  + \|\theta\|_{L^\infty}\|\bv^\star\|_{\bL^2}\|\nabla\vartheta^\star\|_{\bL^2} \\
					&\quad  + \big| ( (\bq - \bq^\star)\chi_{\omega_q}, \bu - \bu^\star - \bv^\star ) \big| + \big| ( (\bq - \bq^\star)\chi_{\omega_q}, \bu - \bu^\star) \big|\\
					&\quad + \big| ( (\Theta - \Theta^\star)\chi_{\omega_h}, \theta - \theta^\star -  \vartheta^\star ) \big| + \big| ( (\Theta - \Theta^\star)\chi_{\omega_h}, \theta - \theta^\star) \big|\Big)\\
					& \le \mathfrak{c}\left(\|\be_2\|_{\bL^\infty}^2,\|\bu\|_{\bL^\infty}^2,\|\theta \|_{L^\infty}^2 \right)\Big(  \|\bv^\star\|_{\bL^2}^2 + \|\vartheta^\star\|_{L^2}^2 \Big) + \frac{\nu}{2} \|\nabla\bv^\star\|_{\bL^2}^2 + \frac{\kappa}{2} \|\nabla\vartheta^\star\|_{L^2}^2\\
					& \quad + \|\bq - \bq^\star\|_{\bL^2(\omega_q)}\|\bu - \bu^\star - \bv^\star\|_{\bL^2} +  \|\bq - \bq^\star\|_{\bL^2(\omega_q)}\|\bu - \bu^\star\|_{\bL^2}\\
					& \quad + \|\Theta - \Theta^\star\|_{L^2(\omega_h)}\|\theta - \theta^\star -  \vartheta^\star\|_{L^2} +  \|\Theta - \Theta^\star\|_{L^2(\omega_h)}\|\theta - \theta^\star\|_{L^2}.
				\end{aligned}
			\end{align*}
			Moving the terms with the gradient on the last line of the computation above the left-hand side, and using \eqref{diffvslin}, we get
			\begin{align}\label{prestimate}
				\begin{aligned}
					& \frac{1}{2}\Big( \|\bv^\star\|_{\bL^2}^2 + \|\vartheta^\star\|_{L^2}^2  \Big) +  \frac{\nu}{2} \|\nabla\bv^\star\|_{\bL^2}^2 + \frac{\kappa}{2} \|\nabla\vartheta^\star\|_{L^2}^2\\
					&\le  \int_0^T  \mathfrak{c}\left(\|\be_2\|_{\bL^\infty}^2,\|\bu\|_{\bL^\infty}^2,\|\theta\|_{L^\infty}^2 \right)\Big(  \|\bv^\star\|_{\bL^2}^2 + \|\vartheta^\star\|_{L^2}^2 \Big) \du t\\
					&\quad  + \int_0^T \|\bro - \bro^\star\|_{\bL^2\times L^2} \left(\|\bu - \bu^\star\|_{\bL^\infty} + 1 \right) \Big(\|\bu - \bu^\star\|_{\bL^2} + \|\theta - \theta^\star\|_{L^2}  \Big) \du t.
				\end{aligned}
			\end{align}
			
			The Gr{\"o}nwal lemma thus leads to
			\begin{align*}
				\|\bv^\star\|_{L^\infty(\bL^2)}^2 + \|\vartheta^\star\|_{L^\infty(L^2)}^2 \le \|\bro - \bro^{\star}\|_{\bL^2\times L^2}\Big(\|\bu - \bu^\star\|_{\bL^2} + \|\theta - \theta^\star\|_{L^2}  \Big),
			\end{align*}
			where we used the fact that $\|\bu - \bu^\star\|_{\bL^\infty} \le cM_{\mathcal{U}}^{1/s}$ where $s>2$. Plugging the estimate above into \eqref{prestimate} yields our desired estimate.
		\end{proof}
		
		\begin{lemma}\label{lemma:supp2}
			Let $\bro,\bro^\star \in L^s(I;\bL^s(\omega_q)\times L^s(\omega_h))$, $(\bu,\theta) = \mathcal{S}(\bro),(\bu^\star,\theta^\star) = \mathcal{S}(\bro^*) \in \bW^{2,1}_{s,\sigma}\times W^{2,1}_s$ and $(\bv^\star,\vartheta^\star) = \mathcal{S}'(\bro^{\star})(\bro-\bro^\star)\in \bW^{2,1}_{s,\sigma}\times W^{2,1}_s$. There exists $\delta>0$ such that whenever $\|\bro-\bro^\star\|_{\bL^1\times L^1}< \delta$ we have 
			\begin{align}\label{estimate:diff-lin}
				\|\bu - \bu^\star\|_{\bL^2} + \|\theta -\theta^\star\|_{L^2} \le 2\left(\| \bv^\star \|_{\bL^2} + \|\vartheta^\star\|_{L^2}\right)\le 3\Big(\|\bu - \bu^\star\|_{\bL^2} + \|\theta - \theta^\star\|_{L^2} \Big).
			\end{align}
		\end{lemma}
		\begin{proof}
			From \eqref{diffvslin} we see that
			\begin{align*}
				\begin{aligned}
					\|\bu - \bu^\star\|_{\bL^2} + &\|\theta -\theta_{\bro^\star}\|_{L^2}   \le  \|\bu - \bu^\star - \bv^\star\|_{\bL^2} + \|\bv^\star\|_{\bL^2} + \|\theta - \theta^\star - \vartheta^\star\|_{L^2} + \|\vartheta^\star\|_{L^2}\\
					& \le c_1 \|\bu  - \bu^\star\|_{\bL^\infty} 
					\Big( \|\bu - \bu^\star\|_{\bL^2} + \| \theta - \theta^\star\|_{L^2}\Big) + \|\bv^\star\|_{\bL^2} + \|\vartheta^\star\|_{L^2}.
				\end{aligned}
			\end{align*}
			Meanwhile, \eqref{stcongap} and the embedding $\bW^{2,1}_{s,\sigma}\hookrightarrow C(\overline{Q})^2$ imply that
			\begin{align*}
				\begin{aligned}
					\|\bu - \bu^\star\|_{\bL^\infty} & \le c( \|\bq-\bq^{\star}\|_{\bL^s(\omega_q)} + \|\Theta - \Theta^{\star}\|_{L^s(\omega_h)} ) \le c_2M_{\mathcal{U}}^{(s-1)/s}\|\bro-\bro^{\star}\|_{\bL^1\times L^1}^{1/s}.
				\end{aligned}
			\end{align*}
			By choosing $\delta >0 $ such that $\delta \le 1/ (2c_1c_2)^{s}M_{\mathcal{U}}^{s-1}$, we get that if $\|\bro-\bro^{\star}\|_{\bL^1\times L^1} < \delta$ then 
			\begin{align*}
				\frac{1}{2}\left(\|\bu - \bu^\star\|_{\bL^2} + \|\theta -\theta^\star\|_{L^2}\right) \le \|\bv^\star\|_{\bL^2} + \|\vartheta^\star\|_{L^2}.
			\end{align*}
			On the other hand, using similar arguments and similar choice of $\delta>0$ yield
			\begin{align*}
				\begin{aligned}
					&\|\bv^\star\|_{\bL^2} + \|\vartheta^\star\|_{L^2} \le \|\bu - \bu^\star - \bv^\star\|_{\bL^2} + \|\theta -\theta^\star - \vartheta^\star\|_{L^2} + \|\bu - \bu^\star \|_{\bL^2} + \|\theta -\theta^\star\|_{L^2}\\
					&\le c_1c_2M_{\mathcal{U}}^{(s-1)/s}\|\bro-\bro^{\star}\|_{\bL^1\times L^1}^{1/s}\Big( \|\bu - \bu^\star\|_{\bL^2} + \| \theta - \theta^\star\|_{L^2}\Big) + \|\bu - \bu^\star \|_{\bL^2} + \|\theta -\theta^\star\|_{L^2}\\
					&\le \frac{3}{2}\Big( \|\bu - \bu^\star\|_{\bL^2} + \| \theta - \theta^\star\|_{L^2}\Big).
				\end{aligned}
			\end{align*}
		\end{proof}
		
		\begin{lemma}\label{lemma:supp3}
			Let $\bro,\bro^\star,\delta\bro \in L^s(I;\bL^s(\omega_q)\times L^s(\omega_h))$, $(\bu,\theta) = \mathcal{S}(\bro),(\bu^\star,\theta^\star) = \mathcal{S}(\bro^*) \in \bW^{2,1}_{s,\sigma}\times W^{2,1}_s$ and $(\bv,\vartheta) = \mathcal{S}'(\bro)\delta\bro,(\bv^\star,\vartheta^\star) = \mathcal{S}'(\bro^{\star})\delta\bro\in \bW^{2,1}_{s,\sigma}\times W^{2,1}_s$. Then there exists $c>0$ such that 
			\begin{align}\label{linstgap}
				\begin{aligned}
					&\|\bv -\bv^\star\|_{L^\infty(\bL^2)} + \|\vartheta -\vartheta^\star\|_{L^\infty(L^2)} + \|\bv -\bv^\star\|_{W^2(\bV_\sigma)} + \|\vartheta -\vartheta^\star\|_{W^2(V)}\\
					&\le c\Big( \|\bu^\star -\bu \|_{\bL^\infty} + \|\theta^\star -\theta \|_{L^\infty}\Big)\Big( \|\bv^\star \|_{\bL^2} + \|\vartheta^\star \|_{L^2} \Big).
				\end{aligned}
			\end{align}
			As a consequence, one can find a $\delta>0$ such that whenever $\|\bro - \bro^\star\|_{\bL^1\times L^1} < \delta$ we have
			\begin{align}
				&\|\bv^\star \|_{\bL^2} + \|\vartheta^\star\|_{L^2} \le 2\Big( \|\bv \|_{\bL^2} + \|\vartheta \|_{L^2} \Big) \le 3\Big(  \|\bv^\star \|_{\bL^2} + \|\vartheta^\star \|_{L^2} \Big).\label{estimate:linearcomp}
			\end{align}
		\end{lemma}
		\begin{proof}
			The elements $\overline{\bv} = \bv -\bv^\star \in\bW^{2,1}_{s,\sigma}$ and $\overline{\vartheta} = \vartheta -\vartheta^\star\in W^{2,1}_s$ solves a system of the form \eqref{Blg1}--\eqref{Blg5} with $\bu_1 = \bu_2 = \bu$, $\bF = ((\bu^\star - \bu)\cdot\nabla)\bv^\star + (\bv^\star\cdot\nabla)(\bu^\star - \bu)$, and $G = (\bu^\star - \bu)\cdot\nabla\vartheta^\star + \bv^\star\cdot\nabla(\theta^\star-\theta)$, with initial data both equal to zero. Using the same arguments as in the proof of \Cref{lemma:supp1} with the help of \Cref{theorem:weaklBouss} we get the following computation
			\begin{align}
				\begin{aligned}
					&\|\overline{\bv}\|_{L^\infty(\bL^2)} + \|\overline{\bv}\|_{W^2(\bV_\sigma)} + \|\overline{\vartheta}\|_{L^\infty(L^2)} + \|\overline{\vartheta}\|_{W^2(V)} \\
					& \le c\Big(\|((\bu^\star - \bu)\cdot\nabla)\bv^\star\|_{L^2(\bV_\sigma^*)} + \|(\bv^\star\cdot\nabla)(\bu^\star - \bu)\|_{L^2(\bV_\sigma^*)}\\
					&\qquad + \|(\bu - \bu)\cdot\nabla\vartheta^\star\|_{L^2(V^*)} + \|\bv^\star\cdot\nabla(\theta^\star -\theta)\|_{L^2(V^*)}\Big)\\
					& \le c\Big( \|\bu^\star - \bu\|_{\bL^\infty}\left(2\|\bv^\star\|_{\bL^2} + \|\vartheta^\star\|_{L^2}   \right) +  \|\theta^\star - \theta\|_{L^\infty}\|\bv^\star\|_{\bL^2} \Big)\\
					& \le c\Big( \|\bu^\star - \bu\|_{\bL^\infty} + \|\theta^\star -\theta\|_{L^\infty}\Big)\Big( \|\bv^\star \|_{\bL^2} + \|\vartheta^\star \|_{L^2} \Big).
				\end{aligned}\label{estimate1}
			\end{align}
			From \eqref{estimate1} we thus infer that
			\begin{align*}
				\begin{aligned}
					& \|\bv^\star\|_{\bL^2} + \|\vartheta^\star\|_{L^2}\le \|\bv - \bv^\star\|_{\bL^2} + \|\vartheta - \vartheta^\star\|_{L^2} + \|\bv\|_{\bL^2} + \|\vartheta\|_{L^2} \\
					& \le c\Big( \|\bu^\star - \bu \|_{\bL^\infty} + \|\theta^\star -\theta \|_{L^\infty}\Big)\Big( \|\bv^\star \|_{\bL^2} + \|\vartheta^\star \|_{L^2} \Big) + \|\bv\|_{\bL^2} + \|\vartheta\|_{L^2}.
				\end{aligned}
			\end{align*}
			Using same argument as we did in \Cref{lemma:supp1}, we can find an $\delta>0$ such that if $\|\bro - \bro^\star\|_{\bL^1\times L^1} < \delta$ then $c\Big( \|\bu^\star-\bu \|_{\bL^\infty} + \|\theta^\star - \theta \|_{L^\infty}\Big) \le \frac{1}{2}$. It holds
			\begin{align*}
				\begin{aligned}
					\|\bv^\star\|_{\bL^2} + \|\vartheta^\star\|_{L^2} \le 2\Big(\|\bv\|_{\bL^2} + \|\vartheta\|_{L^2}\Big).
				\end{aligned}
			\end{align*}
			Similar steps can be done to get the other inequality in \eqref{estimate:linearcomp}. 
		\end{proof}
		
		\begin{remark}\label{remark:toweak}
			We mention that \Cref{lemma:supp2} and the second part of \Cref{lemma:supp3} can be established with a weaker assumption, i.e. there exists $\delta>0$ such that whenever $\|\bu - \bu^\star\|_{\bL^\infty} + \|\theta - \theta^\star\|_{L^\infty} < \delta$ we get \eqref{estimate:diff-lin} and \eqref{estimate:linearcomp}.
		\end{remark}
		
		\begin{lemma}\label{lemma:adjstability}
			Let $s>2$, $\bro,\bro^\star \in \mathcal{U}$, $(\bu,\theta) = \mathcal{S}(\bro),(\bu^\star,\theta^\star) = \mathcal{S}(\bro^*) \in \bW^{2,1}_{s,\sigma}\times W^{2,1}_s$ and $(\bw,\Psi) = \mathcal{D}(\bro),(\bw^\star,\Psi^\star) = \mathcal{D}(\bro^{\star})\in \bW^{2,1}_{s,\sigma}\times W^{2,1}_s$. There exists $c>0$ such that
			\begin{align}\label{nablaadjgap}
				\begin{aligned}
					&\|\nabla(\bw - \bw^\star)\|_{L^\infty(\bL^2)} + \|\nabla(\Psi - \Psi^\star )\|_{L^\infty(L^2)}\\
					&\le c\Big(\|\bu -\bu^\star\|_{\bL^2} + \|\theta - \theta^\star\|_{L^2} + \beta_1 \|\bu-\bu^\star\|_{\bW^{2,1}_{2,\sigma}} + \beta_2\| \theta - \theta^\star\|_{W^{2,1}_2}\Big).
				\end{aligned}
			\end{align}
		\end{lemma}
		
		\begin{proof}
			Following the same arguments as in the proof of \Cref{theorem:strongLP} we can get
			\begin{align*}
				\begin{aligned}
					& \|\nabla(\bw - \bw^\star)\|_{L^\infty(\bL^2)} + \|\nabla(\Psi - \Psi^\star)\|_{L^\infty(L^2)}
					\le c \left(\|\bF\|_{\bL^2} + \|G\|_{L^2} + \|\bw_T\|_{\bV_\sigma} +  \|\Psi_T\|_{V}\right),
				\end{aligned}
			\end{align*}
			where 
			\begin{align*}
				\begin{aligned}
					& \bF = ( (\bu - \bu^\star)\cdot\nabla )\bw + (\nabla(\bu^\star - \bu))^\top\bw + \Psi\nabla(\theta^\star - \theta) + \alpha_1(\bu - \bu^\star),\\
					& G = (\bu^\star - \bu)\cdot\nabla\Psi + \alpha_2(\theta - \theta^\star),\\
					& \bw_T = \beta_1(\bu(T) - \bu^\star(T)),\text{ and }\Psi_T =  \beta_2(\theta(T)-\theta^\star(T)).
				\end{aligned}
			\end{align*}
			
			Before we move forward in estimating $\bF$ and $G$, we take note that from \Cref{theorem:adjointLP}, for any $s>4$ we have
			\begin{align}
				\begin{aligned}\label{gradadest}
					&\|\bw\|_{\bW^{2,1}_{s,\sigma}} + \|\Psi\|_{W^{2,1}_{s}}\\ & \le c\Big( \alpha_1\|\bu - \bu_d\|_{\bL^{s}} + \alpha_2\|\theta - \theta_d\|_{L^{s}} + \beta_1\|\bu(T) - \bu_T\|_{\bW^{2-1/s,s}_{0,\sigma}} + \beta_2\|\theta(T) - \theta_T\|_{W^{2-1/s,s}_{0}} \Big)\\
					& \le c\Big(  \|\bu\|_{W^{2,1}_{s,\sigma}} + \|\theta\|_{W^{2,1}_{s}} + \|\bu_d\|_{\bL^{s}} + \|\theta_d\|_{L^{s}} + \|\bu_T\|_{\bW^{2-1/s,s}_{0,\sigma}} + \|\theta_T\|_{W^{2-1/s,s}_{0}} \Big)\\
					& \le c\Big(  \|\bro\|_{\bL^s\times L^{s}} + \|\bu_0\|_{\bW^{2-1/s,s}_{0,\sigma}} + \|\theta_0\|_{W^{2-1/s,s}_{0}} + \|\bu_d\|_{\bL^{s}} + \|\theta_d\|_{L^{s}} + \|\bu_T\|_{\bW^{2-1/s,s}_{0,\sigma}} + \|\theta_T\|_{W^{2-1/s,s}_{0}} \Big)\\
					& \le c\Big(  M_{\mathcal{U}}^{1/s} + \|\bu_0\|_{\bW^{2-1/s,s}_{0,\sigma}} + \|\theta_0\|_{W^{2-1/s,s}_{0}} + \|\bu_d\|_{\bL^{s}} + \|\theta_d\|_{L^{s}} + \|\bu_T\|_{\bW^{2-1/s,s}_{0,\sigma}} + \|\theta_T\|_{W^{2-1/s,s}_{0}} \Big)
				\end{aligned}
			\end{align}
			Since all of the terms in the last line are fixed, we can denote the last line of the inequality as $M_{\mathcal{A}}$, and as a consequence of the embedding $\bW^{2,1}_{s,\sigma}\times W^{2,1}_{s}\hookrightarrow C(\overline{I};C^1(\overline{\Omega})^2)\times C(\overline{I};C^1(\overline{\Omega}))$ we get
			\begin{align}\label{uniboundadj}
				\|\nabla\bw\|_{\bL^\infty} + \|\nabla\Psi\|_{L^\infty} \le M_{\mathcal{A}}.
			\end{align}
			
			The boundedness above helps us in estimating $\bF$ and $G$. In fact, we have for any $\bphi\in \bL^{2}(Q)$ with $\|\bphi\|_{\bL^{2}} \le 1$
			\begin{align*}
				\begin{aligned}
					\|\bF\|_{\bL^2} & \le (\|\nabla\bw\|_{\bL^\infty} + \alpha_1) \|\bu - \bu^\star\|_{\bL^2} + \left|\left( (\bphi\cdot\nabla)(\bu - \bu^\star),\bw_{\bro} \right)\right| + | (\bphi\cdot\nabla(\theta^\star - \theta), \Psi ) | \\
					& \le (M_{\mathcal{A}} + \alpha_1)\|\bu - \bu^\star\|_{\bL^2} + M_{\mathcal{A}}\|\bphi\|_{\bL^{2}}(\|\bu - \bu^\star\|_{\bL^2} + \|\theta - \theta^\star\|_{L^2})\\
					& \le (2M_{\mathcal{A}} + \alpha_1)\|\bu - \bu^\star\|_{\bL^2} + M_{\mathcal{A}}\|\theta - \theta^\star\|_{L^2}.
				\end{aligned}
			\end{align*}
			Similarly, using H{\"o}lder inequality 
			\begin{align*}
				\begin{aligned}
					\|G\|_{L^2} \le M_{\mathcal{A}}\|\bu - \bu^\star\|_{\bL^2} + \alpha_1\|\theta - \theta^\star\|_{L^2}.
				\end{aligned}
			\end{align*}
			
			Finally, from the embedding $\bW^{2,1}_{2,\sigma}\times W^{2,1}_2 \hookrightarrow C(\overline{I};\bV_\sigma\times V) $ we get the estimates needed for $\bw_T$ and $\Psi_T$.  
		\end{proof}

		\begin{lemma}\label{secestc}
			Let $\bro^\star\in \mathcal{U}$. There exists $\mu\in [1,2)$ such that for every $\varepsilon>0 $ one finds a $\delta>0$ such that
			\begin{align}
				\vert [ J''(\bro ) & - J''(\bro^\star)](\bro-\bro^\star)^2\vert \leq \varepsilon \Big ( \|\bro-\bro^\star\|_{\bL^1\times L^1}^{1+\mu} + \beta_1\| \bv^\star(T)\|_{\bL^2}^2+\beta_2\| \vartheta^\star(T)\|_{L^2}^2 \Big),
			\end{align}
			for all $\bro\in \mathcal{U}$ with  $\|\bro-\bro^\star\|_{\bL^1\times L^1}<\delta$, where $(\bv^\star,\vartheta^\star) = \mathcal{S}'(\bro^\star)(\bro-\bro^\star)$.
		\end{lemma}
		\begin{proof}
			Let $(\bv,\vartheta) = \mathcal{S}'(\bro)(\bro-\bro^\star)$, $(\bu^\star,\theta^\star) = \mathcal{S}(\bro^\star)$, $(\bu,\theta) = \mathcal{S}(\bro)$, $(\bw,\Psi) = \mathcal{D}(\bro)$ and $(\bw^\star,\Psi^\star) = \mathcal{D}(\bro^\star)$. Then
			\begin{align*}
				&\Big[J''(\bro) - J''(\bro^\star)\Big](\bro-\bro^\star)^2= \alpha_1\big( \|\bv \|_{\bL^2}^2-\|\bv^\star\|_{\bL^2}^2\Big) + \alpha_2\Big(\|\vartheta\|_{L^2}^2-\|\vartheta^\star\|_{L^2}^2\Big)\\
				&+ \beta_1\Big(\|\bv(T) \|_{\bL^2}^2 - \|\bv^\star(T) \|_{\bL^2}^2\Big) +\beta_2\Big(\|\vartheta(T) \|_{L^2}^2-\|\vartheta^\star(T) \|_{L^2}^2\Big)\\
				&+2\Big[-( (\bv \cdot\nabla)\bv, \bw )_Q - ( \bv\cdot\nabla \vartheta, \Psi)_Q + ( (\bv^\star \cdot\nabla)\bv^\star, \bw^\star )_Q + ( \bv^\star\cdot\nabla \vartheta^\star, \Psi^\star)_Q\Big]\\
				&=I_1+I_2+I_3+I_4+I_5.
			\end{align*}
			We estimate the terms $I_1$ and  $I_2$ by virtue of H{\"o}lder inequality and \eqref{estimate:linearcomp}
			\begin{align*}
				&\begin{aligned}
					\vert I_1 \vert&\leq \alpha_1\|\bv + \bv^\star \|_{\bL^2}\|\bv - \bv^\star\|_{\bL^2} \leq \alpha_1(\|\bv\|_{\bL^2}+\|\bv^\star\|_{\bL^2})\|\bv-\bv^\star\|_{\bL^2}\\
					&\leq 2\alpha_1c (\|\bv\|_{\bL^2}^2 + \|\vartheta\|_{\bL^2}^2)\| \bu^\star-\bu\|_{\bL^\infty}.
				\end{aligned}\\
				&\begin{aligned}
					\vert I_2 \vert&\leq \alpha_2\|\vartheta+\vartheta^\star \|_{L^2}\|\vartheta -\vartheta^\star \|_{L^2} \leq \alpha_2(\|\vartheta\|_{L^2}+\|\vartheta^\star\|_{L^2})\|\vartheta -\vartheta^\star\|_{L^2}\\
					&\leq 2\alpha_2c (\|\bv \|_{\bL^2}^2 + \|\vartheta\|_{\bL^2}^2) \| \theta^\star -  \theta\|_{L^\infty}.
				\end{aligned}
			\end{align*}
			Taking the sum of these two estimates, using \eqref{stcongap} and \eqref{estimate:estLsL1} would then give us for any $\tilde{s}\in (3/2,2)$
			\begin{align*}
				\begin{aligned}
					&|I_1| + |I_2|  \le c\Big(\|\bv\|_{\bL^2}^2 + \|\vartheta \|_{\bL^2}^2 \Big)\Big(\| \bu^\star-\bu\|_{\bL^\infty} + \| \theta^\star -\theta\|_{L^\infty} \Big)\\
					& \le c\Big(\|\bv \|_{\bL^{\tilde{s}}}^{\tilde{s}}\|\bv\|_{\bL^{\infty}}^{2-\tilde{s}} + \|\vartheta\|_{L^{\tilde{s}}}^{\tilde{s}}\|\vartheta \|_{L^{\infty}}^{2-\tilde{s}} \Big)\| \bro^\star - \bro \|_{\bL^{\tilde{s}'}\times L^{\tilde{s}'}}\\
					& \le c\| \bro^\star - \bro \|_{\bL^{1}\times L^{1}}^{\tilde{s}}\| \bro^\star - \bro \|_{\bL^{\tilde{s}'}\times L^{\tilde{s}'}}^{2-\tilde{s}}\| \bro^\star - \bro \|_{\bL^{\tilde{s}'}\times L^{\tilde{s}'}} \le c\| \bro^\star - \bro \|_{\bL^{1}\times L^{1}}^{\frac{3-\tilde{s}}{\tilde{s}'} + \tilde{s}}
				\end{aligned}
			\end{align*}
			
			For $I_3$ and $I_4$ we utilize the embedding $W^2(\bV_\sigma)\times W^2(V) \hookrightarrow C(\overline{I};\bH_\sigma\times H)$ and \eqref{linstgap}
			\begin{align*}
				&\begin{aligned}
					\vert I_3 \vert & \leq \beta_1\|\bv(T)+\bv^\star(T)\|_{\bL^2}\|\bv(T) - \bv^\star(T)\|_{\bL^2} \leq \beta_1(\|\bv(T)\|_{\bL^2} + \|\bv^\star(T)\|_{\bL^2})\|\bv-\bv^\star\|_{W^{2}(\bV_\sigma)}\\
					&\leq c \beta_1( \|\bv(T) - \bv^\star(T) \|_{\bL^2} + 2\|\bv^\star(T)\|_{\bL^2})\|\bv-\bv^\star \|_{W^{2}(\bV_\sigma)}\\
					&\leq c \beta_1( \|\bv  - \bv^\star\|_{W^{2}(\bV_\sigma)} + 2\|\bv^\star(T)\|_{\bL^2})\|\bv-\bv^\star \|_{W^{2}(\bV_\sigma)}\\
					& \leq c\beta_1\Big[ \Big( \|\bu^\star-\bu\|_{\bL^\infty} + \|\theta^\star-\theta\|_{L^\infty}\Big)\Big( \|\bv^\star \|_{\bL^2} + \|\vartheta^\star \|_{L^2} \Big) + 2\|\bv^\star(T)\|_{\bL^2} \Big]\|\bv-\bv^\star \|_{W^{2}(\bV_\sigma)}
				\end{aligned}\\
				&\begin{aligned}
					\vert I_4 \vert & \leq \beta_2\|\vartheta(T)+\vartheta^\star(T)\|_{L^2}\|\vartheta(T)-\vartheta^\star(T) \|_{L^2} \leq \beta_2(\|\vartheta(T)\|_{L^2} + \|\vartheta^\star(T)\|_{L^2})\|\vartheta-\vartheta^\star\|_{W^{2}(V)}\\
					&\leq c\beta_2( \|\vartheta(T) -\vartheta^\star(T)\|_{L^2} + 2\|\vartheta^\star(T)\|_{L^2})\|\vartheta-\vartheta^\star\|_{W^{2}(V)}\\
					& \leq c\beta_2( \|\vartheta -\vartheta^\star\|_{W^{2}(V)} + 2\|\vartheta^\star(T)\|_{L^2})\|\vartheta-\vartheta^\star \|_{W^{2}(V)}\\
					& \leq c\beta_2\Big[\Big( \|\bu^\star - \bu \|_{\bL^\infty} + \|\theta^\star - \theta\|_{L^\infty}\Big)\Big( \|\bv^\star \|_{\bL^2} + \|\vartheta^\star \|_{L^2} \Big) + 2\|\vartheta^\star(T)\|_{L^2} \Big]\|\vartheta -\vartheta^\star\|_{W^{2}(V)}.
				\end{aligned}
			\end{align*}
			The sum of the two estimates above can then be further majorized by using \eqref{linstgap}, \eqref{stcongap}, and \eqref{estimate:estLsL1}.
			\begin{align}
				\begin{aligned}
					\vert I_3 \vert &  + \vert I_4 \vert 
					\\ 
					\le &\,  c\Big[\Big(\beta_1\|\bv^\star(T)\|_{\bL^2} + \beta_2\|\vartheta^\star(T)\|_{L^2} \Big) + \Big( \|\bu^\star-\bu\|_{\bL^\infty} + \|\theta^\star-\theta \|_{L^\infty}\Big)\Big( \|\bv^\star\|_{\bL^2} + \|\vartheta^\star\|_{L^2} \Big) \Big]\\
					&\quad \times \Big(\|\bv -\bv^\star \|_{W^{2}(\bV_\sigma)} +  \|\vartheta-\vartheta^\star\|_{W^{2}(V)}\Big)\\
					\le & \, c\Big[ \Big(\beta_1\|\bv^\star(T)\|_{\bL^2} + \beta_2\|\vartheta^\star(T)\|_{L^2} \Big) + \|\bro^\star - \bro\|_{\bL^{\tilde{s}'}\times L^{\tilde{s}'}}\Big( \|\bv\|_{\bL^2} + \|\vartheta \|_{L^2} \Big)\Big]\\
					&\times \Big( \|\bu^\star-\bu\|_{\bL^\infty} + \|\theta^\star-\theta \|_{L^\infty}\Big)\Big( \|\bv^\star\|_{\bL^2} + \|\vartheta^\star\|_{L^2} \Big)\\
					\le & \,c\Big(\beta_1\|\bv^\star(T)\|_{\bL^2} + \beta_2\|\vartheta^\star(T)\|_{L^2} \Big)\|\bro^\star - \bro\|_{\bL^{\tilde{s}'}\times L^{\tilde{s}'}}\Big( \|\bu^\star -\bu \|_{\bL^2} + \|\theta^\star-\theta \|_{L^2}\Big)\\
					& +  c\|\bro^\star - \bro\|_{\bL^{\tilde{s}'}\times L^{\tilde{s}'}}^2\Big( \|\bv \|_{\bL^2}^2 + \|\vartheta\|_{L^2}^2 \Big) \\
					\le & \, c\Big(\beta_1\|\bv^\star(T)\|_{\bL^2}^2 + \beta_2\|\vartheta(T)\|_{L^2}^2 \Big)\|\bro^\star - \bro\|_{\bL^{1}\times L^{1}}^{1/\tilde{s}'}\|\bro^\star - \bro\|_{\bL^{2}\times L^{2}}\\
					& + c\|\bro^\star - \bro\|_{\bL^{\tilde{s}'}\times L^{\tilde{s}'}}^2\Big( \|\bv \|_{\bL^2}^2 + \|\vartheta\|_{L^2}^2 \Big)\\
					\le & \, c\| \bro^\star - \bro \|_{\bL^{1}\times L^{1}}^{\frac{3-\tilde{s}}{\tilde{s}'} + \tilde{s} + 1} + c\Big(\beta_1\|\bv^\star(T)\|_{\bL^2}^2 + \beta_2\|\vartheta^\star(T)\|_{L^2}^2 \Big)\|\bro^\star - \bro\|_{\bL^{1}\times L^{1}}^{1/\tilde{s}' + 1/2}.
				\end{aligned}        
			\end{align}

			To estimate the term $I_5$, let us first rewrite $I_5$ as
			\begin{align*}
				\frac{1}{2} I_5&=\Big[( (\bv^\star \cdot\nabla)\bv^\star, \bw^\star )_Q-( (\bv \cdot\nabla)\bv, \bw)_Q\Big]+ \Big[( \bv^\star\cdot\nabla \vartheta^\star, \Psi^\star)_Q- ( \bv\cdot\nabla \vartheta, \Psi)_Q\Big]\\
				& = \Big[ ( ((\bv^\star -\bv)\cdot\nabla)\bv^\star, \bw^\star )_Q + ( (\bv\cdot\nabla)(\bv^\star -\bv), \bw^\star )_Q + ( (\bv\cdot\nabla)\bv, \bw^\star-\bw)_Q\Big]\\
				&+\Big[( (\bv^\star -\bv)\cdot\nabla\vartheta^\star, \Psi^\star )_Q + ( \bv\cdot\nabla(\vartheta^\star-\vartheta), \Psi^\star )_Q + ( \bv\cdot\nabla\vartheta, \Psi^\star-\Psi )_Q\Big].
			\end{align*}
			Using the antisymmetry of the trilinear forms, the H{\"o}lder and Gagliardo-Nirenberg inequalities, and the estimates \eqref{nablavest}, \eqref{linstgap}, \eqref{nablaadjgap} and \eqref{uniboundadj}, we get
			\begin{align*}
				\begin{aligned}
					&\frac{1}{2} |I_5|  \le  \|\nabla\bw\|_{\bL^\infty}\|\bv^\star - \bv\|_{\bL^2}\Big(\|\bv^\star\|_{\bL^2} + \|\bv\|_{\bL^2} \Big) + \|\nabla( \bw^\star - \bw) \|_{L^\infty(\bL^2)}\|\bv \|_{\bL^4}^2\\
					&\quad + \|\nabla\Psi^\star\|_{\bL^\infty}\Big( \|\bv^\star - \bv\|_{\bL^2}\|\vartheta^\star\|_{L^2} + \|\bv\|_{\bL^2} \|\vartheta^\star - \vartheta\|_{L^2} \Big) + \|\nabla( \Psi^\star - \Psi) \|_{L^\infty(L^2)}\|\bv \|_{\bL^4}\|\vartheta \|_{L^4}\\
					& \le M_{\mathcal{A}}\Big( \|\bv^\star - \bv\|_{\bL^2} + \|\vartheta^\star - \vartheta\|_{L^2} \Big)\Big(  \|\bv\|_{\bL^2} + \|\vartheta\|_{L^2} \Big)+ \Big(\|\nabla( \bw^\star - \bw ) \|_{L^\infty(\bL^2)} +  \|\nabla( \Psi^\star - \Psi) \|_{L^\infty(L^2)}\Big)\\
					& \quad \times \Big(  \|\bv\|_{\bL^2} + \|\vartheta\|_{L^2} \Big)\Big(  \|\nabla\bv\|_{\bL^2} + \|\nabla\vartheta\|_{\bL^2} \Big)\\
					& \le M_{\mathcal{A}}\Big( \|\bu^\star - \bu\|_{\bL^\infty} + \|\theta^\star - \theta\|_{L^\infty} \Big)\Big(  \|\bv\|_{\bL^2}^2 + \|\vartheta\|_{L^2}^2 \Big) + \Big( \|\bv\|_{\bL^2} + \|\vartheta\|_{L^2} \Big) \Big(  \|\bu^\star - \bu\|_{\bL^2}^{1/2} + \|\theta^\star - \theta\|_{L^2}^{1/2} \Big)\\
					&\quad \times \Big(  \|\bu^\star - \bu\|_{\bL^2} + \|\theta^\star - \theta\|_{L^2} +  \beta_1\|\bu^\star - \bu\|_{\bW^{2,1}_{2,\sigma}} + \beta_2\|\theta^\star - \theta\|_{W^{2,1}_{2}} \Big)\|\bro - \bro^\star \|_{\bL^2\times L^2}^{1/2}.
				\end{aligned}
			\end{align*}        
			After some rearrangement, and by vitue of \eqref{estimate:diff-lin} and \eqref{estimate:linearcomp}, we see that if $\|\bro-\bro^\star\|_{\bL^1\times L^1}<\min\{ \delta_1, \delta_2 \}$ -- where $\delta_1,\delta_2>0$ are as prescribed in \Cref{lemma:supp2} and \Cref{lemma:supp3} --- then
			\begin{align*}
				\begin{aligned} 
					\frac{1}{2} |I_5|  & \le M_{\mathcal{A}}\Big( \|\bu^\star - \bu\|_{\bL^\infty} + \|\theta^\star - \theta\|_{L^\infty} \Big)\Big(  \|\bv\|_{\bL^2}^2 + \|\vartheta\|_{L^2}^2 \Big) + \|\bro - \bro^\star \|_{\bL^2\times L^2}^{1/2}\Big(  \|\bv\|_{\bL^2}^{5/2} + \|\vartheta\|_{L^2}^{5/2} \Big)\\
					& \quad + \|\bro - \bro^\star \|_{\bL^2\times L^2}^{1/2}\Big(  \beta_1\|\bu^\star - \bu\|_{\bW^{2,1}_{2,\sigma}} + \beta_2\|\theta^\star - \theta\|_{W^{2,1}_{2}} \Big)\Big(  \|\bv\|_{\bL^2}^{3/2} + \|\vartheta\|_{L^2}^{3/2} \Big).
				\end{aligned}
			\end{align*}
			Furthermore, \eqref{stcongap} and \eqref{estimate:estLsL1} imply that
			\begin{align*}
				\begin{aligned} 
					\frac{1}{2} |I_5|  & \le M_{\mathcal{A}}\Big( \|\bu^\star - \bu\|_{\bL^\infty} + \|\theta^\star - \theta\|_{L^\infty} \Big)\Big(  \|\bv\|_{\bL^2}^2 + \|\vartheta\|_{L^2}^2 \Big)\\
					& \quad + \|\bro - \bro^\star \|_{\bL^2\times L^2}^{1/2}\Big(  \|\bu^\star - \bu\|_{\bL^2}^{1/2} + \|\theta^\star - \theta\|_{L^2}^{1/2} \Big)\Big(  \|\bv\|_{\bL^2}^2 + \|\vartheta\|_{L^2}^2 \Big)\\
					& \quad + \|\bro - \bro^\star \|_{\bL^2\times L^2}^{1/2}\Big(  \beta_1\|\bu^\star - \bu\|_{\bW^{2,1}_{2,\sigma}} + \beta_2\|\theta^\star - \theta\|_{W^{2,1}_{2}} \Big)\Big(  \|\bv\|_{\bL^2}^{3/2} + \|\vartheta\|_{L^2}^{3/2} \Big)\\
					& \le M_{\mathcal{A}}\Big( \|\bu^\star - \bu\|_{\bL^\infty} + \|\theta^\star - \theta\|_{L^\infty} \Big)\Big(  \|\bv\|_{\bL^2}^2 + \|\vartheta\|_{L^2}^2 \Big)\\
					& \quad + \|\bro - \bro^\star \|_{\bL^2\times L^2}^{1/2}\Big(  \|\bv\|_{\bL^2}^{5/2} + \|\vartheta\|_{L^2}^{5/2} \Big) + \|\bro - \bro^\star \|_{\bL^2\times L^2}^{3/2}\Big(  \|\bv\|_{\bL^2}^{3/2} + \|\vartheta\|_{L^2}^{3/2} \Big)\\
					& \le c\Big( \| \bro^\star - \bro \|_{\bL^{1}\times L^{1}}^{\frac{3-\tilde{s}}{\tilde{s}'} + \tilde{s}} + \| \bro^\star - \bro \|_{\bL^{1}\times L^{1}}^{\frac{16\tilde{s}-10}{4\tilde{s}}} + \| \bro^\star - \bro \|_{\bL^{1}\times L^{1}}^{\frac{12\tilde{s}-6}{4\tilde{s}}} \Big).
				\end{aligned}
			\end{align*}
			
			To summarize, we have the following estimate for the difference between the second derivatives
			\begin{align*}
				&\Big\vert\Big[J''(\bro_t) - J''(\bro^\star)\Big](\bro-\bro^\star)^2\Big\vert \le  c\Big[\| \bro^\star - \bro \|_{\bL^{1}\times L^{1}}^{\frac{3-\tilde{s}}{\tilde{s}'} + \tilde{s} + 1} + \| \bro^\star - \bro \|_{\bL^{1}\times L^{1}}^{\frac{3-\tilde{s}}{\tilde{s}'} + \tilde{s}} \\
				&  + \| \bro^\star - \bro \|_{\bL^{1}\times L^{1}}^{\frac{16\tilde{s}-10}{4\tilde{s}}} + \| \bro^\star - \bro \|_{\bL^{1}\times L^{1}}^{\frac{12\tilde{s}-6}{4\tilde{s}}} + \Big(\beta_1\|\bv^\star(T)\|_{\bL^2}^2 + \beta_2\|\vartheta^\star(T)\|_{L^2}^2 \Big)\|\bro^\star - \bro\|_{\bL^{1}\times L^{1}}^{1/\tilde{s}'+1/2}\Big].
			\end{align*}
			Since $\tilde{s}\in (3/2,2)$ one finds $\ell_1,\ell_2,\ell_3,\ell_4,\ell_5 >0$ such that $\frac{3-\tilde{s}}{\tilde{s}'} + \tilde{s} = 2 + \ell_1 + \ell_2$, $\frac{3-\tilde{s}}{\tilde{s}'} + \tilde{s} + 1 = 2 + \ell_1 + \ell_3$, $\frac{16\tilde{s}-10}{4\tilde{s}} = 2 + \ell_1 + \ell_4$, $\frac{12\tilde{s}-6}{4\tilde{s}} = 2 + \ell_1 + \ell_5$. If we then choose $$\delta \le \min\{ (\varepsilon/4c)^{\frac{1}{\ell_2}},(\varepsilon/4c)^{\frac{1}{\ell_3}},(\varepsilon/4c)^{\frac{1}{\ell_4}},(\varepsilon/4c)^{\frac{1}{\ell_5}},(\varepsilon/c)^{2\tilde{s}'/(\tilde{s}'+2)}, \delta_1,\delta_2 \},$$
			we get 
			\begin{align*}
				&\Big \vert\Big[J''(\bro) - J''(\bro^\star)\Big](\bro-\bro^\star)^2 \Big\vert\le  \varepsilon\Big(\| \bro^\star - \bro \|_{\bL^{1}\times L^{1}}^{1 + \mu} + \beta_1\|\bv^\star(T)\|_{\bL^2}^2 + \beta_2\|\vartheta^\star(T)\|_{L^2}^2   \Big)
			\end{align*}
			whenever $\| \bro^\star - \bro \|_{\bL^{1}\times L^{1}} < \delta$, where $\mu = 1+\ell_1$.
		\end{proof}

		\begin{lemma}\label{secests}
			Let $\bro^\star\in \mathcal{U}$, $\beta_1 = \beta_2 = 0$ and $(\bu^\star,\theta^\star) = \mathcal{S}(\bro^\star)$. For every $\varepsilon>0 $ there exists a $\delta>0$ such that
			\begin{align}
				&\Big \vert \Big[ J''(\bro) - J''(\bro^\star)\Big](\bro-\bro^\star)^2\Big\vert\leq \varepsilon \Big ( \|\bu -  \bu^\star \|_{\bL^2}^2 +\|\theta - \theta^\star\|_{L^2}^2 \Big),
			\end{align}
			for all $\bro\in \mathcal{U}$ with $\|\bu -  \bu^\star \|_{\bL^\infty} +\|\theta - \theta^\star\|_{L^\infty} < \delta$, where $(\bu,\theta) = \mathcal{S}(\bro)$.
		\end{lemma}
		
		\begin{proof}
			Using similar notations and arguments as in the proof of \Cref{secestc}, we see that
			\begin{align*}
				\begin{aligned}
					&\Big \vert \Big[ J''(\bro) - J''(\bro^\star)\Big](\bro-\bro^\star)^2\Big\vert \le |I_1| + |I_2| + |I_5|\\
					&\le  c\Big(\|\bv\|_{\bL^2}^2 + \|\vartheta\|_{\bL^2}^2 \Big)\Big(\| \bu^\star - \bu\|_{\bL^\infty} + \| \theta^\star -\theta\|_{L^\infty} \Big)\\
					& \quad + M_{\mathcal{A}}\Big( \|\bu^\star - \bu\|_{\bL^\infty} + \|\theta^\star - \theta\|_{L^\infty} \Big)\Big(  \|\bv\|_{\bL^2}^2 + \|\vartheta\|_{L^2}^2 \Big)\\
					& \quad + \|\bro - \bro^\star \|_{\bL^2\times L^2}^{1/2}\Big(  \|\bu^\star - \bu\|_{\bL^2}^{1/2} + \|\theta^\star - \theta\|_{L^2}^{1/2} \Big)\Big(  \|\bv\|_{\bL^2}^2 + \|\vartheta\|_{L^2}^2 \Big)\\
					&\le  c\Big[\Big(\| \bu^\star -\bu\|_{\bL^\infty} + \| \theta^\star - \theta\|_{L^\infty} \Big) + \Big(\| \bu^\star - \bu\|_{\bL^\infty}^{1/2} + \| \theta^\star - \theta\|_{L^\infty}^{1/2} \Big) \Big]\Big(  \|\bv\|_{\bL^2}^2 + \|\vartheta\|_{L^2}^2 \Big),
				\end{aligned}
			\end{align*}
			where $c = \max\left\{c,M_{\mathcal{A}}, \sqrt{2M_{\mathcal{U}}} \right\}$. Choosing $\delta \le \min\{ \varepsilon/2c, (\varepsilon/2c)^2 \}$ and utilizing \Cref{lemma:supp2} and \Cref{lemma:supp3} gives us
			\begin{align*}
				&\Big \vert \Big[ J''(\bro) - J''(\bro^\star)\Big](\bro-\bro^\star)^2\Big\vert\leq \varepsilon \Big ( \|\bu-  \bu^\star \|_{\bL^2}^2 +\|\theta - \theta^\star\|_{L^2}^2 \Big)
			\end{align*}
			whenever $\|\bu-  \bu^\star \|_{\bL^\infty} +\|\theta - \theta^\star\|_{L^\infty} < \delta$.
		\end{proof}
		
		We consider two distinct growth conditions that guarantee strict local optimality.
		\begin{assumption}\label{growth1}
			Let $\mu\in [1,2)$, $\tau\in \{1/2,1\}$ and $\bro^\star\in \mathcal{U}$. There exist positive constants $c$ and  $\alpha$ such that
			\begin{equation}\label{growth:controls}
				J'(\bro^\star)(\bro-\bro^\star)+ \tau J''(\bro^\star)(\bro-\bro^\star)^2\geq c\Big (\|\bro-\bro^\star\|_{\bL^1\times L^1}^{1+\mu} + \beta_1\| \bv^\star(T)\|_{\bL^2}^2+\beta_2\| \vartheta^\star(T)\|_{L^2}^2 \Big)
			\end{equation}
			for all $\bro\in \mathcal{U}$ with $\|\bro- \bro^\star \|_{\bL^1\times L^1}<\alpha $, where $(\bv^\star,\vartheta^\star) = \mathcal{S}'(\bro^\star)(\bro - \bro^\star)$.
		\end{assumption}
		
		Assumption \ref{growth1} implies the optimal controls to be bang-bang, see \cite{dominguez2022}. 
		We also mention that this assumption was first considered in the stability analysis in ODE optimal control in \cite{OV} and in \cite{dominguez2022, jork2023} for semilinear elliptic and parabolic PDEs. 
		Furthermore, as proved in \cite[Proposition 2.12]{ADW2023}, bang-bang optimal controls for affine optimal control problems are shown to necessarily satisfy a growth condition of the same spirit as Assumption \ref{growth1}.
		
		The following condition which is weaker than Assumption \ref{growth1} allows us to consider optimal controls that are not of bang-bang structure. However, we can only provide error estimates for the optimal states under this assumption.
		\begin{assumption}\label{growth2}
			Let $\tau\in \{1/2,1\}$, $\bro^\star\in \mathcal{U}$ and $(\bu^\star,\theta^\star) = \mathcal{S}(\bro^\star)$. There exist positive constants $c$ and $\alpha$ such that
			\begin{equation}
				J'(\bro^\star)(\bro-\bro^\star)+ \tau J''(\bro^\star)(\bro-\bro^\star)^2\geq c\Big ( \|\bu - \bu^\star \|_{\bL^2}^2 +\|\theta - \theta^\star\|_{L^2}^2  \Big),
			\end{equation}
			for all $\bro\in \mathcal{U}$ with $\|\bu -\bu^\star \|_{\bL^\infty} +\|\theta - \theta^\star\|_{L^\infty} <\alpha$, where $(\bu,\theta) = \mathcal{S}(\bro)$.
		\end{assumption}
		Assumptions \ref{growth1} and \ref{growth2} imply strict weak or strong local optimality.
		
		\begin{theorem}\label{theorem:sufficientoptimality}
			Let $\bro^\star\in\mathcal{U}$ satisfy Assumption \ref{growth1} with $\tau = 1/2$, then there exist $\delta,\sigma > 0$ such that
			\begin{equation*}
				J(\bro)-J(\bro^\star)\geq \sigma\Big ( \|\bro- \bro^\star \|_{\bL^1}^{1+\mu} + \beta_1\| \bv^\star(T)\|_{\bL^2}^2 + \beta_2\| \vartheta^\star(T)\|_{L^2}^2 \Big)
			\end{equation*}
			for all $\bro\in \mathcal{U}$ with $\|\bro- \bro^\star \|_{\bL^1\times L^1}<\delta $, where $(\bv^\star,\vartheta^\star) = \mathcal{S}'(\bro^\star)(\bro - \bro^\star)$.
			Suppose that $\bro^\star\in\mathcal{U}$ and $(\bu^\star,\theta^\star) = \mathcal{S}(\bro^\star)$ satisfy Assumption \ref{growth2} and $\beta_1 = \beta_2 = 0$, one finds  $\delta,\sigma > 0$ such that
			\begin{equation*}
				J(\bro)-J(\bro^\star)\geq \sigma\Big ( \|\bu -  \bu^\star\|_{\bL^2}^2 +\|\theta - \theta^\star\|_{L^2}^2 \Big),
			\end{equation*}
			for all $\bro\in \mathcal{U}$ with $\|\bu -  \bu^\star\|_{\bL^\infty} +\|\theta - \theta^\star\|_{L^\infty} <\delta$, where $(\bu,\theta) = \mathcal{S}(\bro)$.
		\end{theorem}
		
		\begin{proof}
			By Taylor's theorem, there exists an $t\in(0,1)$ such that
			\begin{align*}
				&J(\bro)-J(\bro^\star)=J'(\bro^\star)(\bro-\bro^\star)+\frac{1}{2}J''(\bro_t)(\bro-\bro^\star)^2\\
				&\geq J'(\bro^\star)(\bro-\bro^\star)+\frac{1}{2}J''(\bro^\star)(\bro-\bro^\star)^2 + \frac{1}{2t^2}\Big( \Big[ J''(\bro_t)-J''(\bro^\star)\Big](\bro_t-\bro^\star)^2\Big).
			\end{align*}
			where $\bro_t = \bro^\star + t(\bro - \bro^\star)$.
			
			\noindent{\it Part 1.} From \Cref{growth1} and \Cref{secestc} we get
			\begin{align*}
				J(\bro)-J(\bro^\star) &\geq c \Big ( \|\bro-\bro^\star\|_{\bL^1\times L^1}^{1+\mu} + \beta_1\| \bv^\star(T)\|_{\bL^2}^2+\beta_2\| \vartheta^\star(T)\|_{L^2}^2 \Big)\\
				&\quad - \varepsilon \Big ( t\|\bro -\bro^\star\|_{\bL^1\times L^1}^{1+\mu} + \beta_1\| \bv_t^\star(T)\|_{\bL^2}^2 + \beta_2\| \vartheta_t^\star(T)\|_{L^2}^2 \Big).
			\end{align*}
			whenever $ \|\bro -\bro^\star\|_{\bL^1\times L^1} < \delta := \min\{\alpha,\delta_1 \}$, where $\delta_1>0$ is as in \Cref{secestc}, and $(\bv_t^\star,\vartheta_t^\star) = \mathcal{S}'(\bro^\star)(\bro_t-\bro^\star)$. By uniqueness of solution to the linearized Boussinesq system and since $\delta \bro_t = t\delta\bro$, we see that $\| \bv_t^\star(T)\|_{\bL^2}^2 + \| \vartheta_t^\star(T)\|_{L^2}^2 = t\| \bv^\star(T)\|_{\bL^2}^2 + t\| \vartheta^\star(T)\|_{L^2}^2  $. Thus, because $t\in[0,1]$ and by choosing $\varepsilon = c/2$, we have 
			\begin{align*}
				J(\bro)-J(\bro^\star) &\geq \sigma \Big ( \|\bro-\bro^\star\|_{\bL^1\times L^1}^{1+\mu} + \beta_1\| \bv^\star(T)\|_{\bL^2}^2+\beta_2\| \vartheta^\star(T)\|_{L^2}^2 \Big).
			\end{align*}
			
			\noindent{\it Part 2.} Using \Cref{growth2} and \Cref{secests} imply
			\begin{align*}
				J(\bro)-J(\bro^\star) &\geq c\Big ( \|\bu -   \bu^\star \|_{\bL^2}^2 +\|\theta - \theta^\star\|_{L^2}^2  \Big) - \varepsilon \Big ( \|\bu_t -  \bu^\star \|_{\bL^2}^2 +\|\theta_t - \theta^\star\|_{L^2}^2  \Big).
			\end{align*}
			From \Cref{lemma:supp2} and because $\|\bro_t-  \bro^\star \|_{\bL^1\times L^1} \le \|\bro -  \bro^\star \|_{\bL^1\times L^1}$, we get that whenever $\|\bro-  \bro^\star \|_{\bL^1\times L^1} < \delta_1$
			\begin{align*}
				J(\bro)-J(\bro^\star) &\geq c\Big ( \|\bu -  \bu^\star\|_{\bL^2}^2 +\|\theta - \theta^\star\|_{L^2}^2  \Big) - \frac{2\varepsilon}{3} \Big ( \| \bv_t^\star\|_{\bL^2}^2 + \| \vartheta_t^\star\|_{L^2}^2  \Big).
			\end{align*}
			From uniqueness of solution to \eqref{Blg1}--\eqref{Blg5} we get $\| \bv_t^\star\|_{\bL^2}^2 + \| \vartheta_t^\star\|_{L^2}^2 \le \| \bv^\star\|_{\bL^2}^2 + \| \vartheta^\star\|_{L^2}^2$. Thus, from applying \Cref{lemma:supp2} again we get taking $\sigma  = c/3$
			\begin{align*}
				J(\bro)-J(\bro^\star) &\geq \sigma\Big ( \|\bu -  \bu^\star \|_{\bL^2}^2 +\|\theta - \theta^\star\|_{L^2}^2  \Big)
			\end{align*}
			whenever $\|\bu -  \bu^\star\|_{\bL^\infty} +\|\theta - \theta^\star\|_{L^\infty} <\delta$ where $\delta = \min\left\{ \alpha, (2M_{\mathcal{U}})^{(s-1)/s}\delta_1^{1/s}\right\}$.
		\end{proof}

		Let us discuss the recovery of the assumptions above. It turns out that for the type of optimal control problem -- with $\beta_1 = \beta_2 = 0$ -- considered in this paper, given that the data is tracked well enough, the growth condition in \Cref{growth2} appears to be natural.
		
		\begin{theorem}\label{suffgro2}
			Let $\bro^\star\in\mathcal{U}$ be a local minimizer of \eqref{optcon}. Then there exists $\delta>0$, such that if $(\bu^\star, \theta^\star) = \mathcal{S}(\bro^{\star})$ satisfies
			\begin{align}\label{trackingclose}
				\|\bu^\star - \bu_d\|_{\bL^s(Q)} + \|\theta^\star -\theta_d \|_{L^s(Q)} < \frac{\min\{\alpha_1,\alpha_2 \}}{\delta} \quad\text{for }s\ge 4,
			\end{align}
			then \Cref{growth2} is satisfied.
		\end{theorem}
		\begin{proof}
			Since $\bro^\star\in\mathcal{U}$ is a local minimizer, it holds that $J'(\bro^\star)(\bro-\bro^\star)\geq 0$ for all $\bro \in \mathcal U$. Further, we can estimate the second variation from below by
			\begin{align*}
				J''(\bro^\star)(\bro-\bro^\star)^2&= \alpha_1 \|\bv^\star \|_{\bL^2}^2 + \alpha_2\|\vartheta^\star \|_{L^2}^2 + \beta_1 \|\bv^\star(T) \|_{\bL^2}^2 + \beta_2\|\vartheta^\star(T) \|_{L^2}^2\\
				&\quad - 2( (\bv^\star \cdot\nabla)\bv^\star, \bw^\star )_Q - 2( \bv^\star\cdot\nabla \vartheta^\star, \Psi^\star)_Q\\
				&\ge \alpha_1 \|\bv^\star\|_{\bL^2}^2 + \alpha_2\|\vartheta^\star\|_{L^2}^2 - 2c(\|\bv^\star \|_{\bL^2}^2 + \|\vartheta^\star\|_{L^2}^2)(\|\nabla \bw^\star\|_{\bL^\infty} + \|\nabla \Psi^\star\|_{L^\infty})\\
				&\ge \{\min\{\alpha_1,\alpha_2\} - \delta( \|\bu^\star - \bu_d\|_{\bL^s(Q)} + \|\theta^\star -\theta_d \|_{L^s(Q)})\}(\|\bv^\star \|_{\bL^2}^2 + \|\vartheta^\star\|_{L^2}^2).
			\end{align*}
			Now it is clear that due to the estimate \eqref{gradadest} Assumption \ref{growth2} holds.
		\end{proof}
		If we invoke the so-called structural assumption on the adjoint state, as it is often done for affine optimal control problems without a Tikhonov term, we can guarantee the growth of the first order derivative of $J$ with the same spirit as in Assumption \ref{growth1}.
		\begin{proposition}\label{prop:structasu}
			Let there exist positive constants $c$ and $c_i$, $i\in \{1,2\}$ such that the adjoint states $(\bw^\star, \Psi^\star) = \mathcal{D}(\bro^\star) $ corresponding to the optimal control $\bro^\star\in\mathcal{U}$ satisfy for some $\mu\in [1,2)$, and all $\varepsilon>0$
			\begin{equation}\label{measurecond1}
				\vert \{(t,x) \in \ Q : \vert\bw^\star_{i} \vert\leq \varepsilon \}\vert \leq c_i \varepsilon^\mu \text{ and } \vert \{(t,x) \in \ Q : \vert \Psi^\star \vert \leq \varepsilon \}\vert \leq c \varepsilon^\mu,\text{ where }\bw^\star = (\bw^\star_1,\bw^\star_2).
			\end{equation}
			Then there exists $\widetilde{c}>0$ such that $J'(\bro^\star)(\bro-\bro^\star) \ge \widetilde{c} \|\bro-\bro^\star\|_{\bL^1\times L^1}^{1+\mu}$.
		\end{proposition}
		
		\begin{proof}
			Suppose that \eqref{measurecond1} holds, and let us  take positive constants $\kappa_1$ and $\kappa_2$ such that
			\begin{equation*}
				\kappa_1<\Big(\frac{1}{c_1\| \overline{\bq}-\underline{\bq}\|_{L^\infty}}\Big)^{1/\mu} \text{ and }  \kappa_2<\Big(\frac{1}{c_2\| \overline{\bq}-\underline{\bq}\|_{L^\infty}}\Big)^{1/\mu}.
			\end{equation*}
			For brevity, let us define $A_i:=\{(t,x) \in \ Q : \vert \bw^\star_i \vert > \varepsilon \}$. Now we take $\varepsilon_i:=\kappa_i \|\bq_i-\bq^\star_i \|_{L^1}^{\mu}$, where the subscripts stand for the components of the vectors $\bq$ and $\bq^\star$,
			\begin{align}
				&\int_Q  \bw^\star\cdot (\bq-\bq^\star) \text{ d} x  \text{d}t=\int_Q  \vert \bw^\star_{1}\vert \vert \bq_1-\bq_1^\star\vert + \vert \bw^\star_{2}\vert \vert\bq_2-\bq_2^\star\vert \text{ d} x  \text{d}t\nonumber\\
				&\geq \int_{ A_1} \varepsilon  \vert\bq_1-\bq_1^\star \vert \text{ d} x  \text{d}t + \int_{A_2}\varepsilon \vert\bq_2-\bq_2^\star\vert \text{ d} x  \text{d}t\nonumber\\
				&\geq \varepsilon_1\int_Q  \vert \bq_1-\bq_1^\star\vert \text{ d} x  \text{d}t- \varepsilon_1 \| \overline{\bq}-\underline{\bq}\|_{\bL^\infty} \vert Q\setminus A_1\vert+\varepsilon_2\int_Q  \vert \bq_2-\bq_2^\star\vert \text{ d} x  \text{d}t- \varepsilon_2\| \overline{\bq}-\underline{\bq}\|_{\bL^\infty} \vert Q\setminus A_2\vert\nonumber\\
				&\geq \hat c\Big( \|\bq_1-\bq_1^\star \|_{L^1}^{1+\mu}+\|\bq_2-\bq_2^\star \|_{L^1}^{1+\mu}\Big)\geq \tilde c\|\bro-\bro^\star\|_{\bL^1\times L^1}^{1+\mu}.\nonumber
			\end{align}
			The proof of the claim for $\Psi^\star_{\bro^\star}$ follows the same arguments.
		\end{proof}
		
		\begin{remark}
			The implication of the proposition above is that one can relax the assumption on the second-order derivative of the objective functional to obtain a sufficient optimality condition. It allows us to have the second variation of the objective functional attain negative values, i.e., we can assume $$J''(\bro^\star)(\bro-\bro^\star)^2 \ge -{c} \|\bro-\bro^\star\|_{\bL^1\times L^1}^{1+\mu}$$ where $c < \widetilde{c}$, which would then imply Assumption \ref{growth1} and thus the optimality of $\bro^\star\in\mathcal{U}$.
		\end{remark}
		
		As it would, in most cases, be hard to verify if \eqref{measurecond1} indeed holds, the next result gives us a property from which we can recover \eqref{measurecond1}. The proof is based on the proof of \cite[Lemma 3.2]{DH2012}.
		\begin{proposition}\label{prop:enoughreggrowth}
			Let $\bro\in\mathcal{U}$ and $s>4$. Suppose that for any $t\in [0,T]$ the adjoint $(\bw,\psi) = \mathcal{D}(\bro) \in \bW^{2,1}_{s,\sigma}\times W^{2,1}_s$ satisfy
			\begin{equation}\label{eq:bound1}
				\begin{aligned}
					&\min\left\{ \min_{K_\bw} \vert \nabla \bw(x,t) \vert , \min_{K_\psi} \vert \nabla \psi(x,t) \vert \right\} > 0,
				\end{aligned}
			\end{equation}
			where $K_{\bw}:=\{x\in \Omega : \bw(x,\cdot) =0\}$ and $ K_{\psi}:=\{x\in \Omega\vert \ \psi(x,\cdot)=0\}$.
			Then $\bro\in\mathcal{U}$ satisfies \eqref{measurecond1} for $\mu=1$.
		\end{proposition}
		\begin{proof}
			Let us define for $|\tau|\le \varepsilon$ the sets $K_{\bw}(\tau):=\{x\in \Omega : |\bw(x,\cdot)| = \tau\}$ and $ K_{\psi}(\tau):=\{x\in \Omega : \vert \psi(x,\cdot)\vert= \tau\}$. From \eqref{eq:bound1} and since $\bW^{2,1}_{s,\sigma}\times W^{2,1}_s \hookrightarrow C(\overline{I};C^1(\overline{\Omega})^2)\times C(\overline{I};C^1(\overline{\Omega}))$ we infer the existence of constants $c_1,c_2>0$ such that 
			$$|\nabla \bw| > c_1 \text{ on }K_{\bw}(s) \quad \text{and }\quad \mathcal{H}(K_{\bw}(\tau)) \le c_2,$$
			where $\mathcal{H}$ denotes the Hausdorff measure. By virtue of the co-area formula, we find for $i\in \{1,2\}$
			\begin{align*}
				\vert \{(x,t)\in Q : \vert \bw_i(x,t)\vert\leq \varepsilon\} \vert &\leq \frac{1}{c_1}\int_0^T \int_{\{x\in \Omega : \vert \bw_i(x,t)\vert\leq \varepsilon\}} \vert \nabla \bw \vert \du x\du t\le \frac{1}{c_1}  \int_0^T\int_{-\varepsilon}^\varepsilon \mathcal{H}(K_{\bw}(\tau)) \du \tau \du t \leq   \frac{2c_2T}{c_1}\varepsilon.
			\end{align*}
			Using the same arguments, we obtain the claim for the adjoint $\psi$.
		\end{proof}

		\section{Solution stability}\label{sec5}
		We study the stability properties of optimal controls and states under perturbation in the Boussinesq system and the objective functional.
		To prove the solution stability of the optimal control problem at local optimal controls satisfying the growth Assumptions \ref{growth1} or \ref{growth2}, let us introduce some terminology. The normal cone to $\mathcal{U}$ at $\widehat{\bro}$ is denoted and defined as
		\begin{equation*}
			\mathcal{N}_{\mathcal{U}}(\widehat{ \bro}):=\Bigg\{ (\bvsigma,\Lambda)\in L^1(I;\bL^1(\omega_q)\times L^1(\omega_h)) : \ \int_0^T\!\!\int_{\omega_q} \bvsigma \cdot (\bq -\widehat \bq) \du x + \int_{\omega_h} \Lambda (\Theta -\widehat \Theta) \text{ d}x \text{d}t \leq 0 \ \forall \bro \in \mathcal{U} \Bigg \}.
		\end{equation*}
		Let us fix an initial data $(\bu_0, \theta_0)$. Utilizing the notion of the normal cone and the mappings \eqref{themap}, \eqref{themapadj}, we can now define the set-valued optimality mapping.
		
		\begin{align*}
			&\digamma (\bu,\theta,\bw,\Psi,\bq,\Theta) =  \left[ 
			\begin{aligned}
				\mathcal{F}(\bu,\theta,\bu_0,\theta_0,\bq,\Theta)\\
				\mathcal{G}(\bu,\theta,\bw,\Psi)\\
				(\bw, \Psi)+\mathcal{N}_{\mathcal{U}}(\bro)
			\end{aligned}
			\right].
		\end{align*}

		If $\bro^\star\in\mathcal{U}$ is a local optimal control of \eqref{optcon} it holds -- according to \Cref{firstoderoptcon} -- that
		\begin{equation}
			0\in \digamma (\bu^\star, \theta^\star, \bw^\star, \Psi^\star, \bro^\star),
		\end{equation}
		where $(\bu^\star,\theta^\star) = \mathcal{S}(\bro^\star) $ and $(\bw^\star,\Psi^\star) = \mathcal{D}(\bro^\star)$. 
		We also use the perturbed control-to-state operator $\widehat{\mathcal{S}}:L^s(I;\bL^s(\omega_q)\times L^s(\omega_h)) \to \bW^{2,1}_{s,\sigma}\times W^{2,1}_s $ defined as $$\widehat{\mathcal{S}}(\bro) = \widetilde{\mathcal{S}}(\blf + \widehat{\blf},h + \widehat{h},\bu_0+\widehat{\bu}_0,\theta_0+\widehat{\theta}_0,\bro).$$
		Using similar arguments as in the previous section, we get the differentiability of the perturbed control-to-state operator. 
		\begin{proposition}
			The perturbed control-to-state operator is of class $C^\infty$. The first-order Fr\'{e}chet derivative at $\bro\in L^s(I;\bL^s(\omega_q)\times L^s(\omega_h))$ in direction $\delta\bro\in L^s(I;\bL^s(\omega_q)\times L^s(\omega_h))$, denoted by 
			\[
			\widehat{\mathcal{S}}'(\bro)(\delta\bro)=:(\widehat{\bv}, \widehat{\vartheta}),
			\]
			is given as the solution to the linearized Boussinesq system \eqref{Blg1}--\eqref{Blg5} with $\bu_1 = \bu_2 = \widehat{S}_1(\bro)$ and $\theta = \widehat{S}_2(\bro)$, $\bF=\delta\bq\chi_{\omega_q}$, $G=\delta\Theta\chi_{\omega_h}$, $\bv_0=0$ and $\vartheta_0=0$.
			The second-order Fr\'{e}chet derivative at $\bro\in L^s(I;\bL^s(\omega_q)\times L^s(\omega_h))$ in direction $(\delta\bro_1,\delta \bro_2)\in L^s(I;\bL^s(\omega_q)\times L^s(\omega_h))^2$, denoted by 
			\begin{equation}
				\widehat{\mathcal{S}}''(\bro)[\delta\bro_1,\delta\bro_2]:=\big(\widetilde{\bv},\widetilde{\vartheta}\big),
			\end{equation}
			is the solution to the linearized Boussinesq system \eqref{Blg1}--\eqref{Blg5} with $\bu_1 = \bu_2 = \widehat{S}_1(\bro)$ and $\theta = \widehat{S}_2(\bro)$, $\bF=- [(\widehat{\bv}_1\cdot \nabla) \widehat{\bv}_2+(\widehat{\bv}_2\cdot \nabla) \widehat{\bv}_1]$, $G=-\left[\widehat{\bv}_1\cdot \nabla \widehat{\vartheta}_2 + \widehat{\bv}_2\cdot \nabla \widehat{\vartheta}_1\right]$, $\bv(0,\cdot)=0$ and $\vartheta(0,\cdot)=0$, where $(\widehat{\bv}_1,\widehat{\vartheta}_1) = \widehat{\mathcal{S}}'(\bro)(\delta\bro_1)$ and $(\widehat{\bv}_2,\widehat{\vartheta}_2) = \widehat{\mathcal{S}}'(\bro)(\delta\bro_2)$.
		\end{proposition}
		
		Analogously, we see that the adjoint variable $(\bw,\Psi) = \widehat{\mathcal{S}}'(\bro)^*( \delta\bro + g_T(\bw_T,\Psi_T) )$ solves \eqref{Badg1}--\eqref{Badg5} with $\bu_1 = \bu_2 = \widehat{S}_1(\bro)$ and $\theta = \widehat{S}_2(\bro)$, $\bF = \delta{\bq}$, $G = \delta \Theta$, $\bw(T) = \bw_T$ and $\Psi(T) = \Psi_T$. We also define the adjoint operator $$\widehat{\mathcal{D}}(\bro) = \widehat{\mathcal{S}}'(\bro)^*( (\alpha_1(\widehat{S}_1(\bro) - \bu_d) + \bta, \alpha_2(\widehat{S}_2(\bro) - \theta_d) + \eta) + g_T(\beta_1(\widehat{S}_1(\bro)(T) - \bu_T  ) , \beta_2(\widehat{S}_2(\bro)(T) - \theta_T  ) ) ).$$
		
		As we mentioned in the introduction, the goal is to study the stability of the optimal control problem with respect to several perturbations. For simplicity we denote by $$\mathcal{P} \subset \bL^s(Q)\times L^s(Q)\times \bW^{2-2/s,s}_{0,\sigma}\times W^{2-2/s,s}_0 \times \bL^\infty(Q)\times L^\infty(Q)\times L^\infty(I;\bL^\infty(\omega_q)\times L^\infty(\omega_h))$$
		the space of admissible perturbations which is a Banach space under the norm defined as the sum of the norms of its arguments. 
		
		For the set of feasible perturbations 
		$\mathcal{P}$, we assume that there exists a constant $M_{\mathcal{P}}>0$ such that $$\|(\widehat{\blf},\widehat{h},\widehat{\bu}_0,\widehat{\theta}_0,\bta,\eta,\bvsigma,\Lambda)\|_{\mathcal{P}}\le M_{\mathcal{P}} \text{ for all } (\widehat{\blf},\widehat{h},\widehat{\bu}_0,\widehat{\theta}_0,\bta,\eta,\bvsigma,\Lambda)\in \mathcal{P}.$$
		To lighten the notation,  we write $\bzeta = (\widehat{\blf},\widehat{h},\widehat{\bu}_0,\widehat{\theta}_0,\bta,\eta,0,0,-\bvsigma,-\Lambda)$. Now we can formulate one of our main tasks of the paper, the study of the solution stability of the optimal controls and states under perturbations in the optimality map, that is
		\begin{equation}\label{inclusion:1}
			\bzeta \in \digamma (\bu,\theta,\bw,\Psi,\bq,\Theta).
		\end{equation}
		In fact, $\bzeta \in \digamma (\bu,\theta,\bw,\Psi,\bq,\Theta)$ implies $-(\bvsigma,\Lambda) \in \widehat{\mathcal{D}}(\bro) + \mathcal{N}_{\mathcal{U}}(\bro).$
		
		The following lemmata lays the foundation for us to be able to prove the stability results under Assumptions \ref{growth1} and \ref{growth2}.
		
		\begin{lemma}\label{estpersta}
			Let $\bro\in L^s(I;\bL^s(\omega_q)\times L^s(\omega_h))$ and $(\bu,\theta) ={\mathcal{S}}(\bro), (\widehat{\bu}, \widehat{\theta})  = \widehat{\mathcal{S}}(\bro)\in \bW^{2,1}_{s,\sigma}\times W^{2,1}_{s}$.
			Then it holds that
			\begin{align}
				&\| \bu - \widehat{\bu}\|_{\bW^{2,1}_{s,\sigma}}+ \| \theta - \widehat{\theta}\|_{W^{2,1}_s}\leq c\Big(\|\widehat{\bu}_0\|_{\bW^{2-2/s,s}_{0,\sigma}} + \|\widehat{\theta}_0\|_{W^{2-2/s,s}_{0}} + \|\widehat{\blf}\|_{\bL^s} + \|\widehat{h}\|_{L^s}\Big).
			\end{align}
		\end{lemma}
		\begin{proof}
			The functions $\bv:=\bu - \widehat{\bu}$ and $\vartheta:=\theta - \widehat{\theta}$ are a solution to \eqref{Blg1}-\eqref{Blg5} with $\bu_1=\widehat{\bu}$, $\bu_2=\bu$, $\theta=\theta$, $\bF=\widehat{\blf}$, $G=\widehat{\theta}$, $\bv_0=\widehat{\bu}_0$ and $\vartheta_0=\widehat{\theta}_0$.
			Then, applying Theorem \ref{theorem:linearLP} yields the claim.
		\end{proof}

		\begin{lemma}\label{lemma:pertlinsest}
			Let $\bro,\delta \bro\in L^s(I;\bL^s(\omega_q)\times L^s(\omega_h))$, $(\bu,\theta) = \mathcal{S}(\bro)$, $(\widehat{\bu},\widehat{\theta}) = \widehat{\mathcal{S}}(\bro)$, $(\bv,\vartheta) = \mathcal{S}'(\bro)\delta\bro$ and $(\widehat{\bv},\widehat{\vartheta}) = \widehat{\mathcal{S}}'(\bro)\delta\bro$.
			There exists a constant $c>0$ such that 
			\begin{align}\label{lineargap:pertononper}
				\begin{aligned}
					&\|\bv - \widehat{\bv}\|_{\bL^2} + \|\nabla(\bv - \widehat{\bv})\|_{\bL^2} + \|\vartheta - \widehat{\vartheta}\|_{L^2} + \|\nabla(\vartheta - \widehat{\vartheta})\|_{\bL^2} \\
					&\le c\big( \|\bu - \widehat{\bu}\|_{\bL^\infty} + \|\theta - \widehat{\theta}\|_{L^\infty} \big)\big( \|\bv\|_{\bL^2} + \|\vartheta\|_{L^2} \big)
				\end{aligned}
			\end{align}
		\end{lemma}
		\begin{proof}
			As the element $(\overline{\bv},\overline{\vartheta}) =( \bv - \widehat{\bv}, \vartheta - \widehat{\vartheta})\in \bW^{2,1}_{s,\sigma}\times W^{2,1}_s$ satisfies a system of the form \eqref{Blg1}--\eqref{Blg5} with $\bu_1 = \bu_2 = \widehat{\bu}$, $\theta = \widehat{\theta}$, $\bF = ( (\widehat{\bu} - {\bu})\cdot\nabla )\bv + (\bv\cdot\nabla)(\widehat{\bu} - {\bu})$, $G = (\widehat{\bu} - {\bu})\cdot\nabla\vartheta + \bv\cdot\nabla(\widehat{\theta}-{\theta})$, $\bv_0 = 0$, and $\vartheta_0 = 0$, \eqref{estimate:lweak} implies that
			\begin{align*}
				\|\bv - \widehat{\bv}\|_{\bL^2} + \|\nabla(\bv - \widehat{\bv})\|_{\bL^2} + \|\vartheta - \widehat{\vartheta}\|_{L^2} + \|\nabla(\vartheta - \widehat{\vartheta})\|_{\bL^2} \le c(\|\bF\|_{L^2(\bV^*_\sigma)} + \|G\|_{L^2(V^*)})
			\end{align*}
			Using similar arguments as in the proof of \Cref{lemma:supp1} gives us the estimate
			\begin{align*}
				\begin{aligned}
					&\|\bv - \widehat{\bv}\|_{\bL^2} + \|\nabla(\bv - \widehat{\bv})\|_{\bL^2} + \|\vartheta - \widehat{\vartheta}\|_{L^2} + \|\nabla(\vartheta - \widehat{\vartheta})\|_{\bL^2} \\
					&\le c\big( \|\bu - \widehat{\bu}\|_{\bL^\infty} + \|\theta - \widehat{\theta}\|_{L^\infty} \big)\big( \|\bv\|_{\bL^2} + \|\vartheta\|_{L^2} \big).
				\end{aligned}
			\end{align*}
			
		\end{proof}
		
		\begin{lemma}\label{lemma:adjperest}
			Let $s>2$, $\bro\in\mathcal{U}$, $(\bw,\Psi)\in \mathcal{D}(\bro)$ and $(\widehat{\bw},\widehat{\Psi})\in\widehat{\mathcal{D}}(\bro)$. There exists $c>0$ such that
			\begin{align}
				&\| \bw - \widehat{\bw}\|_{\bW^{2,1}_{2,\sigma}}+ \| \Psi - \widehat{\Psi} \|_{W^{2,1}_2}\nonumber\\
				&\leq c\Big(\|\widehat{\bu}_0\|_{W^{2-2/s,s}_{0,\sigma}}+\|\widehat{\theta}_0\|_{W^{2-2/s,s}_{0,\sigma}}+\|\widehat{\blf}\|_{\bL^s}+\|\widehat{h}\|_{L^s} + \| \bta\|_{\bL^\infty}+\|\eta\|_{L^\infty} \Big).
			\end{align}
		\end{lemma}
		\begin{proof}
			The element $(\overline{\bw},\overline{\Psi}) = (\bw -\widehat{\bw},\Psi - \widehat{\Psi})\in \bW^{2,1}_{2,\sigma}\times W^{2,1}_2$ solves a system of the form \eqref{Badg1}-\eqref{Badg5} with right hand side
			\begin{align*}
				& \bF = ( (\bu - \widehat{\bu})\cdot\nabla )\bw + (\nabla(\widehat{\bu} - \bu))^\top\bw + \Psi\nabla(\widehat{\theta} - \theta ) + \alpha_1(\bu - \widehat{\bu}) + \bta,\\
				& G = (\widehat{\bu} - \bu)\cdot\nabla\Psi + \alpha_2(\theta - \widehat{\theta}) + \eta,\\
				& \bw_T = \beta_1(\bu(T) - \widehat{\bu}(T)),\text{ and }\Psi_T =  \beta_2(\theta(T)-\widehat{\theta}(T)),
			\end{align*}
			where $(\bu,\theta) = \mathcal{S}(\bro)$ and $(\widehat{\bu},\widehat{\theta}) = \widehat{\mathcal{S}}(\bro)$.
			
			We obtain for any $\bphi\in \bL^{2'}(Q)$ with $\|\bphi\|_{\bL^{2'}} \le 1$ that
			\begin{align*}
				\begin{aligned}
					\|\bF\|_{\bL^2} & \le c(\|\nabla\bw\|_{\bL^\infty} + \alpha_1) \|\bu - \widehat{\bu}\|_{\bL^\infty} + \left|\left( (\bphi\cdot\nabla)(\bu - \widehat{\bu}),\bw \right)\right| + | (\bphi\cdot\nabla(\widehat{\theta} - \theta ), \Psi) |+ c\|\bta\|_{\bL^\infty} \\
					& \le c(\|\nabla\bw\|_{\bL^\infty} + \alpha_1) \|\bu - \widehat{\bu}\|_{\bL^\infty} + \left|\left( (\bphi\cdot\nabla)\bw,\bu - \widehat{\bu} \right)\right| + | (\bphi\cdot\nabla\Psi, \widehat{\theta} - \theta ) | + c\|\bta\|_{\bL^\infty}\\
					& \le c(\|\nabla\bw\|_{\bL^\infty} + \alpha_1) \|\bu - \widehat{\bu}\|_{\bL^\infty} + cM_{\mathcal{A}}\|\bphi\|_{\bL^{2'}}(\|\bu - \widehat{\bu}\|_{\bL^\infty} + \|\theta - \widehat{\theta}\|_{L^\infty}) + c\|\bta\|_{\bL^\infty}\\
					& \le c(\|\nabla\bw\|_{\bL^\infty} + \alpha_1) \|\bu - \widehat{\bu}\|_{\bL^\infty} + cM_{\mathcal{A}}\|\theta - \widehat{\theta}\|_{L^\infty} + c\|\bta\|_{\bL^\infty}
				\end{aligned}
			\end{align*}
			For $G$, we estimate 
			\begin{align*}
				\begin{aligned}
					\|G\|_{L^r} \le c\big(M_{\mathcal{A}}\|\bu - \widehat{\bu}\|_{\bL^\infty} + \alpha_2\|\theta - \widehat{\theta} \|_{L^\infty} + \|\eta\|_{\bL^\infty}\big)
				\end{aligned}
			\end{align*}
			Through the embedding $\bW^{2,1}_{s,\sigma}\times W^{2,1}_s \hookrightarrow C(\overline{I}; \bW^{2-1/2,2}_{0,\sigma}(\Omega)\times W^{2-1/2,2}_{0}(\Omega)) $  we get
			\begin{align*}
				\begin{aligned}
					\|\bw_T\|_{\bW^{2-1/2,2}_{0,\sigma}(\Omega)} +  \|\Psi_T\|_{W^{2-1/2,2}_{0}(\Omega)} \le \beta_1\|\bu - \widehat{\bu}\|_{\bW^{2,1}_{s,\sigma}} + \beta_2\|\theta - \widehat{\theta} \|_{W^{2,1}_s}.
				\end{aligned}
			\end{align*}
			From here on, it is clear that the claim follows due to Lemma \ref{estpersta}.
		\end{proof}
		
		\subsection{Stability under Assumption \ref{growth1}}\label{subsection4.17}
		As the title of this subsection suggests, in this portion of the paper we shall utilize Assumption \ref{growth1} to establish the stability of controls.
		As a preparation for the proof of the main statement, we consider perturbations on the normal cone.
		\begin{lemma}\label{perpcest}
			Let $\bro^\star\in \mathcal{U}$ satisfy Assumption \ref{growth1} with $\tau = 1$. Then there exist a positive constant $c$ and $\alpha$ such that
			\begin{equation}
				\| \bro-\bro^\star\|_{\bL^1\times L^1} \leq c (\|\bvsigma\|_{\bL^\infty(I\times\omega_q)}+ \|\Lambda\|_{L^\infty(I\times\omega_h)})^{1/\mu}.
			\end{equation}
			for all $(\bvsigma,\Lambda)\in \bL^\infty(I\times\omega_q)\times L^\infty(I\times\omega_h)$ and $\bro \in \mathcal U$ with $(\bvsigma,\Lambda)\in (\bw,\Psi)+\mathcal{N}_{\mathcal U}(\bro)$ and $\|\bro-\bro^\star\|_{\bL^1\times L^1}<\alpha$, where $(\bw,\Psi) = \mathcal{D}(\bro)$.
		\end{lemma}
		\begin{proof}
			From the definition of the normal cone, we see that
			\begin{align}\label{stab1}
				\begin{aligned}
					0&\geq   \int_0^T\left[ \int_{\omega_q} (\bw - \bvsigma)\cdot (\bq- \bq^\ast) \du x + \int_{\omega_h} (\Psi - \Lambda) (\Theta- \Theta^\star) \text{ d}x \right]\text{d}t\\
					&\geq J'(\bro^\star)(\bro-\bro^\star)+[J'(\bro)(\bro-\bro^\star)-J'(\bro^\star)(\bro-\bro^\star)]\\
					&\quad -\|\bvsigma\|_{\bL^\infty(I\times\omega_q)}\|\bq-\bq^\star\|_{\bL^1}-\|\Lambda\|_{L^\infty(I\times\omega_h)}\|\Theta-\Theta^\star\|_{L^1}
				\end{aligned}
			\end{align}
			By virtue of the mean value theorem there exists $t\in (0,1)$ such that $$J''(\bro_t)(\bro - \bro^\star)^2 = J'(\bro)(\bro-\bro^\star)-J'(\bro^\star)(\bro-\bro^\star),$$ 
			where $\bro_t = \bro^\star + t (\bro - \bro^\star)$. This, together with  \Cref{growth1} and \Cref{secestc}, give us further estimation for \eqref{stab1}:
			\begin{align}\label{stab2}
				\begin{aligned}
					0 &\geq J'(\bro^\star)(\bro-\bro^\star)+J''(\bro^\star)(\bro-\bro^\star)^2 + [J''(\bro_t)(\bro-\bro^\star)^2 - J''(\bro^\star)(\bro-\bro^\star)^2]\\
					&\quad -\|\bvsigma\|_{\bL^\infty(I\times\omega_q)}\|\bq-\bq^\star\|_{\bL^1}-\|\Lambda\|_{L^\infty(I\times\omega_h)}\|\Theta-\Theta^\star\|_{L^1}\\
					&\geq J'(\bro^\star)(\bro-\bro^\star)+J''(\bro^\star)(\bro-\bro^\star)^2 + \frac{1}{t^2}[J''(\bro_t) - J''(\bro^\star)](\bro_t-\bro^\star)^2\\
					&\quad -\|\bvsigma\|_{\bL^\infty(I\times\omega_q)}\|\bq-\bq^\star\|_{\bL^1}-\|\Lambda\|_{L^\infty(I\times\omega_h)}\|\Theta-\Theta^\star\|_{L^1}\\
					&\geq  c\Big (\|\bro-\bro^\star\|_{\bL^1\times L^1}^{1+\mu} + \beta_1\| \bv^\star(T)\|_{\bL^2}^2+\beta_2\| \vartheta^\star(T)\|_{L^2}^2 \Big) + [J''(\bro_t) - J''(\bro^\star)](\bro_t-\bro^\star)^2\\
					&\quad -\|\bvsigma\|_{\bL^\infty(I\times\omega_q)}\|\bq-\bq^\star\|_{\bL^1}-\|\Lambda\|_{L^\infty(I\times\omega_h)}\|\Theta-\Theta^\star\|_{L^1}\\
					&\geq  (c-t\varepsilon)\|\bro-\bro^\star\|_{\bL^1\times L^1}^{1+\mu} + (c-\varepsilon)\Big (\beta_1\| \bv^\star(T)\|_{\bL^2}^2+\beta_2\| \vartheta^\star(T)\|_{L^2}^2 \Big)\\
					&\quad -\|\bvsigma\|_{\bL^\infty(I\times\omega_q)}\|\bq-\bq^\star\|_{\bL^1}-\|\Lambda\|_{L^\infty(I\times\omega_h)}\|\Theta-\Theta^\star\|_{L^1},
				\end{aligned}
			\end{align}
			where $(\bv^\star,\vartheta^\star) = \mathcal{S}'(\bro^\star)(\bro-\bro^\star)$. We note that we get the last two lines from the fact that $t\in (0,1)$ and by assuming that $\|\bro-\bro^\star\|_{\bL^1\times L^1} < \min\{\delta,\alpha \}$, where $\delta,\alpha>0$ are the constants from \Cref{growth1} and \Cref{secestc}. We can, furthermore, choose $\varepsilon = c$ which gives us
			\begin{align}\label{npp1}
				\|\bvsigma\|_{\bL^\infty(I\times\omega_q)}\|\bq-\bq^\star\|_{\bL^1}+\|\Lambda\|_{L^\infty(I\times\omega_h)}\|\Theta-\Theta^\star\|_{\bL^1}\geq c\|\bro-\bro^\star\|_{\bL^1\times L^1}^{1+\mu}.
			\end{align}
			We employ the Peter-Paul inequality to estimate
			\begin{equation}\label{npp2}
				\|\bvsigma\|_{\bL^\infty(I\times\omega_q)}\|\bq-\bq^\star\|_{\bL^1}\leq  \frac{\|\bvsigma\|_{\bL^\infty(I\times\omega_q)}^{(\mu+1)/\mu}}{c^{1/\mu}(\mu+1)/\mu}+\frac{c\|\bq-\bq^\star\|_{\bL^1}^{\mu+1}}{(\mu+1)},
			\end{equation}
			and
			\begin{equation}\label{npp3}
				\|\Lambda\|_{L^\infty(I\times\omega_h)}\|\Theta-\Theta^\star\|_{L^1}\leq  \frac{\|\Lambda\|_{L^\infty(I\times\omega_h)}^{(\mu+1)/\mu}}{c^{1/\mu}(\mu+1)/\mu}+\frac{c\|\Theta-\Theta^\star\|_{L^1}^{\mu+1}}{(\mu+1)}.
			\end{equation}
			Utilizing \eqref{npp2} and \eqref{npp3} and since $\mu\in [1,2)$, \eqref{npp1} can further be estimated as
			\begin{equation*}
				\frac{\mu}{c^{1/\mu}(\mu+1)}\Big(\|\bvsigma\|_{\bL^\infty(I\times\omega_q)} + \|\Lambda\|_{L^\infty(I\times\omega_h)} \Big)^{(\mu+1)/\mu}\geq  \frac{c}{2}\|\bro-\bro^\star\|_{\bL^1\times L^1}^{1+\mu}.
			\end{equation*}
			From this, we infer the existence of a positive constant $\hat c$ such that
			\begin{equation*}
				(\|\bvsigma\|_{\bL^\infty(I\times\omega_q)}+\|\Lambda\|_{\bL^\infty(I\times\omega_h)})^{1/\mu}\geq \hat c\|\bro-\bro^\star\|_{\bL^1\times L^1}^{1+\mu}.
			\end{equation*}
		\end{proof}

		Now, we can prove the strong metric Hölder subregularity of the optimality map. 
		\begin{theorem}\label{theorem:smhsr}
			Let $\bro^\star \in \mathcal{U}$ satisfy Assumption \ref{growth1} with $\tau = 1$. Then the optimality map is strong metric Hölder subregular at $\bro^\star$ with constants $c, \alpha$ and $1/\gamma$, that is
			\begin{equation}\label{estimate:controls}
				\| \bro-\bro^\star\|_{\bL^1\times L^1}\leq c\| \bzeta \|_{\mathcal{P}}^{1/\mu}
			\end{equation}
			and 
			\begin{equation}\label{estimate:states}
				\| \bu-\bu^\star\|_{\bL^2}+\| \theta - \theta^\star\|_{L^2}\leq c\| \bzeta \|_{\mathcal{P}}^{1/2\mu}
			\end{equation}
			for all $\bzeta \in \digamma(\widehat{\mathcal{S}}(\bro),\widehat{\mathcal{D}}(\bro),\bro)$ whenever $\bro \in \mathcal{U}$ satisfies $\|\bro-\bro^\star\|_{\bL^1\times L^1}<\alpha$, where $(\bu,\theta)\in \mathcal{S}(\bro), (\bu^\star,\theta^\star) = \mathcal{S}(\bro^\star)$.
		\end{theorem}
		\begin{proof}
			By assumption, $-(\bvsigma,\Lambda) \in (\widehat{\bw},\widehat{\Psi}) + \mathcal{N}_{\mathcal{U}}(\bro)$, where $(\widehat{\bw},\widehat{\Psi}) = \widehat{D}(\bro)$, we find 
			\begin{align}\label{perfir}
				0\geq \int_0^T \left[\int_{\omega_q} (\widehat{\bw} + \bvsigma)\cdot({\bq} -\widetilde{\bq}) \text{ d}x + \int_{\omega_h} (\widehat{\Psi} 
				+ \Lambda)({\Theta} - \widetilde{\Theta}) \text{ d}x\right]\text{ d}t,
			\end{align}
			for arbitrary $\widetilde{\bro}\in \mathcal{U}$.
			We define
			\begin{equation*}
				\bB:= (\bw - \widehat{\bw} - \bvsigma), \text{  and  }   B:= (\Psi  - \widehat{\Psi} - \Lambda),
			\end{equation*}
			where $(\bw,\Psi) = \mathcal{D}(\bro)$.
			Thus \eqref{perfir} becomes
			\begin{equation*}
				0\geq \int_0^T\left[\int_{\omega_q} (\bw - \bB)\cdot({\bq} -\widetilde{\bq}) \text{ d}x+ \int_{\omega_h} (\Psi - B)\cdot({\Theta} - \widetilde{\Theta}) \text{ d}x\right]\text{ d}t.
			\end{equation*}
			This implies that $(\bB,B)\in (\bw,\Psi) + \mathcal{N}_{\mathcal{U}}(\bro)$. Now we apply Lemma \ref{perpcest} to conclude
			\begin{align*}
				\|\bro-\bro^\star\|_{\bL^1\times L^1}&\leq c (\|\bB\|_{\bL^\infty}+\|B\|_{L^\infty})^{1/\mu}.
			\end{align*}
			We estimate the right-hand side using Lemma \ref{lemma:adjperest} and the embedding $\bW^{2,1}_{s,\sigma}\times W^{2,1}_s \hookrightarrow \bL^\infty(Q)\times L^\infty(Q) $ and obtain
			\begin{align*}
				&\|\bB\|_{\bL^\infty}+\|B\|_{L^\infty}\leq \|\bvsigma\|_{\bL^\infty} + \|\Lambda\|_{L^\infty} + \|\widehat{\bw}-\bw\|_{\bL^\infty}+\|\widehat{\Psi} - \Psi\|_{L^\infty}\\
				&\leq \|\bvsigma\|_{\bL^\infty} + \|\Lambda\|_{L^\infty} + c\Big(\|\widehat{\bu}_0\|_{W^{2-2/s,s}_{0,\sigma}}+\|\widehat{\theta}_0\|_{W^{2-2/s,s}_{0,\sigma}}+\|\widehat{\blf}\|_{\bL^s}+\|\widehat{h}\|_{L^s} + \| \bta\|_{\bL^s} + \|\eta\|_{L^s}\Big).
			\end{align*}
			Thus, \eqref{estimate:controls} holds true. The estimate \eqref{estimate:states} follows from the estimate
			\begin{equation*}
				\| \bu - \widehat{\bu}\|_{\bL^2}+\| \theta - \widehat{\theta}\|_{L^2}\leq c\| \bro-\bro^\star \|_{\bL^1}^{1/2}.
			\end{equation*}
		\end{proof}

		\subsubsection{Solution with respect to linear perturbations}
		To demonstrate the meaning of \eqref{inclusion:1}, let us define the perturbed objective function
		\begin{align}
			&\begin{aligned}\label{pobjectivfunk}
				J_{\bzeta}(\bro):= \frac{\alpha_1}{2}\int_Q \vert \bu-\bu_{d}\vert^2\ \du x\du t + \frac{\alpha_2}{2}\int_Q \vert \theta-\theta_{d}\vert^2\ \du x\du t + \frac{\beta_1}{2}\int_\Omega \vert \bu(T)-\bu_{T}\vert^2 \ \text{d}x\\
				+ \frac{\beta_2}{2}\int_\Omega |\theta(T) - \theta_T|^2 \du x+\int_0^T \left[\int_{\omega_q}\bvsigma\cdot\, \bq\du x +\int_{\omega_h}\Lambda\Theta \du x\right] \du t+ \int_Q \bta\cdot\bu+ \eta\,\theta \du x\du t,
			\end{aligned}
		\end{align}
		while we also consider the perturbed state equations
		\begin{align}
			\partial_t \bu - \nu\Delta\bu + (\bu\cdot\nabla)\bu + \nabla p & = \be_2\theta + \blf + \widehat{\blf} + \bq\chi_{\omega_q} &&\text{in }Q,\label{pB1}\\
			\dive\bu & = 0 &&\text{in }Q,\label{pB2}\\
			\partial_t\theta -\kappa\Delta \theta + \bu\cdot\nabla\theta & = h + \widehat{h} + \Theta\chi_{\omega_h} &&\text{in }Q,\label{pB3}\\
			\bu = 0, \quad \theta & = 0 && \text{on }\Sigma,\label{pB4}\\
			\bu(0,\cdot) = \bu_0 + \widehat{\bu}_0, \quad \theta(0,\cdot) & = \theta_0+ \widehat{\theta}_0 && \text{in }\Omega,\label{pB5}
		\end{align}
		The goal is to investigate the behavior of the solution of the perturbed optimal control problem
		\begin{align}\label{peropticon}\tag{P$_{\bzeta}$}
			\min_{\bro\in \mathcal{U}} J_{\bzeta}(\bro) \text{ subject to }\eqref{pB1}-\eqref{pB5}
		\end{align}
		compared to the original problem \eqref{optcon}.
		
		We can easily find the derivatives of the perturbed objective functional just as in the case of the previous one. 
		\begin{theorem}
			The objective functional $J_{\bzeta}$ is of class $C^\infty$. Furthermore, the first and second variation can be calculated as stated below:
			\begin{align}\label{jderivative:per}
				\begin{aligned}
					&J_{\bzeta}'(\bro)(\delta\bro) =\alpha_1\int_Q(\widehat{\bu} -\bu_d)\cdot\widehat{\bv} \du x\du t+\alpha_2\int_Q ( \widehat{\theta} - \theta_d)\widehat{\vartheta} \du x\du t +\beta_1\int_\Omega (\widehat{\bu}(T)-\bu_T)\cdot\widehat{\bv}(T) \du x \\
					& + \beta_2\int_\Omega (\widehat{\theta}(T)-\theta_T)\widehat{\vartheta}(T)\du x+\int_0^T \left[\int_{\omega_q}\bvsigma\cdot\, \delta\bq\du x +\int_{\omega_h}\Lambda\delta\Theta \du x\right] \du t+ \int_Q \bta\cdot\widehat{\bv}+ \eta\,\widehat{\vartheta} \du x\du t\\
					&= \int_0^T\int_{\omega_q} (\widehat{\bw} + \bvsigma)\cdot\delta\bq \du x\du t + \int_0^T\int_{\omega_h} (\widehat{\Psi} + \Lambda)\delta\Theta \du x\du t
				\end{aligned}
			\end{align}
			\begin{align}
				\begin{aligned}
					J_{\bzeta}''(\bro)(\delta\bro, \delta\bro) &= \alpha_1 \|\widehat{\bv}\|_{\bL^2}^2 + \alpha_2\|\widehat{\vartheta}\|_{L^2}^2 + \beta_1 \|\widehat{\bv}(T) \|_{\bL^2}^2 + \beta_2\|\widehat{\vartheta}(T) \|_{L^2}^2\\
					&\quad - 2( (\widehat{\bv} \cdot\nabla)\widehat{\bv}, \widehat{\bw} )_Q - 2( \widehat{\bv}\cdot\nabla \widehat{\vartheta}, \widehat{\Psi})_Q.
				\end{aligned}
			\end{align}
			where $(\widehat{\bu},\widehat{\theta}):=\widehat{\mathcal{S}}(\bro)$, $(\widehat{\bv},\widehat{\vartheta}):=\widehat{\mathcal{S}}'(\bro)(\delta\bro)$, $(\widehat{\bw},\widehat{\Psi}):=  \widehat{\mathcal{D}}(\bro)$.
		\end{theorem}

		We can thus see that a minimizer $\bro_{\bzeta}\in\mathcal{U}$ of \eqref{pobjectivfunk} satisfies \eqref{inclusion:1}, 
		i.e., $\bzeta\in \digamma(\widehat{\mathcal{S}}(\bro_{\bzeta}),\widehat{\mathcal{D}}(\bro_{\bzeta}),\bro_{\bzeta})$. For simplicity, we also use the notation $\digamma(\bro) := \digamma(\widehat{\mathcal{S}}(\bro),\widehat{\mathcal{D}}(\bro),\bro)$
		The next theorem shows the convergence of solutions of the perturbed problem \eqref{peropticon} to a solution of \eqref{optcon}.

		\begin{theorem}
			Let $\{\bzeta_k\}\subset \mathcal{P}$ be a sequence satisfying $\|\bzeta_k\|_{\mathcal{P}}\to  0$ as $k\to \infty$. If $\bro_{\bzeta_k}\in\mathcal{U}$ is a global solution to \eqref{peropticon} for each $\bzeta_k$, then there exists $\bro^\star\in \mathcal{U}$ such that, up to a subsequence, $\bro_{\bzeta_k}\ws \bro^\star$ in $L^\infty(I;\bL^\infty(\omega_q)\times L^\infty(\omega_h))$. Furthermore, $\bro^\star\in\mathcal{U}$ is a global solution of \eqref{optcon} and $\widehat{\mathcal{S}}(\bro_{\bzeta_k})\to \mathcal{S}(\bro^\star)$ in $C(\overline{Q})^2\times C(\overline{Q})$.
			
		\end{theorem}
		\begin{proof}
			The claim follows from arguments similar to those in \cite[Theorem 4.2]{casas2022}.
		\end{proof}

		The following converse holds.
		\begin{theorem}
			Let $\bro^\star$ be a strict strong local solution to \eqref{optcon}. Then there exists a sequence of strong local minimizers $\bro_{\zeta_k}$ \eqref{peropticon} such that $\bro_{\bzeta_k}\ws \bro^\star$ in $L^\infty(I;\bL^\infty(\omega_q)\times L^\infty(\omega_h))$ and $\widehat{\mathcal{S}}(\bro_{\bzeta_k})\to \mathcal{S}(\bro^\star)$ in $C(\overline{Q})^2\times C(\overline{Q})$.
		\end{theorem}
		\begin{proof}
			The claim follows from arguments similar to those in \cite[Theorem 4.3]{casas2022}.
		\end{proof}
		
		As an application of the obtained regularity, we get the stability of the solution of \eqref{peropticon} with respect to linear perturbations on the optimal control problem \eqref{optcon}.
		

		\begin{theorem}\label{theorem:pertproblin}
			Let $\bro^\star\in \mathcal{U}$ be a weak local solution to \eqref{optcon}. Let Assumption \ref{growth1} be satisfied for a $\tau=1$. Then, there exist positive constants $c,\alpha $ such that
			\begin{align*}
				\begin{aligned}
					&\|\bq-\bq^\star\|_{L^1(I\times\omega_q)}+\|\Theta-\Theta^\star\|_{L^1(I\times\omega_h)}\\
					&\leq c \Big( \|\bvsigma\|_{\bL^\infty(I\times\omega_q)}+\|\Lambda\|_{L^\infty(I\times\omega_h)}+ \|\widehat{\bu}_0\|_{W^{2-2/s,s}_{0,\sigma}}+\|\widehat{\theta}_0\|_{W^{2-2/s,s}_{0,\sigma}}+\|\widehat{\blf}\|_{\bL^s}+\|\widehat{h}\|_{L^s} + \| \bta\|_{\bL^s} + \|\eta\|_{L^s} \Big)^{1/\mu}
				\end{aligned}
			\end{align*}
			for all local minimizers $\bro\in \mathcal{U}$ of \eqref{pobjectivfunk}  with $\|\bro-\bro^\star \|_{\bL^1\times L^1}<\alpha$.
		\end{theorem}
		\begin{proof}
			One can easily check that because $\bro\in\mathcal{U}$ is a local minimizer of \eqref{pobjectivfunk} then $\bzeta \in \digamma(\bro)$. We then directly apply \Cref{theorem:smhsr} to obtain the desired estimate.
		\end{proof}
		
		\begin{remark}
			In Theorem \ref{theorem:pertproblin}, stability of the optimal controls and states under linear perturbations was considered. The approach used also allows for nonlinear perturbations. To not further extend the presentation, we refer to \cite[Section 5]{dominguez2022} and \cite[Section 5]{jork2023} a detailed discussion on dealing with nonlinear perturbations.  As an example, we present the Tikhonov regularization in the subsection below.
		\end{remark}
		
		\subsubsection{Stability with respect to the Tikhonov regularization and the desired data}
		
		As a special case of the perturbed problem above, we may consider the Tikhonov regularization of the  problem with additional disturbances in the initial data and the data that is to be tracked; that is, we may consider
		\begin{align}\label{pobjectivfunkTik}
			\min_{\bro\in \mathcal{U}}J_{\mathcal{T}}(\bro)&:= \frac{\alpha_1}{2}\int_Q \vert \bu-(\bu_{d}+\widehat{\bu}_d)\vert^2\ \du x\du t + \frac{\alpha_2}{2}\int_Q \vert \theta-(\theta_{d}+\widehat{\theta}_d)\vert^2\ \du x\du t \\
			&+ \frac{\alpha_T}{2}\int_\Omega \vert \bu(T,\cdot)-\bu_{T}\vert^2 \ \text{d}x+\int_0^T\left[\frac{\varepsilon_1}{2}\int_{\omega_q} \vert\bq \vert^2 \du x+ \frac{\varepsilon_1}{2}\int_{\omega_h} \vert\Theta \vert^2 \du x\right]\du t\nonumber
		\end{align}
		subject to \eqref{box} and \eqref{pB1}-\eqref{pB5}. 
		
		We can easily find the derivatives of objective functional just as in the case of the previous one. 
		\begin{theorem}
			The objective functional $J_{\mathcal{T}}$ is of class $C^\infty$. Furthermore, the first and second variation can be calculated as stated below:
			\begin{align}\label{jderivative:per}
				\begin{aligned}
					J_{\mathcal{T}}'(\bro)(\delta\bro) =&\, \alpha_1\int_Q(\widehat{\bu} -\bu_{d} - \widehat{\bu}_d)\cdot\widehat{\bv} \du x\du t+\alpha_2\int_Q ( \widehat{\theta} - \theta_d - \widehat{\theta}_d)\widehat{\vartheta} \du x\du t +\beta_1\int_\Omega (\widehat{\bu}(T)-\bu_T)\cdot\widehat{\bv}(T) \du x \\
					& + \beta_2\int_\Omega (\widehat{\theta}(T)-\theta_T)\widehat{\vartheta}(T)\du x+\int_0^T \left[\varepsilon_1\int_{\omega_q}\bq\cdot\, \delta\bq\du x +\varepsilon_2\int_{\omega_h}\Theta\, \delta\Theta \du x\right] \du t\\
					=&\, \int_0^T\int_{\omega_q} (\widehat{\bw} + \varepsilon_1\bq)\cdot\delta\bq \du x\du t + \int_0^T\int_{\omega_h} (\widehat{\Psi} + \varepsilon_2\Theta)\delta\Theta \du x\du t,
				\end{aligned}
			\end{align}
			\begin{align}
				\begin{aligned}
					J_{\mathcal{T}}''(\bro)(\delta\bro, \delta\bro) &= \alpha_1 \|\widehat{\bv}\|_{\bL^2}^2 + \alpha_2\|\widehat{\vartheta}\|_{L^2}^2 + \beta_1 \|\widehat{\bv}(T) \|_{\bL^2}^2 + \beta_2\|\widehat{\vartheta}(T) \|_{L^2}^2  + \varepsilon_1\|\bq\|_{L^2(I;\bL^2(\omega_q))}^2 \\
					&\quad + \varepsilon_2\|\Theta\|_{L^2(I;L^2(\omega_h))}^2  - 2( (\widehat{\bv} \cdot\nabla)\widehat{\bv}, \widehat{\bw} )_Q - 2( \widehat{\bv}\cdot\nabla \widehat{\vartheta}, \widehat{\Psi})_Q.
				\end{aligned}
			\end{align}
			where $(\widehat{\bu},\widehat{\theta}):=\widehat{\mathcal{S}}(\bro)$, $(\widehat{\bv},\widehat{\vartheta}):=\widehat{\mathcal{S}}'(\bro)(\delta\bro)$, $(\widehat{\bw},\widehat{\Psi}):=  \widehat{\mathcal{D}}(\bro)$ with $\bta = -\alpha_1\widehat{\bu}_d$ and $\eta = -\alpha_2\widehat{\theta}_d$
		\end{theorem}

		Below, we obtain the estimate for the Tikhonov regularized problem.
		\begin{theorem}
			Let $\bro^\star\in \mathcal{U}$ be a weak local solution to \eqref{optcon}. Let Assumption \ref{growth1} be satisfied for $\tau = 1$. Then, there exist positive constants $c,\alpha$ such that
			\begin{align*}
				\begin{aligned}
					&\|\bq-\bq^\star\|_{L^1(I\times\omega_q)}+\|\Theta-\Theta^\star\|_{L^1(I\times\omega_h)}\\
					&\leq c \Big( \varepsilon_1 + \varepsilon_2 + \|\widehat{\bu}_0\|_{W^{2-2/s,s}_{0,\sigma}}+\|\widehat{\theta}_0\|_{W^{2-2/s,s}_{0,\sigma}}+\|\widehat{\blf}\|_{\bL^s}+\|\widehat{h}\|_{L^s} + \| \widehat{\bu}_d\|_{\bL^s} + \|\widehat{\theta}_d\|_{L^s} \Big)^{1/\mu},
				\end{aligned}
			\end{align*}
			for all local minimizers $\bro\in \mathcal{U}$ of \eqref{pobjectivfunkTik}  with $\|\bro-\bro^\star \|_{\bL^1\times L^1}<\alpha$.  
		\end{theorem}
		\begin{proof}
			Because $\bro\in\mathcal{U}$ is a weak local minimizer of \eqref{pobjectivfunkTik} we see that $$\bzeta = (\widehat{\blf},\widehat{h},\widehat{\bu}_0,\widehat{\theta}_0,-\alpha_1\widehat{\bu}_d,-\alpha_2\widehat{\theta}_d,0,0,-\varepsilon_1\bq,-\varepsilon_2\Theta)\in\digamma(\bro).$$
			A direct application of \Cref{theorem:smhsr} and \eqref{boundcon} will thus give us the desired estimate.
		\end{proof}
		\begin{remark}
			We mention that the underlying assumption to achieve this stability is that $\tau = 1$ in \Cref{growth1}. Nevertheless, we can also achieve such stability with $\tau = 1/2$, see the proof of \cite[Theorem 5.2]{DJNS2023}. We decided not to repeat the proof here since if the second variation is nonnegative both the assumptions are equivalent.
		\end{remark}
		\begin{remark}
			If both Assumption \ref{growth1} and Assumption \ref{growth2} are satisfied at the same time, we can extend the interval for feasible $\mu$ from $[1,2)$ to $[1, \infty)$. 
			See the arguments in \cite[Remark 6.10]{J2023}.
		\end{remark}

		\subsection{Stability under Assumption \ref{growth2}}\label{sec6}
		To obtain the error estimates for the optimal controls we used Assumption \ref{growth1} which deals with the growth of the objective functional with regards to the controls. Under Assumption \ref{growth2}, which employs a growth of the objective functional with respect to the states, we can still obtain stability estimates. However, Assumption \ref{growth2} allows only stability estimates of the states and adjoint states, and we can not utilize the normal cone formulation and Lemma \ref{perpcest}. Instead, we prove the estimates directly. In what follows, we assume $\beta_1 = \beta_2 = 0$, and use perturbations of the form $\bzeta_0 = (\widehat{\blf},\widehat{h},\widehat{\bu}_0,\widehat{\theta}_0,\bta,\eta,0,0,0,0)$ if not otherwise indicated.
		\begin{theorem}\label{thm:stabstate}
			Let $\bro^\star \in \mathcal{U}$ satisfy Assumption \ref{growth2} for $\tau=1$, and $(\bu^\star,\theta^\star) = \mathcal{S}(\bro^\star)$. Then there exists constants $c, \alpha$ such that
			\begin{equation}\label{statestab}
				\| \bu -\bu^\star \|_{\bL^2}+\| \theta -\theta^\star\|_{L^2} \leq c\| \bzeta_0 \|_{\mathcal{P}}
			\end{equation}
			for all $\bro \in \mathcal{U}$ with $\bzeta_0 \in \digamma (\bro)$ and $\| \bu -\bu^\star \|_{\bL^\infty}+\| \theta -\theta^\star\|_{L^\infty}<\alpha$, where $(\bu,\theta) = \mathcal{S}(\bro)$.
		\end{theorem}
		\begin{proof}
			Since $0 \in (\widehat{\bw},\widehat{\Psi}) +  \mathcal{N}_{\mathcal{U}}(\bro)$, where $(\widehat{\bw},\widehat{\Psi}) = \widehat{\mathcal{D}}(\bro)$, we see, according to \eqref{jderivative:per}, Assumption \ref{growth2} and the mean value theorem, that
			\begin{align*}
				0&\geq J_{\bzeta_0}'(\bro)(\bro-\bro^\star)  \geq J'(\bro^\star)(\bro-\bro^\star)+ J''(\bro^\star)(\bro-\bro^\star)^2-\Big\vert J_{\bzeta_0}'(\bro)(\bro-\bro^\star) - J'(\bro)(\bro-\bro^\star)\Big\vert\\
				&-\Big\vert  J'(\bro)(\bro-\bro^\star)- J'(\bro^\star)(\bro-\bro^\star)- J''(\bro^\star)(\bro-\bro^\star)^2\Big\vert\\
				&\geq c\big(\|\bu-  \bu^\star\|_{\bL^2}^2+\|\theta - \theta^\star\|_{L^2}^2\big)-\Big\vert J_{\bzeta_0}'(\bro)(\bro-\bro^\star)-J'(\bro)(\bro-\bro^\star)\Big\vert\\
				&-\Big\vert  J''(\bro_t)(\bro-\bro^\star)^2- J''(\bro^\star)(\bro-\bro^\star)^2\Big\vert,
			\end{align*}
			where $(\bu,\theta) = \mathcal{S}(\bro)$, $(\bu^\star, \theta^\star) = \mathcal{S}(\bro^\star)$ and $\bro_t = \bro^\star + t (\bro - \bro^\star)$ for some $t\in [0,1]$, as long as $\| \bu -\bu^\star \|_{\bL^\infty}+\| \theta -\theta^\star\|_{L^\infty}<\alpha_1$, where $\alpha_1 >0$ is as in Assumption \ref{growth2}. Thus, by letting $(\bu_t,\theta_t) = \mathcal{S}(\bro_t)$, Lemma \ref{secests} gives us -- whenever $\| \bu -\bu^\star \|_{\bL^\infty}+\| \theta -\theta^\star\|_{L^\infty}<\min\{ \alpha_1, \delta \}$ -- 
			\begin{align*}
				\Big\vert J_{\bzeta_0}'(\bro)(\bro-\bro^\star) - J'(\bro)(\bro-\bro^\star)\Big\vert\geq c\big(\|\bu -  \bu^\star\|_{L^2(Q)}^2+\|\theta - \theta^\star\|_{L^2(Q)}^2\big) - \varepsilon\big( \|\bu_t -  \bu^\star\|_{L^2(Q)}^2+\|\theta_t - \theta^\star\|_{L^2(Q)}^2 \big)
			\end{align*}
			Using the same arguments as in the proof of \Cref{theorem:sufficientoptimality}, we thus see that \begin{align*}
				&\Big\vert J_{\bzeta_0}'(\bro)(\bro-\bro^\star) - J'(\bro)(\bro-\bro^\star)\Big\vert\geq c_1\big(\|\bu -  \bu^\star\|_{\bL^2}^2+\|\theta - \theta^\star\|_{L^2}^2\big)\\
				& \ge \frac{c_1}{2}\big(\|\bu -  \bu^\star\|_{\bL^2} + \|\theta - \theta^\star\|_{L^2} \big)^2
			\end{align*}
			whenever $\| \bu -\bu^\star \|_{\bL^\infty}+\| \theta -\theta^\star\|_{L^\infty}< \alpha$ for some $\alpha>0$.
			It is left for us to estimate the left-hand side. Using \eqref{jderivative}, \eqref{jderivative:per} and H{\"o}lder inequality, we thus see that
			\begin{align*}
				&\Big\vert J'_{\bzeta_0}(\bro)(\bro-\bro^\star)-J'(\bro)(\bro-\bro^\star)\Big\vert \leq  \alpha_1(\|{\bv}\|_{\bL^2} \| \widehat{\bu} -\bu\|_{\bL^2} + \|\widehat{\bu} - \bu_d\|_{\bL^{2}}\| \widehat{\bv} - \bv\|_{\bL^{2}})\\
				& + \alpha_2(\| \vartheta \|_{L^2} \| \widehat{\theta} -\theta\|_{L^2} + \|\widehat{\theta} - \theta_d\|_{L^{2}}\|\widehat{\vartheta} - \vartheta \|_{L^{2}}) + \|\bta\|_{\bL^2} \| \widehat{\bv}\|_{\bL^2} + \|\eta \|_{L^2}\| \widehat{\vartheta}\|_{L^2} 
			\end{align*}
			where $(\bv,\vartheta) = \mathcal{S}'(\bro)(\bro-\bro^\star)$ and $(\widehat{\bv},\widehat{\vartheta}) = \widehat{\mathcal{S}}'(\bro)(\bro-\bro^\star)$. Let us now majorize the terms, first by utilizing Lemma \ref{estpersta} to get
			\begin{align*}
				&\begin{aligned}
					&\alpha_1\| \bv\|_{\bL^2} \| \widehat{\bu} - \bu\|_{\bL^2} + \alpha_2 \| \vartheta \|_{L^2} \|  \widehat{\theta} - \theta \|_{L^2} \\
					& \leq c \big(\| \bv \|_{\bL^2} + \| \vartheta \|_{L^2} \big)\big(\|\widehat{\bu}_0\|_{\bW^{2-2/s,s}_{0,\sigma}} + \|\widehat{\theta}_0\|_{W^{2-2/s,s}_{0}} + \|\widehat{\blf}\|_{\bL^s} + \|\widehat{h}\|_{L^s}\big),
				\end{aligned}
			\end{align*}
			On the other hand, we use Lemmata \ref{estpersta} and \ref{lemma:pertlinsest}
			\begin{align*}
				&\begin{aligned}
					&\alpha_1 \|\widehat{\bu} - \bu_d\|_{\bL^{2}}\| \widehat{\bv} - \bv\|_{\bL^{2}} + \alpha_2\| \widehat{\theta} - \theta_d\|_{L^{2}}\|  \widehat{\vartheta} - \vartheta \|_{L^{2}}\\  
					&\leq  c \big(\| \bv \|_{\bL^2} + \| \vartheta \|_{L^2} \big)\big(\|\widehat{\bu}_0\|_{\bW^{2-2/s,s}_{0,\sigma}} + \|\widehat{\theta}_0\|_{W^{2-2/s,s}_{0}} + \|\widehat{\blf}\|_{\bL^s} + \|\widehat{h}\|_{L^s}\big),
				\end{aligned}
			\end{align*}
			where the constant $c>0$ consists of the constants $M_{\mathcal{U}}$ and $M_{\mathcal{P}}$. Similarly, we get by virtue of Lemmata \ref{estpersta} and \ref{lemma:pertlinsest}
			\begin{align*}
				&\begin{aligned}
					& \|\bta\|_{\bL^2} \| \widehat{\bv} \|_{\bL^2} + \|\eta\|_{L^2} \| \widehat{\vartheta} \|_{L^2}  \leq \big( \|\bta\|_{\bL^2} + \|\eta\|_{L^2} \big) \big(\|\bv \|_{\bL^2} + \|\vartheta\|_{L^2} + \| \widehat{\bv} - \bv\|_{\bL^2} + \| \widehat{\vartheta} - \vartheta \|_{L^2}\big)\\
					&  \leq \big(1 + M_{\mathcal{P}}\big) \big( \|\bta\|_{\bL^2} + \|\eta\|_{L^2} \big)\big(\|\bv \|_{\bL^2} + \|\vartheta\|_{L^2} \big) ,
				\end{aligned}
			\end{align*}
			We finally use \eqref{estimate:linearcomp} and \eqref{estimate:diff-lin} to find an $\alpha>0$ for which $\| \bu -\bu^\star \|_{\bL^\infty}+\| \theta -\theta^\star\|_{L^\infty}< \alpha$ implies 
			\begin{align*}
				&\frac{c_1}{2}\big(\|\bu -  \bu^\star\|_{\bL^2} + \|\theta - \theta^\star\|_{L^2} \big)^2\\
				&\leq c \big(\|\bu -  \bu^\star\|_{\bL^2} + \|\theta - \theta^\star\|_{L^2} \big)\big(\|\widehat{\bu}_0\|_{\bW^{2-2/s,s}_{0,\sigma}} + \|\widehat{\theta}_0\|_{W^{2-2/s,s}_{0}} + \|\widehat{\blf}\|_{\bL^s} + \|\widehat{h}\|_{L^s} + \|\bta\|_{\bL^2} + \|\eta\|_{L^2} \big).
			\end{align*}
			If $\|\bu -  \bu^\star\|_{\bL^2} + \|\theta - \theta^\star\|_{L^2} = 0$ then \eqref{statestab} holds trivially. Otherwise, we can divide both sides by $\|\bu -  \bu^\star\|_{\bL^2} + \|\theta - \theta^\star\|_{L^2}$ which then gives us \eqref{statestab}.
		\end{proof}
		If we allow also control perturbations in the objective functional we still obtain Hölder stability.
		\begin{theorem}\label{thm:stabstateholder}
			Let $\bro^\star \in \mathcal{U}$ satisfy Assumption \ref{growth2} for $\tau=1$, and $(\bu^\star,\theta^\star) = \mathcal{S}(\bro^\star)$. Then there exists constants $c, \alpha>0$ such that
			\begin{equation}\label{statestab}
				\| \bu -\bu^\star \|_{\bL^2}+\| \theta -\theta^\star\|_{L^2} \leq c\| \bzeta \|_{\mathcal{P}}^{1/2}
			\end{equation}
			for all $\bro \in \mathcal{U}$ with $\bzeta \in \digamma (\bro)$ and $\| \bu -\bu^\star \|_{\bL^\infty}+\| \theta -\theta^\star\|_{L^\infty}<\alpha$, where $(\bu,\theta) = \mathcal{S}(\bro)$.
		\end{theorem}
		
		To prove the theorem we shall rely on the following lemmata which are analogous to, and whose proofs can be patterned with that of \Cref{lemma:supp2}.
		\begin{lemma}\label{lemma:supp2per}
			Let $\bro,\widehat{\bro} \in L^s(I;\bL^s(\omega_q)\times L^s(\omega_h))$, $(\widehat{\bu},\widehat{\theta}) = \widehat{\mathcal{S}}(\bro),(\widehat{\bu}^\star,\widehat{\theta}^\star) = \widehat{\mathcal{S}}(\widehat{\bro}) \in \bW^{2,1}_{s,\sigma}\times W^{2,1}_s$ and $(\widehat{\bv}^\star,\widehat{\vartheta}^\star) = \widehat{\mathcal{S}}'(\widehat{\bro})(\bro-\widehat{\bro})\in \bW^{2,1}_{s,\sigma}\times W^{2,1}_s$. There exists $\delta>0$ such that whenever $\|\widehat{\bu}^\star -  \widehat{\bu} \|_{\bL^\infty} +\|\widehat{\theta}^\star - \widehat{\theta}\|_{L^\infty} < \delta$ we have 
			\begin{align}\label{estimate:diff-lin}
				\|\widehat{\bu}^\star -  \widehat{\bu} \|_{\bL^2} +\|\widehat{\theta}^\star - \widehat{\theta}\|_{L^2} \le 2\left(\| \widehat\bv^\star \|_{\bL^2} + \|\widehat\vartheta^\star\|_{L^2}\right)\le 3\Big(\|\widehat{\bu}^\star -  \widehat{\bu} \|_{\bL^2} +\|\widehat{\theta}^\star - \widehat{\theta}\|_{L^2} \Big)
			\end{align}
		\end{lemma}

		\begin{proof}[Proof of {\Cref{thm:stabstateholder}}]
			The proof follows that of Theorem \ref{thm:stabstate}, using \eqref{estimate:linearcomp} and \eqref{estimate:diff-lin} to find an $\alpha>0$ for which $\| \bu -\bu^\star \|_{\bL^\infty}+\| \theta -\theta^\star\|_{L^\infty}< \alpha$ implies 
			\begin{align*}
				&\frac{c_1}{2}\big(\|\bu -  \bu^\star\|_{\bL^2} + \|\theta - \theta^\star\|_{L^2} \big)^2\leq c \|\bro-  \bro^\star\|_{\bL^1\times L^1}\| \bzeta \|_{\mathcal{P}}.
			\end{align*}
			Due to the box constraints, we have a uniform upper bound on $\|\bro-  \bro^\star\|_{\bL^1\times L^1}$, and taking the root yields the claim.
		\end{proof}
		
		As a direct consequence, we obtain the following stability result for the adjoint states.
		\begin{corollary}
			\label{corollary:statsolutions}
			Let $\bro^\ast\in \mathcal{U}$ satisfy  Assumption \ref{growth2} with $\tau = 1$. Then, there exist positive constants $c,\alpha$ such that
			\begin{align}
				\|\nabla( \bw -  \bw^\star)\|_{L^\infty(Q)}+\|\nabla(\Psi - \Psi^\star)\|_{L^\infty(Q)} \leq c\| \bzeta_0 \|_{\mathcal{P}}^{2/5}
			\end{align}
			for all local minimizers $\bro\in \mathcal{U}$ of \eqref{peropticon} with $\bzeta_0$ and $\| \bu -\bu^\star \|_{\bL^\infty}+\| \theta -\theta^\star\|_{L^\infty}<\alpha$, where $(\bu,\theta) = \mathcal{S}(\bro)$, $(\bu^{\star},\theta^{\star}) = \mathcal{S}(\bro^{\star})$, $(\bw,\Psi) = \mathcal{D}(\bro)$ and $(\bw^{\star},\Psi^{\star}) = \mathcal{D}(\bro^{\star})$.
			If control perturbations in the objective functional are admissible in $\bzeta$, there exist constants $c,\alpha>0$ such that
			\begin{align}
				\|\nabla( \bw -  \bw^\star)\|_{L^\infty(Q)}+\|\nabla(\Psi - \Psi^\star)\|_{L^\infty(Q)} \leq c\| \bzeta\|_{\mathcal{P}}^{2/10}
			\end{align}
			for all local minimizers $\bro\in \mathcal{U}$ of \eqref{peropticon} with $\bzeta$ and $\| \bu -\bu^\star \|_{\bL^\infty}+\| \theta -\theta^\star\|_{L^\infty}<\alpha$.
		\end{corollary}
		
		\begin{proof}
			From the embedding $\bW^{2,1}_{5,\sigma}\times W^{2,1}_5 \hookrightarrow C(\overline{I};C^1(\overline{\Omega})^2\times C^1(\overline{\Omega}))$ and by virtue of \Cref{theorem:adjointLP}
			\begin{align*}
				\begin{aligned}
					&  \|\nabla( \bw -  \bw^\star)\|_{L^\infty(Q)}+\|\nabla(\Psi - \Psi^\star)\|_{L^\infty(Q)}
					\le c \left(\|\bF\|_{\bL^5} + \|G\|_{L^5}\right),
				\end{aligned}
			\end{align*}
			where 
			\begin{align*}
				\begin{aligned}
					& \bF = ( (\bu - \bu^\star)\cdot\nabla )\bw + (\nabla(\bu^\star - \bu))^\top\bw + \Psi\nabla(\theta^\star - \theta) + \alpha_1(\bu - \bu^\star),\\
					& G = (\bu^\star - \bu)\cdot\nabla\Psi + \alpha_2(\theta - \theta^\star),\\
					& \bw_T = \beta_1(\bu(T) - \bu^\star(T)),\text{ and }\Psi_T =  \beta_2(\theta(T)-\theta^\star(T)).
				\end{aligned}
			\end{align*}
			Using similar arguments as in \Cref{lemma:adjstability} and \ref{lemma:adjperest} we see that
			\begin{align*}
				&\|\nabla( \bw -  \bw^\star)\|_{L^\infty(Q)}+\|\nabla(\Psi - \Psi^\star)\|_{L^\infty(Q)}
				\\
				&\le c\left( \| \bu -\bu^\star \|_{\bL^\infty}+\| \theta -\theta^\star\|_{L^\infty} \right)^{3/5}\left( \| \bu -\bu^\star \|_{\bL^2}+\| \theta -\theta^\star\|_{L^2} \right)^{2/5}.
			\end{align*}
			Finally, \Cref{thm:stabstate} gives us the desired estimate. The second claim is then a consequence of \Cref{thm:stabstateholder}.
		\end{proof}

		\subsubsection{Stability of the the second-order sufficient optimality condition}\label{sec7}
		In this section, we advance in discussing an open question regarding the stability of the second-order sufficient optimality condition commonly used in PDE-constrained optimal control. The investigation of second-order optimality conditions has been carried out extensively in the last decades, for an overview we refer to the survey paper \cite{CT2015}. We are interested in the following stability result:
		Assume that there exists a positive constant $c$ such that the optimal control $\bro^\star$ and optimal state $(\bu^\star,\theta^\star) = \mathcal{S}(\bro^\star)$ of \eqref{optcon} satisfies
		\begin{equation}\label{eq:secorderstab}
			J''(\bro^\star)(\bro-\bro^\star)^2 \geq  c\Big ( \|\bu -  \bu^\star \|_{\bL^2}^2 +\|\theta - \theta^\star\|_{L^2}^2\Big) \text{ for all } (\bro-\bro^\star)\in C_{\bro^\star}^\tau,
		\end{equation}
		where $(\bu,\theta) = \mathcal{S}(\bro)$ and $ C_{\bro^\star}^\tau$ is the extended cone of critical directions in $\bro^\star$.
		Now consider a sufficiently small perturbation $\bzeta$, a solution $\widehat{\bro}$ for the perturbed problem \eqref{peropticon}, and $(\widehat{\bu}^\star,\widehat{\theta}^\star) = \widehat{S}(\widehat{\bro})$ its optimal state. Can we infer from \eqref{eq:secorderstab} the existence of a positive constant $\hat c$ such that 
		\begin{equation}\label{eq:secorderstabperb}
			J''_{\bzeta}(\widehat{\bro})(\bro-\widehat{\bro})^2 \geq  \hat c\Big ( \|\widehat{\bu} -  \widehat{\bu}^{\star} \|_{\bL^2}^2 +\|\widehat{\theta} - \widehat{\theta}^{\star}\|_{L^2}^2\Big) \text{ for all } (\bro-\widehat{\bro})\in C_{\widehat{\bro}}^\tau
		\end{equation}
		where $(\widehat{\bu},\widehat{\theta}) = \widehat{S}(\bro)$, with $\bro^\star-\widehat{\bro}$ being sufficiently small? This question was discussed for instance in \cite[Theorem 4.6]{QW2018}, where the authors show that \eqref{eq:secorderstabperb} holds for the directions $\bro-\bro^\star$, that is for the cone $C_{\bro^\star}^\tau$. But this cone is not a critical cone for the perturbed problem \eqref{peropticon}. In this section, we show that by making a stronger assumption than \eqref{eq:secorderstab}, but which is still very reasonable for tracking type optimal control problems, we can derive the growth in \eqref{eq:secorderstabperb}. Namely, we ask for the following: there exists a positive constant $\tilde c$ such that
		\begin{equation}\label{equation:presecb1}
			\min\{\alpha_1,\alpha_2\}-2\left(\|\nabla \bw^\star\|_{\bL^\infty} + \|\nabla \Psi^\star\|_{L^\infty} \right) \ge \widetilde{c}.
		\end{equation}
		It is clear that \eqref{eq:secorderstab} is implied by \eqref{equation:presecb1} and that \eqref{equation:presecb1} holds if \eqref{trackingclose} is satisfied, which on the other hand is very reasonable for tracking-type problems.
		Then we can infer the stability of the second-order condition \eqref{eq:secorderstabperb} as a consequence of the solution stability obtained in Theorem \ref{thm:stabstate} for sufficiently small perturbations.

		\begin{theorem}\label{stabiltyofthesec}
			Let $\bro^\star\in\mathcal{U}$ satisfy \eqref{trackingclose} 
			and let $\bzeta$ be sufficiently small. Suppose that $\widehat{\bro}\in\mathcal{U}$ and $\alpha>0$ are such that $\bzeta\in\digamma(\widehat{\bro})$ and $\|\overline{\bu} -  \bu^\star \|_{\bL^\infty} +\|\overline{\theta} - \theta^\star\|_{L^\infty} <\alpha$, where $(\bu^\star,\theta^\star) = \mathcal{S}(\bro^\star)$, $(\overline{\bu},\overline{\theta}) = {\mathcal{S}}(\widehat{\bro})$ and $\alpha>0$ is as in \Cref{thm:stabstate}. Then there exist positive constants $c$ and $\delta$ such that 
			\begin{equation}\label{equation:growthper}
				J_{\bzeta}''(\widehat{\bro})(\bro - \widehat{\bro})^2\geq  c\Big ( \|\widehat{\bu} -  \widehat{\bu}^{\star} \|_{\bL^2}^2 +\|\widehat{\theta} - \widehat{\theta}^\star\|_{L^2}^2  \Big)
			\end{equation}
			for all $\bro\in \mathcal{U}$ with $\|\widehat{\bu} -  \widehat{\bu}^\star\|_{\bL^\infty} +\|\widehat{\theta} - \widehat{\theta}^\star\|_{L^\infty} <\delta$, where $(\widehat{\bu},\widehat{\theta}) = \widehat{\mathcal{S}}(\bro)$ and $(\widehat{\bu}^\star,\widehat{\theta}^\star) = \widehat{\mathcal{S}}(\widehat{\bro})$.
			
			
		\end{theorem}

		\begin{proof}
			First, if condition \eqref{trackingclose} is fulfilled, there exist a positive constant $\widetilde{c}$, such that the control $\bro^\star$ satisfies
			\begin{equation}\label{equation:secb1}
				\min\{\alpha_1,\alpha_2\}-2\left(\|\nabla \bw^\star\|_{\bL^\infty} + \|\nabla \Psi^\star\|_{L^\infty} \right) \ge \widetilde{c},
			\end{equation}
			where $(\bw^\star,\Psi^\star) = \mathcal{D}(\bro^\star)$. Furthermore, we infer from the proof of \Cref{suffgro2} that \eqref{equation:secb1} implies
			\begin{equation}\label{equation:grothsecglo}
				J''(\bro^\star)(\bro-\bro^\star)^2\geq  c\Big ( \|\bu-  \bu^\star \|_{\bL^2}^2 +\|\theta - \theta^\star\|_{L^2}^2  \Big),
			\end{equation}
			whenever $\|\bu-  \bu^\star \|_{\bL^\infty} +\|\theta - \theta^\star\|_{L^\infty} < \alpha_1$ for some $\alpha_1 >0$, where $(\bu,\theta) = \mathcal{S}(\bro)$. This also implies that $\bro^\star$ is a strong local minimizer of \eqref{optcon}.
			Now to prove \eqref{equation:growthper}, let us first denote as $(\widehat{\bv}^\star,\widehat{\vartheta}^\star) = \widehat{\mathcal{S}}'(\widehat{\bro})(\bro - \widehat{\bro})$ and $(\widehat{\bw}^\star,\widehat{\Psi}^\star)\in \widehat{\mathcal{D}}(\widehat{\bro})$, we thus estimate the second variation of the perturbed problem below as follows:
			\begin{equation*}
				\begin{aligned}
					J_{\bzeta}''(\widehat{\bro})(\bro-\widehat{\bro})^2 &= \alpha_1 \|\widehat{\bv}^{\star} \|_{\bL^2}^2 + \alpha_2\|\widehat{\vartheta}^{\star}  \|_{L^2}^2  - 2( (\widehat{\bv}^{\star}\cdot\nabla)\widehat{\bv}^{\star}, \widehat{\bw}^{\star} )_Q - 2( \widehat{\bv}^{\star}\cdot\nabla \widehat{\vartheta}^{\star}, \widehat{\Psi}^{\star})_Q\\
					&= \alpha_1 \|\widehat{\bv}^{\star} \|_{\bL^2}^2 + \alpha_2\|\widehat{\vartheta}^{\star}  \|_{L^2}^2  - 2( (\widehat{\bv}^{\star}\cdot\nabla)\widehat{\bv}^{\star}, {\bw}^{\star} )_Q - 2( \widehat{\bv}^{\star}\cdot\nabla \widehat{\vartheta}^{\star}, {\Psi}^{\star})_Q\\
					&\quad - 2( (\widehat{\bv}^{\star}\cdot\nabla)\widehat{\bv}^{\star}, \widehat{\bw}^{\star} - {\bw}^{\star} )_Q - 2( \widehat{\bv}^{\star}\cdot\nabla \widehat{\vartheta}^{\star}, \widehat{\Psi}^{\star} - {\Psi}^{\star})_Q\\
					&\ge \big(\min\{\alpha_1,\alpha_2\}-2\left(\|\nabla \bw^\star\|_{\bL^\infty} + \|\nabla \Psi^\star\|_{L^\infty} \right) \big)\big(\|\widehat{\bv}^{\star} \|_{\bL^2}^2 + \|\widehat{\vartheta}^{\star}\|_{L^2}^2 \big)\\
					&\quad - 2\left(\|\nabla(\widehat{\bw}^\star - \bw^\star)\|_{\bL^\infty} + \|\nabla(\widehat{\Psi}^\star  - \Psi^\star)\|_{L^\infty} \right)\big(\|\widehat{\bv}^{\star} \|_{\bL^2}^2 + \|\widehat{\vartheta}^{\star}\|_{L^2}^2 \big)
				\end{aligned}
			\end{equation*}
			Denoting by $(\overline{\bw},\overline{\Psi}) = \mathcal{D}(\widehat{\bro})$, from \Cref{lemma:adjperest} and \Cref{corollary:statsolutions}  we get
			\begin{align*}
				&\|\nabla(\widehat{\bw}^\star - \bw^\star)\|_{\bL^\infty} + \|\nabla(\widehat{\Psi}^\star  - \Psi^\star)\|_{L^\infty}\\
				&\le \|\nabla(\widehat{\bw}^\star - \overline{\bw})\|_{\bL^\infty} + \|\nabla(\widehat{\Psi}^\star  - \overline{\Psi})\|_{L^\infty} + \|\nabla(\overline{\bw} - \bw^\star)\|_{\bL^\infty} + \|\nabla(\overline{\Psi}  - \Psi^\star)\|_{L^\infty}\\
				& \le c\big( \| \bzeta \|_{\mathcal{P}} + \| \bzeta \|_{\mathcal{P}}^{1/5}\big).
			\end{align*}
			Therefore, by utilizing \Cref{lemma:supp2per} and by choosing $\bzeta$ small enough such that $\| \bzeta \|_{\mathcal{P}} + \| \bzeta \|_{\mathcal{P}}^{1/5} < \widetilde{c}$, we get the desired estimate. 
		\end{proof}

		It is clear that the constant $c>0$ in \eqref{equation:growthper} holds for all $\bzeta$ small enough such that $\| \bzeta \|_{\mathcal{P}} + \| \bzeta \|_{\mathcal{P}}^{1/5} < \widetilde{c}$. As a direct consequence of Theorem \ref{stabiltyofthesec} we have the following result, which highlights the optimality of $\widehat{\bro}\in\mathcal{U}$ and uniform growth.
		\begin{corollary}
			Suppose that the assumptions in \Cref{stabiltyofthesec} hold. Then there exist constants $\delta>0$ and $\widehat{c}>0$, which is independent on $\bzeta$, such that
			\begin{equation}
				J_{\bzeta}(\bro)-J_{\bzeta}(\widehat{\bro})\geq \widehat{c}  \Big(\|\widehat{\bu} -  \widehat{\bu}^\star \|_{\bL^2}^2 +\|\widehat{\theta} - \widehat{\theta}^\star\|_{L^2}^2  \Big)
			\end{equation}   and all $\bro\in \mathcal{U}$ such that $\|\widehat{\bu} -  \widehat{\bu}^\star \|_{\bL^\infty} +\|\widehat{\theta} - \widehat{\theta}^\star\|_{L^\infty} <\delta$, where $(\widehat{\bu}^\star,\widehat{\theta}^\star) = \widehat{S}(\widehat{\bro})$ and $(\widehat{\bu},\widehat{\theta}) = \widehat{S}({\bro})$. This shows that $\widehat{\bro}\in\mathcal{U}$ is a strong local minimizer of \eqref{peropticon}.
		\end{corollary}
		\begin{proof}
			We apply Taylor's theorem and Theorem \ref{stabiltyofthesec} to infer
			\begin{equation*}
				\begin{aligned}
					&J_{\bzeta}(\bro)-J_{\bzeta}(\widehat{\bro}) =J_{\bzeta}'(\widehat{\bro})(\bro-\widehat{\bro})+\frac{1}{2}J''_{\bzeta}(\widehat{\bro}_t)(\bro-\bro^{\star,\zeta})^2\\
					&\geq J_{\bzeta}'(\widehat{\bro})(\bro-\widehat{\bro})+\frac{1}{2}J''_{\bzeta}(\widehat{\bro})(\bro-\widehat{\bro})^2 -\frac{1}{2}\Big\vert J''_{\bzeta}(\widehat{\bro}_t)(\bro-\widehat{\bro})^2-J''_{\bzeta}(\widehat{\bro})(\bro-\widehat{\bro})^2\Big\vert\\
					&\geq  c\Big ( \|\widehat{\bu} -  \widehat{\bu}^{\star} \|_{\bL^2}^2 +\|\widehat{\theta} -\widehat{\theta}^{\star}\|_{L^2}^2  \Big) -\frac{1}{2}\Big\vert J''_{\bzeta}(\widehat{\bro}_t)(\bro-\widehat{\bro})^2-J''_{\bzeta}(\widehat{\bro})(\bro-\widehat{\bro})^2\Big\vert.
				\end{aligned}
			\end{equation*}
			where $\widehat{\bro}_t = \widehat{\bro} + t(\bro-\widehat{\bro})$ for some $t\in (0,1)$ and $c>0$ is as in Theorem \ref{stabiltyofthesec}. Using arguments as in the proof of \Cref{secestc} and \Cref{secests}, the last term can be estimated by
			\begin{equation}
				\begin{aligned}
					\frac{1}{2}\Big\vert J''_{\bzeta}(\widehat{\bro}_t)(\bro-\widehat{\bro})^2-J''_{\bzeta}(\widehat{\bro})(\bro-\widehat{\bro})^2\Big\vert \leq \varepsilon  (\|\widehat{\bu} -  \widehat{\bu}^{\star} \|_{\bL^2}^2 +\|\widehat{\theta} -\widehat{\theta}^{\star}\|_{L^2}^2),
				\end{aligned}
			\end{equation}
			for all $\bro \in \mathcal{U}$ with $\|\widehat{\bu} -  \widehat{\bu}^{\star} \|_{\bL^\infty} +\|\widehat{\theta} -\widehat{\theta}^{\star}\|_{L^\infty} <\alpha$,
			where $\varepsilon>0$ can be chosen arbitrarily small and independent of $\bzeta$.
			Thus we can select a uniform growth constant $\widehat {c}:=c -\varepsilon>0$.
		\end{proof}

\begin{remark}
The results in this section rely on \eqref{trackingclose} which already implies global growth of the second variation, therefore the restriction of the second-order condition to the extended critical cone is not needed.
\end{remark}

\section{Conclusion}\label{conclusion}

In this paper, we studied a tracking-type optimal control problem subject to the Boussinesq system and the controls appear distributively both on the continuity and heat equations which satisfy a box constraint. Since the objective functional contains no regularization for the controls, the solution is expected to be of bang-bang type. We provided first order necessary and second order sufficient conditions for the optimal control problem. The main objective in this paper was to establish the stability of the problem with respect to several perturbations: i.) linear perturbations on both the objective functional and the state equations; ii.) nonlinear perturbations in the form of the Tikhonov regularization and perturbation on the desired data.

Aside from the said stability results, we also provided the $L^p$ regularity of the solutions of the Boussinesq system, its linearization and a corresponding adjoint system, which -- as far as we are aware -- is a novelty. Another novelty that we would want to highlight is the stability of the growth of the second-order derivative of the objective functional with respect to sufficiently small perturbations which does not depend on the extended critical cone of critical directions for the optimal bang-bang control.


\begin{thebibliography}{10}
	
	\bibitem{amann2001}
	Herbert Amann.
	\newblock Linear parabolic problems involving measures.
	\newblock {\em Real Academia de Ciencias Exactas, Fisicas y Naturales. Revista.
		Serie A, Matematicas}, 95(1):85--119, 2001.
	
	\bibitem{amorim2024}
	Charles~Braga Amorim, Marcelo~Fernandes {de Almeida}, and {\`E}der Mateus.
	\newblock Global existence of solutions for {B}oussinesq system with energy
	dissipation.
	\newblock {\em Journal of Mathematical Analysis and Applications}, 531(2, Part
	1):127905, 2024.
	
	\bibitem{baranovskii2021}
	Evgenii~S. Baranovskii.
	\newblock Optimal boundary control of the {B}oussinesq approximation for
	polymeric fluids.
	\newblock {\em Journal of Optimization Theory and Applications},
	189(2):623--645, 2021.
	
	\bibitem{BBS}
	T{\'e}rence Bayen, J.~Fr{\'e}d{\'e}ric Bonnans, and Francisco~J. Silva.
	\newblock Characterization of local quadratic growth for strong minima in the
	optimal control of semi-linear elliptic equations.
	\newblock {\em Trans. Am. Math. Soc.}, 366(4):2063--2087, 2014.
	
	\bibitem{BFA}
	Jose~Luis Boldrini, Enrique Fern{\'a}ndez-Cara, and Marko~Antonio Rojas-Medar.
	\newblock An optimal control problem for a generalized {Boussinesq} model: the
	time dependent case.
	\newblock {\em Rev. Mat. Complut.}, 20(2):339--366, 2007.
	
	\bibitem{casas2012}
	Eduardo Casas.
	\newblock Second order analysis for bang-bang control problems of {PDE}s.
	\newblock {\em SIAM Journal on Control and Optimization}, 50(4):2355--2372,
	2012.
	
	\bibitem{casas2017}
	Eduardo Casas and Konstantinos Chrysafinos.
	\newblock Error estimates for the approximation of the velocity tracking
	problem with bang-bang controls.
	\newblock {\em ESAIM Control Optim. Calc. Var.}, 23(4):1267--1291, 2017.
	
	\bibitem{casas2022}
	Eduardo Casas, Alberto Dom{\'{\i}}nguez~Corella, and Nicolai Jork.
	\newblock New assumptions for stability analysis in elliptic optimal control
	problems.
	\newblock {\em SIAM J. Control Optim.}, 61(3):1394--1414, 2023.
	
	\bibitem{CT2015}
	Eduardo Casas and Fredi Tr\"{o}ltzsch.
	\newblock Second order optimality conditions and their role in {PDE} control.
	\newblock {\em Jahresber. Dtsch. Math.-Ver.}, 117(1):3--44, 2015.
	
	\bibitem{chierici2022}
	Andrea Chierici, Valentina Giovacchini, and Sandro Manservisi.
	\newblock Analysis and computations of optimal control problems for
	{B}oussinesq equations.
	\newblock {\em Fluids}, 7(6), 2022.
	
	\bibitem{jork2023}
	Alberto~Dom{\'\i}nguez Corella, Nicolai Jork, and Vladimir~M. Veliov.
	\newblock On the solution stability of parabolic optimal control problems.
	\newblock {\em Computational Optimization and Applications}, 86(3):1035--1079,
	2023.
	
	\bibitem{DJNS2023}
	Alberto~Domínguez Corella, Nicolai Jork, Šarká Nečasová, and John
	Sebastian~H. Simon.
	\newblock Stability analysis of the {N}avier-{S}tokes velocity tracking problem
	with bang-bang controls, 2023.
	
	\bibitem{DH2012}
	Klaus Deckelnick and Michael Hinze.
	\newblock A note on the approximation of elliptic control problems with
	bang-bang controls.
	\newblock {\em Comput. Optim. Appl.}, 51(2):931--939, 2012.
	
	\bibitem{ADW2023}
	Alberto Dom\'{i}nguez~Corella and Gerd Wachmsuth.
	\newblock Stability and genericity of bang-bang controls in affine problems.
	\newblock {\em preprint}, 2023.
	
	\bibitem{dominguez2022}
	{Domínguez Corella, Alberto}, {Jork, Nicolai}, and {Veliov, Vladimir}.
	\newblock Stability in affine optimal control problems constrained by
	semilinear elliptic partial differential equations.
	\newblock {\em ESAIM: COCV}, 28:79, 2022.
	
	\bibitem{gerhardt1978}
	Claus Gerhardt.
	\newblock L$^p$-estimates for solutions to the instationary {N}avier--{S}tokes
	equations in dimension two.
	\newblock {\em Pacific Journal of Mathematics}, 79(2):375--398, 1978.
	
	\bibitem{gong2023}
	Menghan Gong and Zhuan Ye.
	\newblock Regularity criterion for the 2{D} inviscid {B}oussinesq equations.
	\newblock {\em Journal of Mathematical Fluid Mechanics}, 25(4):88, 2023.
	
	\bibitem{greatbatch2001}
	Richard~J. Greatbatch, Youyu Lu, and Yi~Cai.
	\newblock Relaxing the {B}oussinesq approximation in ocean circulation models.
	\newblock {\em Journal of Atmospheric and Oceanic Technology}, 18(11):1911 --
	1923, 2001.
	
	\bibitem{hsia2015}
	Chun-Hsiung Hsia, Chang-Shou Lin, Tian Ma, and Shouhong Wang.
	\newblock Tropical atmospheric circulations with humidity effects.
	\newblock {\em Proceedings of the Royal Society A: Mathematical, Physical and
		Engineering Sciences}, 471(2173):20140353, 2015.
	
	\bibitem{J2023}
	Nicolai Jork.
	\newblock Finite element error analysis of affine optimal control problems.
	\newblock {\em preprint}, 2023.
	
	\bibitem{ladyzhenskaya1988}
	Olga~Aleksandrovna Ladyzhenskaya and Vsevolod~A. Solonnikov.
	\newblock {\em Linear and Quasi-linear Equations of Parabolic Type}.
	\newblock Translations of mathematical monographs. American Mathematical
	Society, 1988.
	
	\bibitem{lee2002}
	Hyung-Chun Lee and Byeong~Chun Shin.
	\newblock Dynamics for controlled 2-{D} {B}oussinesq systems with distributed
	controls.
	\newblock {\em Journal of Mathematical Analysis and Applications},
	273(2):457--479, 2002.
	
	\bibitem{li2005}
	Shugang Li.
	\newblock Optimal controls of {B}oussinesq equations with state constraints.
	\newblock {\em Nonlinear Analysis: Theory, Methods \& Applications},
	60(8):1485--1508, 2005.
	
	\bibitem{li2003}
	Shugang Li and Gengsheng Wang.
	\newblock The time optimal control of the {B}oussinesq equations.
	\newblock {\em Numerical Functional Analysis and Optimization},
	24(1-2):163--180, 2003.
	
	\bibitem{lin2024}
	Quyuan Lin, Rongchang Liu, and Weinan Wang.
	\newblock Global existence for the stochastic {B}oussinesq equations with
	transport noise and small rough data.
	\newblock {\em SIAM Journal on Mathematical Analysis}, 56(1):501--528, 2024.
	
	\bibitem{OV}
	Nikolai~P. Osmolovskii and Vladimir~M. Veliov.
	\newblock On the regularity of {Mayer}-type affine optimal control problems.
	\newblock In {\em Large-scale scientific computing. 12th international
		conference, LSSC 2019, Sozopol, Bulgaria, June 10--14, 2019. Revised selected
		papers}, pages 56--63. Cham: Springer, 2020.
	
	\bibitem{pata2019}
	Vittorino Pata.
	\newblock {\em The Implicit Function Theorem}, pages 75--80.
	\newblock Springer International Publishing, Cham, 2019.
	
	\bibitem{QW2018}
	Nguyen~Thanh Qui and Daniel Wachsmuth.
	\newblock Stability for bang-bang control problems of partial differential
	equations.
	\newblock {\em Optimization}, 67(12):2157--2177, 2018.
	
	\bibitem{solonnikov2001}
	Vsevolod~A. Solonnikov.
	\newblock Lp-estimates for solutions to the initial boundary-value problem for
	the generalized {S}tokes system in a bounded domain.
	\newblock {\em Journal of Mathematical Sciences}, 105(5):2448--2484, 2001.
	
	\bibitem{sun2024}
	Weixian Sun and Wenjuan Wang.
	\newblock The uniqueness of global solution for the 2d {B}oussinesq equations
	with partial dissipation.
	\newblock {\em Journal of Mathematical Analysis and Applications},
	530(2):127735, 2024.
	
	\bibitem{dwachsmuth2006}
	Daniel Wachsmuth.
	\newblock {\em Optimal controlof the unsteady {N}avier--{S}tokes equations}.
	\newblock Phd thesis, Technische Universit{\"a}t Berlin, Berlin, September
	2006.
	\newblock Available at \url{https://d-nb.info/982143419/34}.
	
\end{thebibliography}
\end{document}